\documentclass[mathpazo]{cicp}
\usepackage[margin=1in]{geometry}
\usepackage{hyperref}
\usepackage{amsmath,epsf,cite}
\usepackage{amssymb,amsthm,tocvsec2}
\usepackage{graphicx}
\usepackage{epsfig,epsf,latexsym,subfigure}
\usepackage{float}
\usepackage{color}
\usepackage{mathrsfs}
\usepackage{indentfirst}
\usepackage[colorinlistoftodos]{todonotes}
\date{ }

\numberwithin{equation}{section}
\numberwithin{figure}{section}
\numberwithin{table}{section}

\theoremstyle{plain}
\newtheorem{thm}{Theorem}[section]

\theoremstyle{remark}
\newtheorem{rem}{Remark}[section]

\def\R{\mathbb R}
\def\hR{\mathbb R}
\def\bR{\mathbf{R}}

\def\bW{\mathbf{W}}

\def\be{\mathbf{e}}

\def\bc{\mathbf{c}}
\def\bx{\mathbf{x}}
\def\ba{\mathbf{a}}

\def\bI{\mathbf{I}}

\def\bQ{\mathbf{Q}}

\def\bG{\mathbf{G}}

\def\bw{\mathbf{w}}
\def\bV{\mathbf{V}}

\def\cL{\mathcal{L}}

\def\cG{\mathcal{G}}
\def\cB{\mathcal{B}}

\def\cN{\mathcal{N}}
\def\cE{\mathcal{E}}

\def\cD{\mathcal{D}}

\def\bU{\mathbf{U}}
\def\bV{\mathbf{V}}

\newcommand{\ben}{\begin{eqnarray}}
\newcommand{\een}{\end{eqnarray}}
\newcommand{\beq}{\begin{equation}}
\newcommand{\eeq}{\end{equation}}
\newcommand{\bea}{\begin{array}}
\newcommand{\eea}{\end{array}}
\newcommand{\bef}{\begin{figure}[H]}
\newcommand{\eef}{\end{figure}}

\newtheorem{scheme}{Scheme}[section]

\begin{document}
\title{General Numerical Framework to Derive Structure Preserving Reduced Order Models for Thermodynamically Consistent Reversible-Irreversible PDEs}
\author[J. Zhao]{
Zengyan Zhang and Jia Zhao\affil{1}\comma \corrauth}
\address{\affilnum{1}\ Department of Mathematics and Statistics, Binghamton University, Binghamton, NY, USA }
\email{ {\tt jiazhao@binghamton.edu.} (J.~Zhao)}

\begin{abstract}
In this paper, we propose a general numerical framework to derive structure-preserving reduced order models for thermodynamically consistent PDEs. Our numerical framework has two primary features: (a) a systematic way to extract reduced order models for thermodynamically consistent PDE systems while maintaining their inherent thermodynamic principles and (b) a strategic process to devise accurate, efficient, and structure-preserving numerical algorithms to solve the forehead reduced-order models. The platform's generality extends to various PDE systems governed by embedded thermodynamic laws. The proposed numerical platform is unique from several perspectives. First, it utilizes the generalized Onsager principle to transform the thermodynamically consistent PDE system into an equivalent one, where the transformed system's free energy adopts a quadratic form of the state variables. This transformation is named energy quadratization (EQ). Through EQ, we gain a novel perspective on deriving reduced order models. The reduced order models derived through our method continue to uphold the energy dissipation law. Secondly, our proposed numerical approach automatically provides numerical algorithms to discretize the reduced order models. The proposed algorithms are always linear, easy to implement and solve, and uniquely solvable. Furthermore, these algorithms inherently ensure the thermodynamic laws. In essence, our platform offers a distinctive approach to derive structure-preserving reduced-order models for a wide range of PDE systems abiding by thermodynamic principles.
\end{abstract}

\maketitle

\section{Introduction}

Partial differential equation (PDE) models play an indispensable role in modeling physical phenomena in various scientific and engineering fields. They have been used to model complex systems, multidimensional phenomena, and various interdisciplinary problems. As a coarse-grained modeling approach, ensuring these PDE models remain consistent with underlying thermodynamic laws is essential, particularly when reversible and irreversible transitions occur, making the models both physically plausible and providing reliable predictions. PDEs that respect the thermodynamic laws are known as thermodynamically consistent PDEs. In particular, when the temperature fluctuation can be ignored, the thermodynamically consistent PDE system will respect laws of free energy dissipation, i.e., the free energy for the system is non-increasing in time. These models are sometimes called gradient flow models. There are many broadly used PDE models that fall into the thermodynamically consistent category, to name a few: the Allen-Cahn equation \cite{AC}, the Cahn-Hilliard equation \cite{Cahn&H1958}, the phase field crystal equation \cite{Elder2002,Gomez2012An}, the Navier-Stokes equation, and the Ericksen-Leslie model for liquid crystals \cite{beri94,OttingerBook}.

The thermodynamically consistent PDEs are usually coupled systems with nonlinearity. The analytical approach usually falls short, making the numerical approach necessary. Unfortunately, they are generally hard to solve due to the stiffness in nonlinearity and time-consuming to simulate numerically when long-time dynamics are desired. A numerical deviation from these laws can result in physically implausible solutions, defeating the purpose of the original modeling intention \cite{DVDM-book}. There is a decent size of numerical analysis community dedicated to deriving numerical algorithms that can preserve the thermodynamic structure, namely structure-preserving numerical algorithms. In particular, numerical algorithms that respect the free energy dissipation laws are usually named energy-stable algorithms \cite{Eyre1998}. If such energy stability is independent of the choices of the time step size, they are known as unconditionally energy stable \cite{Gomez2012An}. In this vein, there are many versatile research results achieved recently, to name a few: the stabilized approach \cite{ShenJ2}, the convex splitting method \cite{Eyre1998,WiseSINUMA2009,WangChengCMS2015,Wang&Wang&WiseDCDS2010}, the energy quadratization method \cite{Yang&Zhao&WangJCP2017,ZhaoIEQrelax}, the scalar auxiliary variable approach \cite{SAV-1,SAV-2,Jiang-SAV-relax}, the SVM method \cite{SVM-Gong} and Lagrangian multiplier approach \cite{LM-Cheng}.

Alternatively, a naive way to tackle the computational complexity of the thermodynamically consistent models is to embrace the reduced-order model (ROM) or model order reduction (MOR) that can reduce the spatial-temporal complexity. ROMs have gained significant attention in recent years for their capability to simplify complex PDE systems without compromising the integrity of the solution, enabling faster computational speed and more efficient simulations. The way for ROM to save computational cost is by projecting the problem from a high-dimensional system into a much lower-dimensional subspace while maintaining a small approximation error. In particular, the ROMs are usually required to conserve the properties and characteristics of the full-order model, i.e., the original model. The broadly used MOR techniques include proper orthogonal decomposition methods \cite{POD-1,POD-2}, reduced basis methods,  balancing methods, and nonlinear manifold methods, as well as many projection-based reductions.
Among these approaches, the proper orthogonal decomposition (POD) Galerkin method will be the major focus of this paper. The POD method computes an optimal subspace to fit the empirical data. The POD was first introduced to study turbulence by the fluid-dynamics community \cite{POD-1} as a way to decompose the random vector field representing turbulent fluid motion into a set of deterministic functions that each capture some portion of the total fluctuating kinetic energy in the flow \cite{POD-2}. The POD method has been widely used in computational fluid dynamics and structure analysis ever since.
In general, the empirical data is generated, and POD modes are computed during the offline stage. In the online stage, the POD-ROM is solved in real-time instead of the full model. Some special techniques are needed to handle the nonlinearity to ensure the ROM is fully independent of the full model's dimension. These techniques include the trajectory piecewise linear approximation, missing point estimation, gappy POD \cite{Gappy-POD}, empirical interpolation, and the discrete empirical interpolation method (DEIM) \cite{DEIM}. 

When it comes to developing POD-ROMs for thermodynamically consistent PDE systems, the challenge lies in not only maintaining the reduced computational complexity and controlled error to the original PDEs but also respecting their embedded thermodynamic laws. Unfortunately, a direct application of the classical POD-ROM strategies to the thermodynamically consistent PDEs will usually destroy the thermodynamic structure of the full-order model, which is problematic since these numerical solutions from the POD-ROM will violate the thermodynamic laws. Particular attention on adjusting the POD-ROM strategies for thermodynamically consistent PDEs is needed. Over the years, some seminal ideas have been proposed to design structure-preserving ROMs for the Hamiltonian (reversible) systems, but there is still little work on dissipative (irreversible) systems. This motivates our research in this paper.  
Here, we provide a brief summary of the existing work in the literature.
A general structure preserving reduced order modeling approach for gradient systems is proposed in \cite{ROM-1}. The authors use the symmetric interior penalty discontinuous Galerkin (SIPG) method for spatial discretization and the average vector field (AVF) method for temporal discretization. The nonlinear terms are taken care of by the discrete empirical interpolation method (DEIM). The major drawback of this approach is that the resulting system is fully implicit and nonlinear. A similar idea is applied to design structure-preserving integration and model order reduction of the skew-gradient reaction-diffusion systems \cite{ROM-skew}. A structure-preserving Galerkin POD reduced-order modeling for the Hamiltonian systems is introduced in \cite{WangZhu-POD-1}. The major idea is to introduce a modified skew-symmetric operator in the reduced order model such that the Hamiltonian is preserved. A further extension introduced in \cite{WangZhu-POD-2} overcomes the high computational complexity with non-polynomial nonlinearities in the Hamiltonian by the discrete empirical interpolation method. There is also some work in designing reduced-order models for the dissipative systems, in particular, the phase field models or gradient flow models. For instance, the authors introduce a reduced order model for the Allen-Cahn equation by embracing the idea of the scalar auxiliary variable (SAV) method in \cite{ROM-AC-1}. An alternative approach to develop the reduced order model for the Allen-Cahn equation is introduced in \cite{ROM-AC-2}, where the authors use the stabilized semi-implicit scheme. A rigorous numerical analysis is also provided. A finite difference approach in space and IMEX Runge-Kutta approach in time is introduced for developing reduced order models of the Allen-Cahn equation in \cite{ROM-AC-3}. Similar techniques of \cite{WangZhu-POD-1} are exploited to develop structure-preserving reduced-order modeling for the Korteweg-de Vries equation in \cite{ROM-KdV}. A method is introduced for the dissipative Hamiltonian systems, but it only applies to dissipative Hamiltonian systems with a quadratic Hamiltonian \cite{ROM-dissipative-1}. Meanwhile, there are several other research directions. The authors consider the inclusion of spatial adaptivity for the snapshot computation in the offline phase of model order reduction \cite{ROM-AdaptivePOD}. Using deep learning techniques, particularly deep neural networks, for model order reductions has also been considered \cite{ROM-NN-1}.

How to systematically derive such structure-preserving POD-ROM is the major focus of this paper. The pursuit of structure-preserving ROMs ensures that while the intricacies of the original system might be reduced for computational efficiency, the core governing principles remain untouched. This balance between efficiency and authenticity is crucial for the reliability and trustworthiness of numerical simulations in various scientific and engineering applications. We embrace the energy quadratization (EQ) idea we introduced to design numerical algorithms for thermodynamically consistent PDEs \cite{Yang&Zhao&WangJCP2017}. The EQ method has offered a transformative perspective in the realm of structure-preserving numerical approximations. By introducing auxiliary variables, the original PDE system is transformed so that the system's free energy adopts a quadratic form. It provides a fresh lens through which these problems can be approached.
As we can see later in this paper, this energy quadratization (EQ) process is a cornerstone in developing structure-preserving ROMs for thermodynamically consistent PDEs.

In this paper, we introduce a numerical platform that can systematically derive ROMs for reversible-irreversible thermodynamically consistent PDE models by embracing several novel techniques. First of all, we utilize the EQ technique \cite{Zhao2018EQreview, Yang&Zhao&WangJCP2017} to reformulate the generic thermodynamically consistent PDE models into the quadratic Onsager form. By this equivalent model reformulation, the reversible-irreversible thermodynamic structures of the original PDEs are fully disclosed. Then, thanks to the quadratic structure in our reformulated system, we are able to slightly modify the classical POD-ROM, which was inspired by \cite{EggerNM2019}, to derive structure-preserving POD-ROM. There are several unique advantages of our POD-ROM framework compared with existing results in the literature: (1) First of all, our approach is rather general in that the POD-ROM framework applies to most existing thermodynamically consistent models and respects their thermodynamic structures after model order reduction; (2) Secondly, the structure-preserving numerical integration of our POD-ROM framework is linear, making it easy-to-implement and cheap-to-compute; (3) Thirdly, given the linearized nature of the POD-ROM framework, all the nonlinear terms in the ROM can be treated explicitly (using many existing techniques such as DEIM) while still preserving the thermodynamic structure for the fully-discrete numerical solutions. These properties make our POD-ROM framework compatible and widely applicable.

The rest of this paper is structured as follows. Section 2 provides a comprehensive view of reformulating the thermodynamically consistent PDEs model via the generalized Onsager principle, particularly with the Onsager triplet. This allows us to recast the system using the energy quadratization (EQ) method, resulting in a system with free energy expressed in quadratic forms in the state variables. In Section 3, we elaborate on the procedure of designing structure-preserving reduced order models in the context of the EQ reformulation. Afterward, we introduce the numerical platform to develop structure-preserving numerical algorithms for the reduced order models in Section 4. In Section 5, we showcase a range of numerical examples for specific thermodynamically consistent PDE models, demonstrating the effectiveness of the proposed computational strategy. We then wrap this paper with brief concluding remarks in the last section.

\section{Thermodynamically consistent reversible-irreversible PDE models based on the generalized Onsager principle} \label{sec:GOP}

\subsection{Generalized Onsager principle}
Consider a domain $\Omega$, and denote the thermodynamic variable $\phi$. We recall the generalized Onsager principle. It consists of three key ingredients: the state or thermodynamic variable $\phi$, the free energy $\cE$, and the kinetic equation dictated by a mobility matrix (or operator) $\cG$. In this paper, we name it 
\beq
\mbox{ the Onsager triplet: } (\phi, \cG, \cE).
\eeq 
The kinetic equation, stemming from the Onsager linear response theory,   is given by
\begin{subequations} \label{eq:evolution-general}
\begin{align}
&  \partial_t \phi (\bx,t) = - \cG \frac{\delta \cE}{\delta \phi} \mbox{ in } \Omega, \\
&\cB(\phi(\bx,t)) = g(\bx,t), \mbox{ on } \partial \Omega,
\end{align}
\end{subequations}
where  $\cB$ is a trace operator,  and $\cG$ is the mobility operator that contains two parts:
\beq
\cG=\cG_a+\cG_s.
\eeq
$\cG_s$  is symmetric and positive semi-definite  to ensure thermodynamically consistency,  $\cG_a$  is  skew-symmetric,  and $\frac{\delta \cE}{\delta \phi}$ is the variational derivative of $\cE$, known as the chemical potential. Then, the triplet $(\phi,\cG,\cE)$ uniquely defines a thermodynamically consistent model. One intrinsic property  of \eqref{eq:evolution-general} owing to the thermodynamical consistency  is the energy dissipation law
\begin{subequations} \label{EDL}
\begin{align}
& \frac{d \cE}{d t} = \Big( \frac{\delta \cE}{\delta \phi},   \frac{\partial\phi}{\partial t} \Big)+ \dot{\cE}_{surf} =\dot{\cE}_{bulk}+\dot{\cE}_{surf},\\
&\dot{\cE}_{bulk}=-\Big( \frac{\delta \cE}{\delta \phi}, \cG_s \frac{\delta \cE}{\delta \phi} \Big) \leq 0, \\
& \Big( \frac{\delta \cE}{\delta \phi}, \cG_a \frac{\delta \cE}{\delta \phi} \Big) = 0, \quad  \dot{\cE}_{surf}=\int_{\partial \Omega} g_b ds, 
\end{align}
\end{subequations}
where the inner product is defined by $(f, g) = \int_\Omega fgd\bx$, $\forall f, g \in L^2(\Omega)$,
and $\dot{\cE}_{surf}$ stems from the boundary contribution, and $g_b$ is the boundary integrand.
When $\cG_a=0$, \eqref{eq:evolution-general} is a purely dissipative system. When $\cG_s=0$, it is a purely dispersive system. $\dot{\cE}_{surf}$ vanishes only for suitable boundary conditions, which include periodic and certain physical boundary conditions.
When the mass, momentum, and total energy conservation are present in hydrodynamic models, these conservation laws are viewed as constraints imposed on the hydrodynamic variables. Then, the energy dissipation rate will have to be calculated subject to the constraints.

\subsection{Model reformulation with the energy quadratization (EQ) method}
Now,  we illustrate the idea of the energy quadratization method. Denote the total energy as
\beq
\cE(\phi) = \int_\Omega e d\bx,
\eeq 
with $e$ the energy density function. We denote $\cL_0$ as a semi positive definite linear operator that can be separated from $e$. 
Introduce the auxiliary variable
\beq \label{eq:EQ-intermediate-variable}
q = \sqrt{2\Big(e -\frac{1}{2}|\cL_0^{\frac{1}{2}} \phi|^2+ \frac{A_0}{|\Omega|}\Big)},
\eeq
where  $A_0 >0$ is a constant so that $q$ is a well defined real variable. Then, we rewrite the energy as
\beq \label{eq:EQ-energy}
\cE(\phi, q) =  \frac{1}{2}\Big(\phi, \cL_0\phi \Big) + \frac{1}{2}\Big( q, q \Big) - A_0,
\eeq 
With the EQ approach above,  we transform the free energy density into a quadratic one by introducing an auxiliary variable to ``remove'' the nonlinear terms from the energy density.
Assuming $q =q(\phi)$ and denoting $g(\phi)= \frac{\partial q}{\partial \phi}$,
we reformulate  \eqref{eq:evolution-general} into an equivalent form
\begin{subequations} \label{eq:evolution-general-EQ}
\begin{align}
&\partial_t \phi = -(\cG_a + \cG_s) \Big[ \cL_0 \phi + q g(\phi)] , \\
&\partial_t q = g(\phi):\partial_t \phi, \\
&q|_{t=0} =\left. \sqrt{2\Big(e -\frac{1}{2}|\cL_0^{\frac{1}{2}} \phi|^2+ \frac{A_0}{|\Omega|}\Big)}\right|_{t=0}.
\end{align}
\end{subequations}
%where $q|_{t=0} =\left. \sqrt{2\Big(f -\frac{1}{2}|\cL^{\frac{1}{2}} \phi|^2+ \frac{A_0}{|\Omega|}\Big)}\right|_{t=0}$.
%$ %\label{eq:EQ-initial-condition}
%q|_{t=0} =\left. \sqrt{2\Big(e -\frac{1}{2}|\cL^{\frac{1}{2}} \phi|^2+ \frac{A_0}{|\Omega|}\Big)}\right|_{t=0}.$
Now, instead of dealing with \eqref{eq:evolution-general} directly, we develop structure-preserving schemes for \eqref{eq:evolution-general-EQ}.
The advantage of using model \eqref{eq:evolution-general-EQ} over model \eqref{eq:evolution-general} is that the energy density is transformed into a quadratic one in \eqref{eq:evolution-general-EQ}.

Denote $\Psi =\begin{bmatrix} \phi \\ q \end{bmatrix}$. We rewrite \eqref{eq:evolution-general-EQ} into a vector form
\beq \label{eq:evolution-general-EQ-vector}
\partial_t \Psi  = -\cN(\Psi) \cL \Psi, 
\eeq 
where $\cN(\Psi)$ is the mobility operator, and $\cL$ is a linear operator.
\begin{subequations} \label{eq:nonlinear-terms}
\begin{align}
& \cN(\Psi) = \cN_s(\Psi) + \cN_a(\Psi), \\
& \cN_a(\Psi) = \cN_0^\ast \cG_a  \cN_0 , \\
&   \cN_s(\Psi) = \cN_0^\ast \cG_s \cN_0,
\end{align}
\end{subequations}
where 
$$
\cN_0= \begin{bmatrix}
   \bI & g(\phi))
\end{bmatrix} , % $\cB = diag(\cL, \,\,\, 1)_{n+1,n+1}$ 
\quad \cL = \begin{bmatrix}
\cL_0 & \\
& \bI
\end{bmatrix},
$$
and $\cN_0^\ast$ is the adjoint operator of $\cN_0$. We name it the Onsager-Q model, where
\beq  \label{eq:energy-law-EQ}
\frac{d \cE(\Psi)}{d t}
= \Big( \frac{\delta \cE}{\delta \Psi} \frac{\partial \Psi}{\partial t}, 1 \Big)
= -\Big( \ \cL \Psi,  \cN(\Psi) \cL \Psi  \Big) = -\Big( \cN_0 \cL \Psi,  \cG_s \cN_0 \cL \Psi  \Big)   \leq  0,
\eeq
where the energy of \eqref{eq:EQ-energy} are reformulated in a vector form as
\beq  \label{eq:EQ-energy-Vector}
\cE(\Psi) = \frac{1}{2} \Big(\Psi, \cL \Psi \Big) - A_0,
\eeq 
%When $\dot{\cE}_{surf}=0$. This is called the energy quadratization (EQ) reformulation.
Note that the energy in \eqref{eq:EQ-energy} or \eqref{eq:EQ-energy-Vector} is quadratized so that we can develop a paradigm to derive linear, energy stable numerical schemes for the model.

\begin{remark}
The EQ approach is also applicable when $q =q(\phi,\nabla \phi)$ and denoting
\beq   \label{eq:q_derivative}
g(\phi)= \frac{\partial q}{\partial \phi},  \quad  \bG(\phi) = \frac{\partial q}{\partial \nabla \phi},
\eeq
the kinetic equation \eqref{eq:evolution-general} can be reformulated into an equivalent form
\begin{subequations} \label{eq:evolution-general-EQ-q}
\begin{align}
&\partial_t \phi = -(\cG_a + \cG_s) \Big[ \cL_0 \phi + q g(\phi) - \nabla \cdot ( q \bG(\phi))\Big] , \\
&\partial_t q = g(\phi): \partial_t \phi + \bG(\nabla \phi): \nabla \partial_t \phi, \\
&q|_{t=0} =\left. \sqrt{2\Big(e -\frac{1}{2}|\cL_0^{\frac{1}{2}} \phi|^2+ \frac{A_0}{|\Omega|}\Big)}\right|_{t=0}.
\end{align}
\end{subequations} 
%\[\cN_0= (\bI_n, \,\,\, g(\phi) + \bG(\phi)\colon\nabla)_{n,n+1}\]
For the simplicity of notation, the approach presented in this paper derives structure-preserving ROM for \eqref{eq:evolution-general-EQ} and it is also applicable for \eqref{eq:evolution-general-EQ-q}.
\end{remark}

\section{Structure-preserving reduced order models (ROMs)} \label{sec:SROM}
In this section, we present the general structure-preserving framework for reduced order models (ROMs).  
We emphasize that the transformed model in \eqref{eq:evolution-general-EQ} is equivalent to \eqref{eq:evolution-general}. Thus, we focus on designing the ROM for \eqref{eq:evolution-general-EQ}, which in turn is a good surrogate model for \eqref{eq:evolution-general}.

\subsection{Model order reduction (MOR)}
Consider the solution for  \eqref{eq:evolution-general-EQ} as
$$
\Psi(\bx, t) = \sum_{k=1}^{\infty} \ba_k(t) \psi_k(\bx),
$$
where $\psi_k(\bx)$ are the spatial modes and $\ba_k(t)$ are the corresponding time coefficients. For the ROM, we look for a sequence of
\beq  \label{eq:ROM-solution}
\Psi_r(\bx, t) = \sum_{k=1}^r \ba_k(t) \psi_k(\bx),
\eeq
such that it provides an accurate approximation to $\Psi(\bx, t)$ and $\psi_k(\bx)$ are the optimal basis modes when r is small. In general, we would like the modes orthonormal, i.e., 
$$
\int_\Omega \psi_i(\bx) \psi_j(\bx) d\bx = \left\{
\begin{array}{l}
1, \quad i =j, \\
0, \quad  i \neq j.
\end{array}
\right.
$$

Plugging the ROM solution of \eqref{eq:ROM-solution} into \eqref{eq:evolution-general-EQ-vector}, we have
\beq \label{eq:ROM}
\sum_{k=1}^r \frac{d}{dt} \ba_k(t)  \psi_k(\bx) = - \cN \Big( \sum_{k=1}^r  \ba_k(t)  \psi_k(\bx) \Big) \cL  \Big[ \sum_{k=1}^r  \ba_k(t)  \psi_k(\bx) \Big].
\eeq 
It is clear that the ROM in \eqref{eq:ROM} preserves the original structure. In fact, if we denote the reduced energy as
\beq
\cE_r(\Psi_r(\bx,t)) = \frac{1}{2} \Big(  \Psi_r(\bx, t), ~ \cL \Psi_r(\bx, t) \Big) - A_0,
\eeq 
we can have the following energy dissipation law
\begin{eqnarray*}
\frac{d \cE_r(\Psi_r(\bx,t)) }{dt} 
&=& \Big( \cL \Psi_r(\bx,t), ~\frac{d}{dt} \Psi_r(\bx, t) \Big) \\
&= &\Big( \cL \Big[ \sum_{k=1}^r  \ba_k(t)  \psi_k(\bx) \Big], ~ \sum_{k=1}^r \frac{d}{dt} \ba_k(t)  \psi_k(\bx)  \Big)  \\
&=& - \Big(  \cL  \Big[ \sum_{k=1}^r  \ba_k(t)  \psi_k(\bx) \Big],  ~ \cN \Big( \sum_{k=1}^r  \ba_k(t)  \psi_k(\bx) \Big) \cL  \Big[ \sum_{k=1}^r  \ba_k(t)  \psi_k(\bx) \Big]\Big) \\
& \leq & 0.
\end{eqnarray*}

\subsection{POD-Galerkin method}

For these structure-preserving spatial discretizations, either finite difference, finite elements, or spectral methods can be used. For simplicity, we assume periodic boundary conditions in the rest of this paper and use pseudo-spectral discretization for space. Then, we focus on the proper orthogonal decomposition (POD) for the selection of optimal spatial modes and the construction of structure-preserving ROMs in combination with Galerkin projection. 

\subsubsection{Pseudo-spectral spatial discretization}

We consider a rectangular domain $\Omega =[l_x,  r_x] \times [l_y, r_y]$, and denote $L_x = r_x - l_x$ and $L_y = r_y - l_y$. For Fourier spectral method, we discretize the domain into equally distanced rectangular meshes with $h_x = \frac{L_x}{N_x}$ and $h_y = \frac{L_y}{N_y}$, with $N_x$, $N_y$ the number of meshes in $x$ and $y$ directions respectively, and $h_x$, $h_y$ the corresponding mesh sizes. Then, we have the discrete coordinates
\beq
(x_m, y_n) = (l_x + mh_x, l_y +n h_y), \quad m = 0, 1,2, \cdots, N_x-1,  \quad n = 0, 1, 2, \cdots, N_y-1.
\eeq 
Given the domain is periodic, we have $x_{N_x} = x_0$, $y_{N_y} = y_0$. Assume $N_x$ and $N_y$ are even numbers with $N_x = 2K_x$ and $N_y=2K_y$. 

Denote $\phi_{ij}$ the numerical approximation for $\phi(x_i, y_j)$. Then, we have the discrete Fourier expansion in 2D as
\beq
\phi_{mn} = \frac{1}{N_x N_y} \sum_{k=-K_x+1}^{K_x} \sum_{l=-K_y+1}^{K_y} \hat{\phi}_{kl} \exp( 2\pi i (k  \frac{x_m}{L_x} + l  \frac{y_n}{L_y}))
\eeq 
and the corresponding Fourier inverse transform is given as 
\beq
\hat{\phi}_{kl}  = \sum_{m=0}^{N_x-1} \sum_{n=0}^{N_y-1} \phi_{mn} \exp(-2\pi i(k \frac{x_m}{L_x} + l \frac{y_n}{L_y}))
\eeq 

With the notations for the Fourier transform, we can calculate the first-order and second-order partial derivatives as
\beq
(\cD_{N_x} \phi)_{ij} = \frac{1}{N_x N_y}  \sum_{k=-K_x+1}^{K_x} \sum_{l=-K_y+1}^{K_y}  \frac{2\pi k i}{L_x} \hat{\phi}_{k,l} \exp(2\pi i(k x_i + l y_j)) 
\eeq 
and the second-order partial derivative is given as
\beq
(\cD^2_{N_x} \phi)_{ij} = \frac{1}{N_x N_y}  \sum_{k=-K_x+1}^{K_x} \sum_{l=-K_y+1}^{K_y}  (-\frac{4\pi^2 k^2 }{L_x^2}) \hat{\phi}_{k,l} \exp(2\pi i(k x_i + l y_j))
\eeq 
Similarly we can define the differential operators for $(\cD_{N_y} \phi)_{ij}$ and $(\cD^2_{N_y}\phi)_{ij}$. Then, we can introduce the discrete Laplacian, gradient, and divergence operators as
\beq
\Delta_N \phi = (\cD_{N_x}^2  + \cD_{N_y}^2) \phi, \quad \nabla_N \phi = \begin{pmatrix} \cD_{N_x} \phi \\ \cD_{N_y} \phi \end{pmatrix}, \quad \nabla_N \cdot \begin{pmatrix} \phi \\ \psi  \end{pmatrix} = \cD_{N_x} \phi + \cD_{N_y} \psi.
\eeq 
Then, we finally can introduce the  discrete operator
\beq  \label{eq:G-L}
\quad \cL_h = -\varepsilon^2\Delta_N +  C, \quad C>0.
\eeq 
Here, $C$ is a positive stabilizing constant. The operator $\cL^{-1}$ can be defined as
\beq
\Big(  \cL_h^{-1} \phi \Big)_{mn} = \frac{1}{N_x N_y} \sum_{k=-K_x+1}^{K_x} \sum_{l=-K_y+1}^{K_y} \frac{1}{\varepsilon^2 \lambda_{k,l} + C} \hat{\phi}_{kl} \exp( 2\pi i (k  \frac{x_m}{L_x} + l  \frac{y_l}{L_y})),
\eeq 
Here $\lambda_{k,l} = (\frac{2k \pi}{L_x})^2 + (\frac{2l \pi}{L_y})^2$. Other exponential operators can be defined in a similar manner.
Given the discrete functions in the 2D mesh $\phi, \psi \in \hR^{N_x, N_y}$, we can define the inner product and induced $l^2$ norm as
\beq
\| \phi \|_2 = \sqrt{(\phi, \phi)}, \quad  (\phi, \psi) =  h_x h_y \sum_{i=0}^{N_x-1} \sum_{j=0}^{N_y-1} \phi_{ij} \psi_{ij}.
\eeq 
It can be easily shown that the following summation by parts formula hold
\beq
(\phi, \Delta_N \psi) = -(\nabla_N \phi, \nabla_N \psi), \quad (\phi, \Delta^2_N \psi) = (\Delta_N \phi, \Delta_N \psi).
\eeq 

For more properties, interested readers can refer to \cite{Gong&Zhao&WangACM}. Without loss of generality, we use the notation $\cL_h$ to denote the corresponding discrete operator for $\cL$ and $\cN_h$ for $\cN$ using the pseudo-spectral spatial discretization.

\subsubsection{POD-Galerkin projection}
To give an accurate low-dimensional approximation from a subspace spanned by a set of reduced basis of dimension $r$ in $\mathbb R^n$, we use the proper orthogonal decomposition (POD) to construct a set of global basis, also known as POD modes, from a singular value decomposition (SVD) of some snapshot data of the system. Then we use Galerkin projection as the means for dimension reduction.

Suppose we have the sampling of the phase variable $\phi(\bx, t)$ as
\beq \label{eq:sampling-phi}
\Phi =  \begin{bmatrix}
\Phi_1  & \Phi_2 & \cdots  & \Phi_m
\end{bmatrix},
\eeq 
which might be measurements from simulations or experimental data. Here $\Phi_k$ is the data collected from $t=t_k$ on the equally distanced rectangular meshes in a vector form, i.e., $\Phi_k  \in \mathbb{R}^{n}$ with $n=N_x \times N_y$.
Typically, $n \gg m$. Denote the full singular value decomposition (SVD) of the data matrix $\Phi \in \mathbb{R}^{n,m}$ as
$$
\Phi= \hat{\bU}_\phi \hat{\Sigma}_\phi \hat{\bV}_\phi^T, 
$$
where $\hat{\bU}_\phi\in \mathbb{R}^{n,n}$, and $\hat{\bV}_\phi\in \mathbb{R}^{m,m}$ and $\hat{\Sigma}_\phi \in \mathbb{R}^{n,m}$. Our objective is to find $r \ll m$ optimal spatial modes that are necessary to represent the spatial dynamics accurately. Given a threshold $\varepsilon$, we can find $r$ such that $\| \Phi - \Phi_r\| < \varepsilon$, where $\Phi_r$ is from the reduced SVD for $\Phi$ as
$$
\Phi \approx \Phi_r = \bU_\phi \Sigma_\phi \bV_\phi^T,
$$
where $\bU_\phi \in \mathbb{R}^{n, r}$, $\Sigma_\phi \in \mathbb{R}^{r, r}$ and $\bV_\phi \in \mathbb{R}^{m, r}$. And we denote our optimal basis modes
\beq \label{eq:phi-r-mode}
\bU_\phi := \hat{\bU}_\phi[:, 1:r],
\eeq 
where the truncation preserves the $r$ most dominant modes. The truncated $r$ modes for $\Phi$ are then used as the low-rank, orthogonal basis to represent the spatial dynamics. Similarly, for the auxiliary variable $q(\bx,t)$, we can obtain the data from the sampling  in \eqref{eq:sampling-phi}, as
\beq \label{eq:sampling-q}
\bQ  =  \begin{bmatrix}
h(\Phi_1)  &  h(\Phi_2) & \cdots  & h(\Phi_m)
\end{bmatrix},
\eeq 
where $h(\Phi_k) \in \hR^n$ can be treated as the sampling for $q(\bx,t)$ at $t_k$ on the same spatial meshes. In a similar manner, we calculate the SVD for $\bQ\in \hR^{n,m}$ as
$$
\bQ = \hat{\bU}_q \hat{\Sigma}_q \hat{\bV}_q^T,
$$
and have the reduced SVD by picking $r\ll m$ as
\beq
\bQ_r = \bU_q \Sigma_q \bV^T_q.
\eeq 
By this approach, we obtain the truncated $r$ modes for $q(\bx,t)$ as
\beq  \label{eq:q-r-mode}
\bU_q :=\hat{\bU}_q[:,1:r].
\eeq 
With \eqref{eq:phi-r-mode} and \eqref{eq:q-r-mode}, we have the $r$ modes for the general variable $\Psi$ as
\beq
\bU = 
\begin{bmatrix}
\bU_\phi &  0 \\
0 & \bU_q
\end{bmatrix}.
\eeq 

Denoting
\beq
\ba(t) = \begin{bmatrix}
\ba_\phi(t) \\ \ba_q(t)
\end{bmatrix},  \quad \Psi = \begin{bmatrix}
\phi \\ q
\end{bmatrix},
\eeq 
where $\ba_\phi(t) \in \hR^r$ and $\ba_q(t) \in \hR^r$ are the time-dependent coefficient vectors, we approximate the solution using the POD expansion:
\beq
\Psi = \bU \ba(t),  \mbox{  i.e., }  \quad  \phi = \bU_\phi \ba_\phi(t), \quad q = \bU_q \ba_q(t).
\eeq

\begin{remark}
Note that the POD approach constructs a reduced basis that is optimal in the sense that a certain approximation error concerning the snapshot data is minimized. The minimum 2-norm error from approximating the snapshot data using the POD modes is given by
\[\sum_{j=1}^m \Big\Vert \phi_j-\Phi\Phi^T\phi_j\Big\Vert_2^2=\sum_{j=r+1}^{d_1} \sigma_{j}^2,~~\sum_{j=1}^m \Big\Vert h(\phi_j)-\bQ\bQ^T h(\phi_j)\Big\Vert_2^2=\sum_{j=r+1}^{d_2} \lambda_{j}^2,\]
where $d_1$ and $d_2$ are the rank of snapshot matrix $\Phi$ and $\bQ$ respectively, and $\sigma_{1}\geq\sigma_{2}\geq\cdots\geq\sigma_{d_1}>0$ and $\lambda_{1}\geq\lambda_{2}\geq\cdots\geq\lambda_{d_2}>0$ are the nonzero singular values of $\Phi$ and $\bQ$ respectively. For more details on the POD basis, we refer the reader to \cite{kunisch2002POD}.   
\end{remark}

Then, we apply POD-Galerkin projection to \eqref{eq:ROM} and derive the POD-ROM as
\beq \label{eq:POD-ROM}
\frac{d \ba(t)}{dt} = - \bU^T \cN_h \Big(\bU \ba(t) \Big) \cL_h \Big[\bU \ba(t)\Big],
\eeq 
which we name POD-ROM-vanilla.
By solving this system of much smaller dimensions as \eqref{eq:POD-ROM}, the solution of a high-dimensional nonlinear dynamical system can be approximated.

However, POD-ROM in \eqref{eq:POD-ROM} does not necessarily respect the energy dissipation law as the full model in \eqref{eq:energy-law-EQ}. Specifically, if we calculate the energy dissipation rate of \eqref{eq:POD-ROM} with respect to the discrete energy
$$
\cE_r(\bU \ba(t)) = \frac{1}{2} \Big[ \bU \ba(t) \Big]^T \cL_h \Big[\bU \ba(t) \Big],
$$
we have
\begin{subequations}
\begin{align}
\frac{ d \cE_r(\bU \ba(t))}{dt} 
&=  \Big( \cL_h \Big[\bU \ba(t) \Big] \Big)^T  \bU \frac{d\ba(t)}{dt}\\
&= -\Big( \cL_h \Big[\bU \ba(t) \Big] \Big)^T \bU \bU^T \cN_h\Big(\bU \ba(t)\Big) \cL_h \Big[\bU \ba(t) \Big],
\end{align}
\end{subequations}
where $\bU \bU^T \cN_h(\bU \ba(t))$ is not necessarily a semi-positive-definite matrix, which might violate the energy dissipation law.

\subsection{General POD-ROM structure-preserving framework}
To address the issue above, we propose two approaches that can automatically preserve the energy dissipation structure in the POD-ROM.

\subsubsection{Approach I: POD-ROM-I}
As a first remedy, we focus on modifying the matrix $\bU \bU^T \cN_h(\bU \ba(t)) $ to make it semi-positive definite. Inspired by \cite{WangZhu-POD-1}, we modify the mobility operator in \eqref{eq:POD-ROM} by assuming that there exists a mobility operator, $\hat\cN_h$ with the same property like $\cN_h$ such that,
\begin{equation}
\bU^T\cN_h=\hat\cN_h\bU^T.
\end{equation}
Then, by multiplying $\bU$ on both sides of the equation, we have
\begin{equation}
\hat\cN_h = \bU^T\cN_h\bU.
\end{equation}
Replacing the mobility operator $\cN_h$ in \eqref{eq:POD-ROM} with $\hat\cN_h$, we came up with the following POD-ROM
$$
\frac{d \ba(t)}{dt} =-\hat\cN_h \Big(\bU \ba(t) \Big) \bU^T \cL_h\Big[\bU \ba(t)\Big],
$$
i.e.,
\beq \label{eq:POD-ROM-I}
\frac{d \ba(t)}{dt} 
=-\bU^T\cN_h \Big(\bU \ba(t) \Big)\bU \bU^T \cL_h\Big[\bU \ba(t)\Big].
\eeq 
We named \eqref{eq:POD-ROM-I} as POD-ROM-I.

\begin{theorem}[Energy Stability]
The POD-ROM-I in \eqref{eq:POD-ROM-I} preserves the energy dissipation law
\beq
\frac{ d \cE_r(\bU \ba(t))}{dt}  =- \Big( \bU \bU^T\cL_h \Big[\bU \ba(t) \Big],~  \cN_h\Big(\bU \ba(t)\Big) \bU \bU^T \cL_h \Big[\bU \ba(t) \Big]\Big) \leq 0,
\eeq 
where the discrete energy is defined as
$
\cE_r(\bU \ba(t)) = \frac{1}{2} \Big[ \bU \ba(t) \Big]^T \cL_h \Big[\bU \ba(t) \Big].
$
\end{theorem}

\begin{proof}
As a matter of fact, this energy dissipation law can be easily observed if take the following calculations
\begin{subequations}
\begin{align}
\frac{ d \cE_r(\bU \ba(t))}{dt} 
&=  \Big( \cL_h \Big[\bU \ba(t) \Big] \Big)^T  \bU \frac{d\ba(t)}{dt}\\
&= -\Big( \cL_h \Big[\bU \ba(t) \Big] \Big)^T \bU \bU^T \cN_h\Big(\bU \ba(t)\Big) \bU \bU^T \Big( \cL_h \Big[\bU \ba(t) \Big] \Big)\\
&= - \Big( \bU \bU^T\cL_h \Big[\bU \ba(t) \Big],~  \cN_h\Big(\bU \ba(t)\Big) \bU \bU^T \cL_h \Big[\bU \ba(t) \Big]\Big)\\
& \leq 0.
\end{align}
\end{subequations}
This completes the proof.
\end{proof}

However, even though the approach in \eqref{eq:POD-ROM-I} respects the energy dissipation property, it has a modified energy dissipation rate, which might lead to inaccurate dissipative dynamics. 

\subsubsection{Approach II: POD-ROM-II}
To overcome these issues, we introduce a new approach that is inspired by \cite{EggerNM2019}, where the author develops energy-stable numerical algorithms for dissipative systems. Specifically, we reformulate the reduced model by applying $\cL_h^*$ on both sides of the equation and then obtain the following POD-ROM by Galerkin projection as
\beq \label{eq:POD-ROM-II}
\bU^T \cL_h^* \bU \frac{d \ba(t)}{dt} = - \bU^T \cL_h^* \cN_h \Big(\bU \ba(t) \Big) \cL_h \Big[\bU \ba(t)\Big],
\eeq 
which we name POD-ROM-II.
Note that $\cL_h = \begin{bmatrix}
\cL_{0,h} & 0 \\ 0 & \bI
\end{bmatrix}$ is an invertible operator, such that the problem \eqref{eq:POD-ROM-II} is well-posed.
\begin{theorem}[Energy Stability]
The POD-ROM-II in \eqref{eq:POD-ROM-II} preserves the energy dissipation law
\beq
\frac{ d \cE_r(\bU \ba(t))}{dt}  =- \Big( \cL_h \Big[\bU \ba(t) \Big], ~ \cN_h\Big(\bU \ba(t)\Big) \cL_h \Big[\bU \ba(t) \Big]\Big) \leq 0,
\eeq 
where the discrete energy is defined as
$
\cE_r(\bU \ba(t)) = \frac{1}{2} \Big[ \bU \ba(t) \Big]^T \cL_h \Big[\bU \ba(t) \Big].
$
\end{theorem}

\begin{proof}
As a matter of fact, this energy dissipation law can be easily observed if take the following calculations
\begin{subequations}
\begin{align}
\frac{ d \cE_r(\bU \ba(t))}{dt} 
&=  \Big( \cL_h \Big[\bU \ba(t) \Big] \Big)^T  \bU \frac{d\ba(t)}{dt}\\
&= \ba(t)^T \bU^T \cL_h^\ast \bU \frac{d\ba(t)}{dt} \\
&= -\Big( \cL_h \Big[\bU \ba(t) \Big] \Big)^T\cN_h\Big(\bU \ba(t)\Big) \Big( \cL_h \Big[\bU \ba(t) \Big] \Big)\\
&= - \Big( \cL_h \Big[\bU \ba(t) \Big], ~ \cN_h\Big(\bU \ba(t)\Big) \cL_h \Big[\bU \ba(t) \Big]\Big)\\
& \leq 0.
\end{align}
\end{subequations}
This completes the proof.
\end{proof}

\begin{rem}
Note that POD-ROM-I in \eqref{eq:POD-ROM-I} and POD-ROM-II \eqref{eq:POD-ROM-II} are rather generic, and they apply to all the reversible-irreversible models that fit in the generic form. However, POD-ROM-II in \eqref{eq:POD-ROM-II} is superior to either POD-ROM-vanilla in \eqref{eq:POD-ROM} or POD-ROM-I in \eqref{eq:POD-ROM-I}, since \eqref{eq:POD-ROM-II}  also preserves the original energy dissipation structure, as shown in \eqref{eq:energy-law-EQ}.
\end{rem}

We can also expand the notations and rewrite POD-ROM-II in matrix forms. The equation of \eqref{eq:POD-ROM-II} can be written as
$$
\begin{bmatrix}
\bU_\phi^T \cL_{0,h}^\ast  \bU_\phi & \\
&\bU_q^T \bU_q 
\end{bmatrix}
\frac{d}{dt} 
\begin{bmatrix}
\ba_\phi(t) \\
\ba_q (t)
\end{bmatrix}
= 
\begin{bmatrix}
\bU_\phi^T & \\
&\bU_q^T 
\end{bmatrix}
\begin{bmatrix}
\cL_{0,h}^\ast & 0 \\
0 & \bI
\end{bmatrix}
\cN_h\Big(\begin{bmatrix}
\bU_\phi \ba_\phi(t) \\
\bU_q \ba_q(t)
\end{bmatrix}\Big)
\begin{bmatrix}
\cL_{0,h} & 0 \\
0 & \bI
\end{bmatrix}
\begin{bmatrix}
\bU_\phi \ba_\phi (t) \\
\bU_q \ba_q (t)
\end{bmatrix}.
$$
The structure-preserving model of \eqref{eq:POD-ROM-II} reads
\begin{subequations}
\begin{align}
\bU_\phi^T \cL_{0,h}^\ast \bU_\phi \frac{d \ba_\phi(t)}{dt} &= -\bU_\phi^T  \cL_{0,h}^\ast \cG\Big(  \cL_{0,h} \bU_\phi \ba_\phi(t) + g(\bU_\phi\ba_\phi(t)) \bU_q  \ba_q(t) \Big), \\
\frac{d \ba_q(t)}{dt} &= -\bU_q^T g(\bU_\phi \ba_\phi(t))^T \cG \Big( \cL_{0,h} \bU_\phi \ba_\phi(t) + g(\bU_\phi \ba_\phi(t)) \bU_q \ba_q(t) \Big).
\end{align}
\end{subequations}

Furthermore, the POD-ROM can be simplified as
\begin{subequations} \label{eq:POD-ROM-linear}
\begin{align}
A_0 \frac{d \ba_\phi(t)}{dt} = -A_1 \ba_\phi(t) - A_2 \ba_q(t), \\
\frac{d \ba_q(t)}{dt} =  - A_3 \ba_\phi(t) - A_4 \ba_q(t),
\end{align}
\end{subequations}
with the following notations for the operators
\begin{subequations} \label{eq:POD-ROM-linear-2}
\begin{align}
& A_ 0= \bU_\phi^T \cL_{0,h}^\ast \bU_\phi, \\
& A_1 = \bU_\phi^T  \cL_{0,h}^\ast \cG \cL_{0,h} \bU_\phi, \\
& A_2(t) = \bU_\phi^T  \cL_{0,h}^\ast \cG g(\bU_\phi\ba_\phi(t)) \bU_q , \\
& A_3(t) = \bU_q^T g(\bU_\phi\ba_\phi(t))^T\cG\cL_{0,h}\bU_\phi,\\
& A_4(t) = \bU_q^T g(\bU_\phi \ba_\phi(t))^T \cG g(\bU_\phi \ba_\phi(t)) \bU_q.
\end{align}
\end{subequations}
The model in \eqref{eq:POD-ROM-linear}-\eqref{eq:POD-ROM-linear-2} will provide guidance on designing structure-preserving numerical algorithms for solving the POD-ROMs.

\section{Linear time-stepping  algorithms for the POD-ROM} \label{Sec:time-marching}

With our framework, the POD-ROM in \eqref{eq:POD-ROM-I} and \eqref{eq:POD-ROM-II} become trivial to be solved. Consider the time domain $[0, T]$. We discretize it into equally distanced intervals $0=t_0 < t_1 < \cdots < t_N=T$ with $h=\frac{T}{N}$, i.e. $t_i = ih$. Then we denote $\ba^{n}$ as the numerical approximation to $\ba(t_n)$. The initial value can be obtained as 
$$
\ba_\phi^0 = \bU_\phi^T\phi_0, \quad \ba_q^0=\bU_q^T h(\phi_0).
$$
With these notations, we discuss the numerical integration for the POD-ROM.

To integrate the ROM in time, we need a structure-preserving time-integrator. The widely used time integration for structure-preserving ROM is the average vector field (AVF) method. 
However, its drawback is obvious. It is highly nonlinear and computationally expensive, and some nonlinear iterative method has to be used (which is also not easy to implement). The existence and uniqueness of the solution strongly limit the time-step size.

\subsection{Linear semi-implicit structure-preserving time integration}
Among the many advantages, our proposed POD-ROMs can be easily integrated in time with linear numerical schemes. If we take the semi-implicit Crank-Nicolson time discretization, we will have the following schemes for the POD-ROM-II in \eqref{eq:POD-ROM-II}.

\begin{scheme}[CN scheme for the POD-ROM-II] \label{sch:POD-ROM-CN}
After we calculated $\ba^{n-1}$ and $\ba^n$, we can obtain $\ba^{n+1}:=\begin{bmatrix} \ba_\phi^{n+1}\\\ \ba_q^{n+1}\end{bmatrix}$ via the following linear scheme
\beq \label{eq:POD-ROM-CN}
\bU^T \cL_h^\ast \bU \frac{\ba^{n+1} - \ba^n}{\delta t} = -\bU^T \cL_h^\ast \cN_h( \bU \overline{\ba}^{n+\frac{1}{2}})  \cL_h [\bU \ba^{n+\frac{1}{2}}],
\eeq 
with the notations $\overline{\ba}^{n+\frac{1}{2}} = \frac{3}{2}\ba^{n} - \frac{1}{2}\ba^{n-1}$ and $\ba^{n+\frac{1}{2}} = \frac{1}{2} \ba^{n+1} + \frac{1}{2}\ba^n$.
\end{scheme}

\begin{thm}
\label{thm:POD-ROM-CN}
The scheme \ref{sch:POD-ROM-CN} is linear, and it preserves the discrete energy dissipation law
\beq
\cE_r(\bU \ba^{n+1}) - \cE_r(\bU \ba^n) =  -\delta t \Big( \cL_h[ \bU \ba^{n+\frac{1}{2}}] \Big)^T  \cN_h( \bU \overline{\ba}^{n+\frac{1}{2}}) \Big( \cL_h [\bU \ba^{n+\frac{1}{2}}] \Big)
\eeq
where the discrete energy $\cE_r$ is defined as $\cE_r(\bU \ba ) = \frac{1}{2} \Big( \bU \ba \Big)^T \cL_h \Big(\bU \ba \Big)$.
\end{thm}

\begin{proof}
It is easy to observe that only a linear system needs to be solved at each time step in \eqref{eq:POD-ROM-CN}, i.e., Scheme \ref{sch:POD-ROM-CN} is linear. The energy stability equality is obtained by multiplying $\delta t (\ba^{n+\frac{1}{2}})^T$ on both sides of  \eqref{eq:POD-ROM-CN}.
\end{proof}

Similarly, if we use the semi-implicit second-order backward differential formula for the temporal discretization, we arrive at the following second-order linear numerical algorithm.

\begin{scheme}[BDF2 scheme for the POD-ROM-II] \label{sch:POD-ROM-BDF2}
After we calculated $\ba^{n-1}$ and $\ba^n$, we can obtain $\ba^{n+1}:=\begin{bmatrix} \ba_\phi^{n+1}\\ \ba_q^{n+1}\end{bmatrix}$ via the following linear scheme
\beq \label{eq:POD-ROM-BDF2}
\bU^T \cL_h^\ast \bU  \frac{3\ba^{n+1} - 4\ba^n + \ba^{n-1}}{2\delta t} = -\bU^T \cL_h^\ast \cN_h( \bU \overline{\ba}^{n+1})  \cL_h [ \bU \ba^{n+1}],
\eeq 
with the notation $\overline{\ba}^{n+1} = 2\ba^n - \ba^{n-1}$.
\end{scheme}

\begin{thm}
\label{thm:POD-ROM-BDF2}
The scheme \ref{sch:POD-ROM-BDF2} is linear, and it satisfies the following discrete energy dissipation law
\beq 
\hat{\cE}_r(\bU \ba^{n+1}, \bU \ba^n) - \hat{\cE}_r(\bU \ba^n, \bU \ba^{n-1}) \leq
-\delta t \Big( \cL_h [\bU \ba^{n+1}] \Big)^T \cN_h( \bU \overline{\ba}^{n+1})  \Big( \cL_h [\bU \ba^{n+1}] \Big),
\eeq
where the modified discrete free energy is defined as 
$$
\hat{\cE}_r(\bU \ba_1, \bU \ba_2) = \frac{1}{4} \Big[ \bU \ba_1 \Big]^T \cL_h \Big[\bU \ba_1 \Big] + \frac{1}{4} \Big[ \bU (2\ba_1-\ba_2) \Big]^T \cL_h \Big[\bU (2\ba_1 -\ba_2) \Big].
$$
\end{thm}

\begin{proof}
Notice the inequality
\begin{equation*}
\begin{aligned}
   (\frac{3a-4b+c}{2},~a)&=\frac{1}{4} \Big[ (a,~a) + (2a-b,~2a-b) - (b,~b) - (2b-c,~2b-c) + (a-2b+c,~a-2b+c)\Big]\\
   &\geq \frac{1}{4} \Big[ (a,~a) + (2a-b,~2a-b) - (b,~b) - (2b-c,~2b-c)\Big].
\end{aligned} 
\end{equation*}

If we multiply $\delta t (\ba^{n+1})^T$ on both sides of \eqref{eq:POD-ROM-BDF2}, the energy stability inequality is obtained.
\end{proof}

\begin{thm}
The Scheme \ref{sch:POD-ROM-CN} and Scheme \ref{sch:POD-ROM-BDF2} are both linear and uniquely solvable. 
\end{thm}

\begin{proof}
Given the similarity between Scheme \ref{sch:POD-ROM-CN} and Scheme \ref{sch:POD-ROM-BDF2}, we only show the proof for Scheme \ref{sch:POD-ROM-CN}. After some basic algebraic calculation, we can rewrite Scheme \ref{sch:POD-ROM-CN} as
\beq
\Big[\frac{1}{\delta t} \bU^T \cL_h^\ast \bU +\frac{1}{2} \bU^T \cL_h^\ast \cN_h(\bU \overline{\ba}^{n+\frac{1}{2}}) \cL_h\bU \Big]\ba^{n+1} = \frac{1}{\delta t} \bU^T \cL_h^\ast [\bU  \ba^n] -\frac{1}{2} \bU^T \cL_h^\ast \cN(\bU \overline{\ba}^{n+\frac{1}{2}}) \cL_h[\bU \ba^n].
\eeq 
% Meanwhile, for Scheme \ref{sch:POD-ROM-BDF2}, the solution is given by
% \beq
% \ba^{n+1} = \Big[ \frac{3}{2\delta t} + \bU^T \cN( \bU \overline{\ba}^{n+1}) \cL  \bU   \Big]^{-1} \Big[ \frac{4 \ba^n - \ba^{n-1}}{2\delta t}  \Big].
% \eeq 

To show the existence and uniqueness of the solution for \eqref{eq:POD-ROM-CN}, we only need to show there is only a zero solution for
$$
A\ba^{n+1} =0, \quad\text{where}\quad A = \frac{1}{\delta t} \bU^T \cL_h^\ast \bU +\frac{1}{2} \bU^T \cL_h^\ast \cN_h(\bU \overline{\ba}^{n+\frac{1}{2}}) \cL_h\bU .
$$
If $\ba$ is a solution to $A\ba^{n+1}=0$, we have
$$
0 =  \ba^T A \ba = \frac{ \ba^T \bU^T \cL^\ast_h \bU \ba}{\delta t} +  \frac{1}{2}\ba^T \Big[  \bU^T \cL_h^\ast \cN_h( \bU \overline{\ba}^{n+\frac{1}{2}})  \cL_h \bU \Big] \ba \geq  \frac{ \ba^T \bU^T \cL^\ast_h \bU \ba}{\delta t},
$$
i.e. $\ba=0$, by noticing $\bU^T \cL^\ast_h \bU$ is a positive definite matrix.  So, $A$ is invertible, and there is a unique solution for Scheme \ref{eq:POD-ROM-CN}. Furthermore, we have shown that $A$ is a positive definite matrix.
\end{proof}

For the POD-ROM-I in \eqref{eq:POD-ROM-I}, similar numerical techniques can be applied. Specifically, we have the following time-marching linear numerical schemes that will also be used for comparisons in the numerical result section.
\begin{scheme}[CN scheme for POD-ROM-I] \label{scheme:CN-ApproachI}
The CN scheme for \eqref{eq:POD-ROM-I} reads
\beq 
\frac{\ba^{n+1} - \ba^n}{\delta t} =  - \bU^T \cN_h( \bU \overline{\ba}^{n+\frac{1}{2}}) \bU \bU^T \cL_h[ \bU \ba^{n+\frac{1}{2}}].
\eeq 
\end{scheme}

\begin{scheme}[BDF2 scheme for POD-ROM-I] \label{scheme:BDF2-ApproachI}
The BDF2 scheme for \eqref{eq:POD-ROM-I} reads 
\beq
\frac{3\ba^{n+1} - 4\ba^n + \ba^{n-1}}{2\delta t} = - \bU^T \cN_h( \bU \overline{\ba}^{n+1}) \bU \bU^T \cL_h[ \bU \ba^{n+1}].
\eeq 
\end{scheme}

From the discussion above, we conclude that Scheme \ref{sch:POD-ROM-CN} and Scheme \ref{sch:POD-ROM-BDF2} are linear, uniquely solvable, and preserve the energy dissipation structure. However, they both have one defect.  We recall the original definition for $q$ in \eqref{eq:EQ-intermediate-variable}, denoting $q:=h(\phi)$. After temporal discretization, the numerical result $\bU_q \ba_q^{n+1}$ from Scheme \ref{sch:POD-ROM-CN} or Scheme \ref{sch:POD-ROM-BDF2} is not necessarily equal to $h(\bU_\phi \ba_\phi^{n+1})$ anymore, which leads to numerical error.

\subsection{Relaxation technique to improve the accuracy}

One remedy to fix the issue above is to utilize the relaxation technique we proposed to improve the accuracy and stability of EQ and SAV methods \cite{ZhaoIEQrelax, Jiang-SAV-relax}. Namely, We can relax the numerical solution $\bU_q \ba_q^{n+1}$ from Scheme \ref{sch:POD-ROM-CN} and Scheme \ref{sch:POD-ROM-BDF2} such that the numerical error $ \| \bU_q \ba_q^{n+1} - h(\bU_\phi \ba_\phi^{n+1}) \|$ will be damped to zero gradually. Therefore, we develop the following relaxed schemes, which will be used in this paper to conduct numerical simulations.

\begin{scheme}[Relaxed CN scheme for the POD-ROM-II] \label{sch:Relaxed-POD-ROM-CN}
After we calculated $\ba^{n-1}$ and $\ba^n$, we can obtain $\ba^{n+1}$ via the following three steps:
\begin{itemize}
\item Step 1. Obtain $\hat{\ba}^{n+1}:=\begin{bmatrix} \ba_\phi^{n+1}\\\hat{\ba}_q^{n+1}\end{bmatrix}$ via the linear scheme
\beq \label{eq:Relaxed-POD-ROM-CN}
\bU^T \cL_h^\ast \bU \frac{\hat{\ba}^{n+1} - \ba^n}{\delta t} = -\bU^T \cL_h^\ast \cN_h( \bU \overline{\ba}^{n+\frac{1}{2}})  \cL_h [ \bU \hat{\ba}^{n+\frac{1}{2}}],
\eeq 
with the notations $\overline{\ba}^{n+\frac{1}{2}} = \frac{3}{2}\ba^{n} - \frac{1}{2}\ba^{n-1}$ and $\hat{\ba}^{n+\frac{1}{2}} = \frac{1}{2} \hat{\ba}^{n+1} + \frac{1}{2}\ba^n$.
\item Step 2. Update $\bU_q\ba_q^{n+1}$ via the relaxation strategy 
$$
\bU_q\ba_q^{n+1} = \xi_0 \bU_q\hat{\ba}_q^{n+1} + (1-\xi_0) h(\bU_\phi\ba_\phi^{n+1}), \quad \xi_0 \in [0, 1],
$$
where $\xi_0$ is a solution for the optimization problem:
\beq
\label{eq:relaxed-CN-xi}
\begin{aligned}
\xi_0=\min_{\xi\in[0,1]}\xi, \quad\text{s.t.}~&\frac{1}{2}(\bU_q\ba_q^{n+1},~\bU_q\ba_q^{n+1})-\frac{1}{2}(\bU_q\hat{\ba}_q^{n+1},~\bU_q\hat{\ba}_q^{n+1})\\
&\leq\delta t\eta\Big(\cL_h[\bU\hat{\ba}^{n+\frac{1}{2}}],~\cN_h(\bU\overline{\ba}^{n+\frac{1}{2}})\cL_h[\bU\hat{\ba}^{n+\frac{1}{2}}]\Big),
\end{aligned}
\eeq
with an artificial parameter $\eta\in[0,1]$ that can be assigned.
\item Step 3. Update $\ba^{n+1}$ as  $\ba^{n+1}:=\begin{bmatrix} \ba_\phi^{n+1}\\ \ba_q^{n+1}\end{bmatrix}$.
\end{itemize}

\end{scheme}

\begin{rem}
The Scheme \ref{sch:Relaxed-POD-ROM-CN} relaxes the solution $\bU_q\hat{\ba}_q^{n+1}$ to get $\bU_q\ba_q^{n+1}$ which is closer to the original definition of $q$ and then we obtain $\ba_q^{n+1}$ by $\bU_q^T\bU_q\ba_q^{n+1}$. Since $\bU_q^T\bU_q=\bI_r$, the relaxation strategy in Step 2 is equivalent to relaxing the numerical solution $\hat{\ba}_q^{n+1}$ via 
$$
\ba_q^{n+1}=\xi_0\hat{\ba}_q^{n+1}+(1-\xi_0)\bU_q^Th(\bU_\phi\ba_\phi^{n+1}).
$$
\end{rem}

\begin{rem}[Optimal choice for $\xi_0$] We emphasize that the optimization problem in \eqref{eq:relaxed-CN-xi} can be simplified as the following algebraic optimization problem
\begin{equation*}
\xi_0=\min_{\xi\in[0,1],~a\xi^2+b\xi+c\leq0}\xi,
\end{equation*}
where the coefficients are given by
\begin{equation*}
\begin{aligned}
& a= \frac{1}{2}\Big(\bU_q\hat{\ba}_q^{n+1}-h(\bU_\phi\ba_\phi^{n+1}),~\bU_q\hat{\ba}_q^{n+1}-h(\bU_\phi\ba_\phi^{n+1})\Big), \\
& b=\Big(h(\bU_\phi\ba_\phi^{n+1}),~\bU_q\hat{\ba}_q^{n+1}-h(\bU_\phi\ba_\phi^{n+1})\Big),\\
& c = \frac{1}{2}\Big(h(\bU_\phi\ba_\phi^{n+1}),~h(\bU_\phi\ba_\phi^{n+1})\Big)-\frac{1}{2}\Big(\bU_q\hat{\ba}_q^{n+1},~\bU_q\hat{\ba}_q^{n+1}\Big)-\delta t\eta\Big(\cL_h[\bU\hat{\ba}^{n+\frac{1}{2}}],~\cN_h(\bU\overline{\ba}^{n+\frac{1}{2}})\cL_h[\bU\hat{\ba}^{n+\frac{1}{2}}]\Big).
\end{aligned}
\end{equation*}
We emphasize that the solution set is nonempty since $\xi=1$ is in the feasible domain. Also, notice that $a+b+c>0$ and $\delta t\eta\Big(\cL_h[\bU\hat{\ba}^{n+\frac{1}{2}}],~\cN_h(\bU\overline{\ba}^{n+\frac{1}{2}})\cL_h[\bU\hat{\ba}^{n+\frac{1}{2}}]\Big)>0$. With $a>0$, the optimization problem in \eqref{eq:relaxed-CN-xi} can be solved as
\[\xi_0=\max\{0,\frac{-b-\sqrt{b^2-4ac}}{2a}\}.\]
\end{rem}

\begin{thm}
The Scheme \ref{sch:Relaxed-POD-ROM-CN} is unconditionally energy stable.
\end{thm}
\begin{proof}
According to the Theorem \ref{thm:POD-ROM-CN}, the first step of Scheme \ref{sch:Relaxed-POD-ROM-CN} gives us the following energy dissipation law
\begin{equation*}
\begin{aligned}
&\frac{1}{2} \Big[\Big(\bU_\phi\ba_\phi^{n+1},~\cL_{0,h} \bU_\phi \ba_\phi^{n+1} \Big) + \Big(\bU_q\hat{\ba}_q^{n+1},~\bU_q\hat{\ba}_q^{n+1}\Big)\Big] - \frac{1}{2} \Big[\Big(\bU_\phi\ba_\phi^{n},~\cL_{0,h} \bU_\phi \ba_\phi^{n} \Big) + \Big(\bU_q\ba_q^{n},~\bU_q\ba_q^{n}\Big)\Big] \\
&=-\delta t \Big( \cL_h( \bU \hat{\ba}^{n+\frac{1}{2}}) \Big)^T  \cN_h( \bU \overline{\ba}^{n+\frac{1}{2}}) \Big( \cL_h [ \bU \hat{\ba}^{n+\frac{1}{2}}] \Big).
\end{aligned}
\end{equation*}
Also, from \eqref{eq:relaxed-CN-xi}, we know
\begin{equation*}
\frac{1}{2}(\bU_q\ba_q^{n+1},~\bU_q\ba_q^{n+1})-\frac{1}{2}(\bU_q\hat{\ba}_q^{n+1},~\bU_q\hat{\ba}_q^{n+1})\leq\delta t\eta\Big(\cL_h[\bU\hat{\ba}^{n+\frac{1}{2}}],~\cN_h(\bU\overline{\ba}^{n+\frac{1}{2}})\cL_h[\bU\hat{\ba}^{n+\frac{1}{2}}]\Big).
\end{equation*}
Adding the above two equations together gives us the following energy dissipation law
\beq
\begin{aligned}
&\frac{1}{2} \Big[\Big(\bU_\phi\ba_\phi^{n+1},~\cL_{0,h} \bU_\phi \ba_\phi^{n+1} \Big) + \Big(\bU_q\ba_q^{n+1},~\bU_q\ba_q^{n+1}\Big)\Big] - \frac{1}{2} \Big[\Big(\bU_\phi\ba_\phi^{n},~\cL_{0,h} \bU_\phi \ba_\phi^{n} \Big) + \Big(\bU_q\ba_q^{n},~\bU_q\ba_q^{n}\Big)\Big] \\
&\leq-\delta t(1-\eta) \Big( \cL_h( \bU \hat{\ba}^{n+\frac{1}{2}}) \Big)^T  \cN_h( \bU \overline{\ba}^{n+\frac{1}{2}}) \Big( \cL_h [ \bU \hat{\ba}^{n+\frac{1}{2}}] \Big)\leq0.
\end{aligned}
\eeq
\end{proof}
 
\begin{scheme}[Relaxed BDF2 scheme for the POD-ROM-II] \label{sch:Relaxed-POD-ROM-BDF2}
After we calculated $\ba^{n-1}$ and $\ba^n$, we can obtain $\ba^{n+1}$ via the following three steps:
\begin{itemize}
\item Step 1. Obtain $\hat{\ba}^{n+1}:=\begin{bmatrix} \ba_\phi^{n+1}\\ \hat{\ba}_q^{n+1}\end{bmatrix}$ via the linear scheme
\beq \label{eq:Relaxed-POD-ROM-BDF2}
\bU^T \cL_h^\ast \bU  \frac{3 \hat{\ba}^{n+1} - 4\ba^n + \ba^{n-1}}{2\delta t} = -\bU^T \cL_h^\ast \cN_h( \bU \overline{\ba}^{n+1})  \cL_h [ \bU \hat{\ba}^{n+1}],
\eeq 
with the notation $\overline{\ba}^{n+1} = 2\ba^n - \ba^{n-1}$.
\item Step 2. Update $\bU_q\ba_q^{n+1}$ via the relaxation strategy
$$
\bU_q\ba_q^{n+1} = \xi_0 \bU_q\hat{\ba}_q^{n+1} + (1-\xi_0) h(\bU_\phi\ba_\phi^{n+1}), \quad \xi_0 \in [0, 1],
$$
where $\xi_0$ is a solution for the optimization problem:
\beq
\label{eq:relaxed-BDF2-xi}
\begin{aligned}
\xi_0=\min_{\xi\in[0,1]}\xi, \quad\text{s.t.}~&\frac{1}{4}\Big((\bU_q\ba_q^{n+1},~\bU_q\ba_q^{n+1})+(2\bU_q\ba_q^{n+1}-\bU_q\ba_q^n,~2\bU_q\ba_q^{n+1}-\bU_q\ba_q^n)\Big)\\
&-\frac{1}{4}\Big((\bU_q\hat{\ba}_q^{n+1},~\bU_q\hat{\ba}_q^{n+1})+(2\bU_q\hat{\ba}_q^{n+1}-\bU_q\ba_q^n,~2\bU_q\hat{\ba}_q^{n+1}-\bU_q\ba_q^n)\Big)\\
&\leq\delta t\eta\Big(\cL_h[\bU\hat{\ba}^{n+1}],~\cN_h(\bU\overline{\ba}^{n+1})\cL_h[\bU\hat{\ba}^{n+1}]\Big),
\end{aligned}
\eeq
with an artificial parameter $\eta\in[0,1]$ that can be assigned.
\item Step 3. Update $\ba^{n+1}$ as  $\ba^{n+1}:=\begin{bmatrix} \ba_\phi^{n+1}\\ \ba_q^{n+1}\end{bmatrix}$.
\end{itemize}

\end{scheme}

\begin{rem}[Optimal choice for $\xi_0$] In a similar manner, the optimization problem in \eqref{eq:relaxed-BDF2-xi} can be simplified as the following algebraic optimization problem
\begin{equation*}
\xi_0=\min_{\xi\in[0,1],~a\xi^2+b\xi+c\leq0}\xi,
\end{equation*}
where the coefficients are given by
\begin{equation*}
\begin{aligned}
a= &\frac{5}{4}\Big(\bU_q\hat{\ba}_q^{n+1}-h(\bU_\phi\ba_\phi^{n+1}),~\bU_q\hat{\ba}_q^{n+1}-h(\bU_\phi\ba_\phi^{n+1})\Big), \\
b=& \frac{1}{2}\Big(\bU_q\hat{\ba}_q^{n+1}-h(\bU_\phi\ba_\phi^{n+1}),~5h(\bU_\phi\ba_\phi^{n+1})-2\bU_q a_q^n\Big),\\
c = &\frac{1}{4}\Big[\Big(h(\bU_\phi\ba_\phi^{n+1}),~h(\bU_\phi\ba_\phi^{n+1})\Big)+\Big(2h(\bU_\phi\ba_\phi^{n+1})-\bU_q\ba_q^n,~2h(\bU_\phi\ba_\phi^{n+1})-\bU_q\ba_q^n\Big)\\
&-\Big(\bU_q\hat{\ba}_q^{n+1},~\bU_q\hat{\ba}_q^{n+1}\Big)-\Big(2\bU_q\hat{\ba}_q^{n+1}-\bU_q\ba_q^n,~2\bU_q\hat{\ba}_q^{n+1}-\bU_q\ba_q^n\Big)\Big]\\
&-\delta t\eta\Big(\cL_h[\bU\hat{\ba}^{n+1}],~\cN_h(\bU\overline{\ba}^{n+1})\cL_h[\bU\hat{\ba}^{n+1}]\Big).
\end{aligned}
\end{equation*}
We emphasize that the solution set is nonempty since $\xi=1$ is in the feasible domain. Also, notice that $a+b+c>0$ and $\delta t\eta\Big(\cL_h[\bU\hat{\ba}^{n+1}],~\cN_h(\bU\overline{\ba}^{n+1})\cL_h[\bU\hat{\ba}^{n+1}]\Big)>0$. With $a>0$, the optimization problem in \eqref{eq:relaxed-BDF2-xi} can be solved as
\[\xi_0=\max\{0,\frac{-b-\sqrt{b^2-4ac}}{2a}\}.\]
\end{rem}

\begin{thm}
The Scheme \ref{sch:Relaxed-POD-ROM-BDF2} is unconditionally energy stable.
\end{thm}
\begin{proof}
According to the Theorem \ref{thm:POD-ROM-BDF2}, the first step of Scheme \ref{sch:Relaxed-POD-ROM-BDF2} gives us the following energy dissipation law
\begin{equation*}
\begin{aligned}
&\frac{1}{4} \Big[\Big(\bU_\phi\ba_\phi^{n+1},~\cL_{0,h} \bU_\phi \ba_\phi^{n+1} \Big) + \Big(2\bU_\phi\ba_\phi^{n+1}-\bU_\phi\ba_q^n,~\cL_{0,h}(2\bU_\phi\ba_\phi^{n+1}-\bU_\phi\ba_\phi^n)\Big)\\
&+\Big(\bU_q\hat{\ba}_q^{n+1},~\bU_q\hat{\ba}_q^{n+1}\Big) + \Big(2\bU_q\hat{\ba}_q^{n+1}-\bU_q\ba_q^n,~2\bU_q\hat{\ba}_q^{n+1}-\bU_q\ba_q^n\Big)\Big] \\
&-\frac{1}{4} \Big[\Big(\bU_\phi\ba_\phi^n,~\cL_{0,h} \bU_\phi \ba_\phi^n \Big) + \Big(2\bU_\phi\ba_\phi^n-\bU_\phi\ba_q^{n-1},~\cL_{0,h}(2\bU_\phi\ba_\phi^n-\bU_\phi\ba_\phi^{n-1})\Big)\\
&+\Big(\bU_q\ba_q^n,~\bU_q\ba_q^n\Big) + \Big(2\bU_q\ba_q^n-\bU_q\ba_q^{n-1},~2\bU_q\ba_q^n-\bU_q\ba_q^{n-1}\Big)\Big] \\
&\leq-\delta t \Big( \cL_h( \bU \hat{\ba}^{n+1}) \Big)^T  \cN_h( \bU \overline{\ba}^{n+1}) \Big( \cL_h [ \bU \hat{\ba}^{n+1}] \Big).
\end{aligned}
\end{equation*}
Also, from \eqref{eq:relaxed-BDF2-xi}, we know
\begin{multline*}
\frac{1}{4}\Big((\bU_q\ba_q^{n+1},~\bU_q\ba_q^{n+1})+(2\bU_q\ba_q^{n+1}-\bU_q\ba_q^n,~2\bU_q\ba_q^{n+1}-\bU_q\ba_q^n)\Big) \\
- \frac{1}{4}\Big((\bU_q\hat{\ba}_q^{n+1},~\bU_q\hat{\ba}_q^{n+1})+(2\bU_q\hat{\ba}_q^{n+1}-\bU_q\ba_q^n,~2\bU_q\hat{\ba}_q^{n+1}-\bU_q\ba_q^n)\Big)\\
\leq\delta t\eta\Big(\cL_h[\bU\hat{\ba}^{n+1}],~\cN_h(\bU\overline{\ba}^{n+1})\cL_h[\bU\hat{\ba}^{n+1}]\Big).  
\end{multline*}

Adding the above two equations together gives us the following energy dissipation law
\beq
\begin{aligned}
&\frac{1}{4} \Big[\Big(\bU_\phi\ba_\phi^{n+1},~\cL_{0,h} \bU_\phi \ba_\phi^{n+1} \Big) + \Big(2\bU_\phi\ba_\phi^{n+1}-\bU_\phi\ba_q^n,~\cL_{0,h}(2\bU_\phi\ba_\phi^{n+1}-\bU_\phi\ba_\phi^n)\Big)\\
&+\Big(\bU_q\ba_q^{n+1},~\bU_q\ba_q^{n+1}\Big) + \Big(2\bU_q\ba_q^{n+1}-\bU_q\ba_q^n,~2\bU_q\ba_q^{n+1}-\bU_q\ba_q^n\Big)\Big] \\
&-\frac{1}{4} \Big[\Big(\bU_\phi\ba_\phi^n,~\cL_{0,h} \bU_\phi \ba_\phi^n \Big) + \Big(2\bU_\phi\ba_\phi^n-\bU_\phi\ba_q^{n-1},~\cL_{0,h}(2\bU_\phi\ba_\phi^n-\bU_\phi\ba_\phi^{n-1})\Big)\\
&+\Big(\bU_q\ba_q^n,~\bU_q\ba_q^n\Big) + \Big(2\bU_q\ba_q^n-\bU_q\ba_q^{n-1},~2\bU_q\ba_q^n-\bU_q\ba_q^{n-1}\Big)\Big] \\
&\leq-\delta t(1-\eta) \Big( \cL_h( \bU \hat{\ba}^{n+1}) \Big)^T  \cN_h( \bU \overline{\ba}^{n+1}) \Big( \cL_h [ \bU \hat{\ba}^{n+1}] \Big).
\end{aligned}
\eeq
\end{proof}

Similar relaxation techniques can also be applied to the numerical schemes for POD-ROM-I in Scheme \ref{scheme:CN-ApproachI} and Scheme \ref{scheme:BDF2-ApproachI}. For instance, we can have the following relaxed CN scheme.
\begin{scheme}[Relaxed CN scheme for the POD-ROM-I] \label{sch:Relaxed-POD-ROM-CN-I}
After we calculated $\ba^{n-1}$ and $\ba^n$, we can obtain $\ba^{n+1}$ via the following three steps:
\begin{itemize}
\item Step 1. Obtain $\hat{\ba}^{n+1}:=\begin{bmatrix} \ba_\phi^{n+1}\\\hat{\ba}_q^{n+1}\end{bmatrix}$ via the linear scheme
\beq \label{eq:Relaxed-POD-ROM-CN-I}
\frac{\hat{\ba}^{n+1} - \ba^n}{\delta t} = -\bU^T \cN_h( \bU \overline{\ba}^{n+\frac{1}{2}}) \bU \bU^T  \cL_h [ \bU \hat{\ba}^{n+\frac{1}{2}}],
\eeq 
with the notations $\overline{\ba}^{n+\frac{1}{2}} = \frac{3}{2}\ba^{n} - \frac{1}{2}\ba^{n-1}$ and $\hat{\ba}^{n+\frac{1}{2}} = \frac{1}{2} \hat{\ba}^{n+1} + \frac{1}{2}\ba^n$.
\item Step 2. Update $\bU_q\ba_q^{n+1}$ via the relaxation strategy 
$$
\bU_q\ba_q^{n+1} = \xi_0 \bU_q\hat{\ba}_q^{n+1} + (1-\xi_0) h(\bU_\phi\ba_\phi^{n+1}), \quad \xi_0 \in [0, 1],
$$
where $\xi_0$ is a solution for the optimization problem:
\beq
\label{eq:relaxed-CN-xi-I}
\begin{aligned}
\xi_0=\min_{\xi\in[0,1]}\xi, \quad\text{s.t.}~&\frac{1}{2}(\bU_q\ba_q^{n+1},~\bU_q\ba_q^{n+1})-\frac{1}{2}(\bU_q\hat{\ba}_q^{n+1},~\bU_q\hat{\ba}_q^{n+1})\\
&\leq\delta t\eta\Big( \bU \bU^T \cL_h[\bU\hat{\ba}^{n+\frac{1}{2}}],~\cN_h(\bU\overline{\ba}^{n+\frac{1}{2}}) \bU \bU^T\cL_h[\bU\hat{\ba}^{n+\frac{1}{2}}]\Big),
\end{aligned}
\eeq
with an artificial parameter $\eta\in[0,1]$ that can be assigned.
\item Step 3. Update $\ba^{n+1}$ as  $\ba^{n+1}:=\begin{bmatrix} \ba_\phi^{n+1}\\ \ba_q^{n+1}\end{bmatrix}$.
\end{itemize}
\end{scheme}
The proof for the energy stability is similarly with the one for Scheme \ref{sch:Relaxed-POD-ROM-CN}.

\subsection{Discrete empirical interpolation method for nonlinear terms}
For the nonlinear term, we can utilize the discrete empirical interpolation method (DEIM) \cite{DEIM} to reduce the computational costs. Mainly, as we can tell from \eqref{eq:POD-ROM-II} and \eqref{eq:nonlinear-terms}, the only nonlinearity comes from the discrete mobility operator $\cN_h(\bU \ba(t))$, which is the term $g(\bU_\phi\ba_\phi(t))$. When a general nonlinearity is present, the cost to evaluate the nonlinear terms still depends on the dimension of the original system since $\bU_\phi \ba_\phi(t)\in\R^n$. The DEIM approach proposed in \cite{DEIM} developed an interpolation-based projection to approximate the nonlinear term where the interpolation indices are selected to limit the growth of an error bound. With the data \eqref{eq:sampling-phi}, we can have
\beq
\mathbf{N} = \begin{bmatrix}
g(\phi_1) & g(\phi_2) & \cdots & g(\phi_{m-1}) & g(\phi_m)
\end{bmatrix}.
\eeq 
The approximation from projecting $g(\bU_\phi\ba_\phi(t))$, denoted as $g(t)$, onto the subspace obtained from the POD approach is of the form
\beq
g(t) \approx \bW \bc(t),
\eeq
where $\bW= \begin{bmatrix} \bw_1&\cdots&\bw_k \end{bmatrix}$
$\in \hR^{n,k}$ with $k\ll n$ consists of the first $k$ columns of the left singular matrix from SVD, and a time-dependent coefficient vector $\bc(t) \in \hR^k$. To determine $\bc(t)$ which is a highly overdetermined system, we construct a basis dependent interpolation matrix, $P=\begin{bmatrix} \be_{\gamma_1}&\cdots&\be_{\gamma_k}\end{bmatrix} \in \hR^{n,k}$ in an algorithmic way given in \cite{DEIM} where $\be_{\gamma_i}\in\hR^n$ is the $\gamma_i$th column of the identity matrix $\bI_n\in\hR^{n,n}$ for $i=1,\cdots,n$. Suppose $P^T \bW \in \hR^{k,k}$ is nonsingular, then the coefficient vector $\bc(t)$ is determined uniquely from
\beq
P^T g(t) = P^T \bW \bc(t).
\eeq 
Finally, $g(t)$ can be approximated by
\beq
g(t) \approx \bW\bc(t)=\bW (P^T \bW)^{-1} P^T g(t).
%\approx  \bU^T \bW (P^T \bW)^{-1}  \cN_h(P^T \bU_\phi \ba_\phi(t))
\eeq 
If $g$ is a component-wise function, we can further have
\beq
g(t) \approx \bW (P^T \bW)^{-1}  g(P^T \bU_\phi \ba_\phi(t)).
\eeq 
Notice the fact $ \bW (P^T \bW)^{-1} \in \hR^{n,k}$ which is precomputed and $P^T \bU_\phi \in \hR^{k,r}$, which are independent of $n$ of the full-order system.  For simplicity, the DEIM is not utilized in the later numerical examples. But we emphasize that the application of DEIM doesn't affect our theoretical results in the previous sections since the nonlinearity in our POD-ROM is dealt explicitly as shown in Scheme \ref{sch:Relaxed-POD-ROM-CN} and Scheme \ref{sch:Relaxed-POD-ROM-BDF2}.

\section{Numerical Examples}
So far, we have provided a unified platform to develop structure-preserving ROMs for thermodynamically consistent PDE models and their structure-preserving numerical approximations. The leading idea is to transform the thermodynamically consistent PDE models into an equivalent form using the energy quadratization approach. The free energy of the transformed system is in the quadratic form of the state variables, which could be further exploited to develop ROMs.

To illustrate the effectiveness of the numerical platform, we present several examples in this section. Specifically, we apply the general framework to several thermodynamically consistent phase field models. Phase field models have numerous applications in various science and engineering fields, along with interdisciplinary applications. The principle of phase field modeling is the concept of using a smooth function to represent the state of the system by introducing an artificial thickness as a transition layer, i.e., this function varies continuously across the interface, allowing for the transition between different states or phases to be represented over a finite width rather than as a sharp interface. Here, we investigate several popular thermodynamically consistent phase field models.

\subsection{Allen-Cahn equation}
 We start with the Allen-Cahn (AC) equation
\begin{subequations}
\begin{align}
& \partial_t \phi = -M( -\varepsilon^2 \Delta  \phi - \phi^3 + \phi), \quad (\bx, t) \in \Omega \times (0, T],\\
& \phi(\bx, 0) = \phi_0(\bx), \quad \bx \in \Omega,
\end{align}
\end{subequations}
with periodic boundary conditions. Here $\Omega$ is a smooth domain, $M >0$ is the mobility constant, and $\varepsilon$ is a parameter to control the interfacial thickness. The Onsager triplet for the AC equation is
$$
(\phi,\quad \cG, \quad \cE) := \Big( \phi, \quad  M, \quad \int_\Omega \Big[ \frac{\varepsilon^2}{2}|\nabla \phi|^2 + \frac{1}{4}(\phi^2 - 1)^2 \Big] d\bx \Big).
$$

If we introduce the auxiliary variables
$$
q := \frac{\sqrt{2}}{2}(\phi^2-1 -\gamma_0), \quad g(\phi):=\frac{\partial q}{\partial \phi} = \sqrt{2}\phi,
$$
we can rewrite the equation as
\begin{subequations}
\begin{align}
& \partial_t \phi = -M \Big(-\varepsilon^2 \Delta \phi + \gamma_0 \phi  + q g(\phi) \Big), \\
& \partial_t q = g(\phi) \partial_t \phi.
\end{align}
\end{subequations}
Denote $\Psi = \begin{bmatrix}
\phi \\ q
\end{bmatrix}$, $\cG_s = M$, $\cL_0 = -\varepsilon^2 \Delta + \gamma_0$, $\cL = \begin{bmatrix}
\cL_0  &  0  \\
0 & 1
\end{bmatrix}$ and $\cN_0 = \begin{bmatrix}
\bI  & g(\phi)
\end{bmatrix}$, such that the equation is written as
\beq
\partial_t \Psi = -\cN(\Psi) \cL  \Psi, \quad \mbox{ where } \quad \cN(\Psi) = \cN_0^T\cG_s\cN_0 .
\eeq 
Hence, the structure-preserving POD-ROMs can be derived from the techniques introduced in previous sections. Now, we consider a specific numerical example.

\textbf{Example 1.}
 We choose the model parameters $M=1$ and  $\varepsilon=0.02$. The domain $\Omega$ is set as $[0, 1]^2$. The initial condition is
 \beq \label{eq:AC-initial-condition}
 \phi_0(x,y) = 2 \Big[ \sum_{i=1}^7 \frac{1}{2} (1 - \tanh \frac{\sqrt{ (x-X(i))^2+(y-Y(i))^2} - R(i)}{\varepsilon}) \Big] - 1,
 \eeq 
 where the parameters are
\begin{subequations}
\begin{align}
&X=\begin{bmatrix} 1/4 & 1/8 & 1/4 & 1/2 & 3/4 & 1/2 & 3/4 \end{bmatrix},  \\
&Y = \begin{bmatrix}  1/4 & 3/8&  5/8&  1/8&  1/8&  1/2& 3/4 \end{bmatrix},  \\
&R = \begin{bmatrix} 1/20& 1/16 &1/12 &1/12& 1/10& 1/8& 1/8  \end{bmatrix}.
\end{align}
\end{subequations}
This represents seven disks of various sizes in different domain locations. Consider $T=15$, we first generate the data using an accurate numerical solver and sample the following data
$
\begin{bmatrix}
\Phi_1 & \Phi_2 & \cdots & \Phi_{m}
\end{bmatrix}
$,
where $\Phi_k \in \bR^{n}$ with $n=N_xN_y$ are the numerical solution at $t=0.1k$ in a vector form. We use $N_x=N_y=128$ and $m=150$ in this example. Then, we follow the POD-ROM numerical framework proposed in previous sections. The numerical parameters are $\delta t = 10^{-3}$ and $\gamma_0=1$. The distribution of the singular values is summarized in Figure \ref{fig:AC-SVD}.

\begin{figure}[H]
\centering
\includegraphics[width=0.85\textwidth]{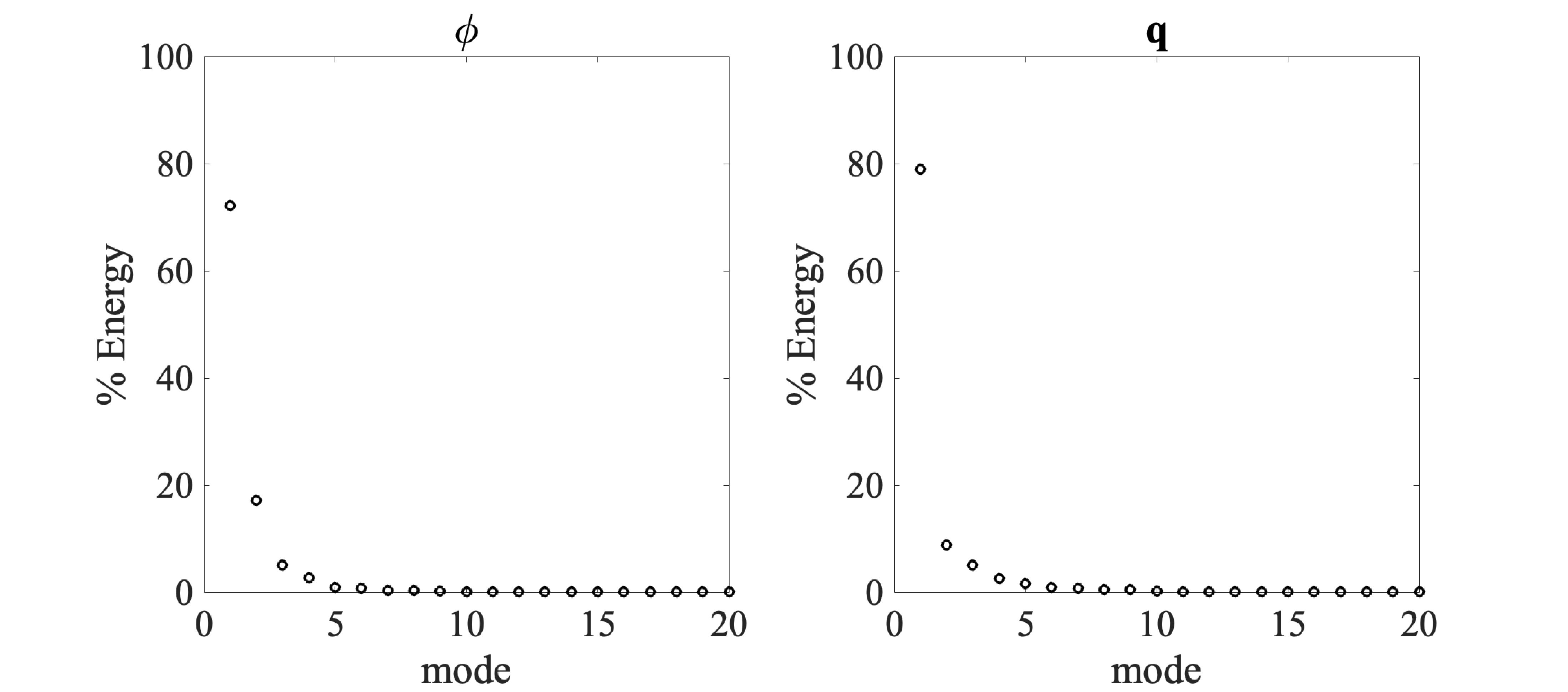}
\caption{Singular value distributions for the data collected from the Allen-Cahn equation.}
\label{fig:AC-SVD}
\end{figure}

The numerical results using the ROM and proposed schemes are summarized in Figure \ref{fig:AC-Example1-Compare}. 
The figure shows that the numerical approximation of the reduced order model, even with just r=4 modes, provides a reasonably good approximation compared to results from the full order model. Moreover, the reduced order model with r=10 modes offers an even more accurate approximation.

\begin{figure}[H]
\centering

\subfigure[Numerical solution from POD-ROM-II with r=1 at $t=1,5,10,14$]{
\includegraphics[width=0.2\textwidth]{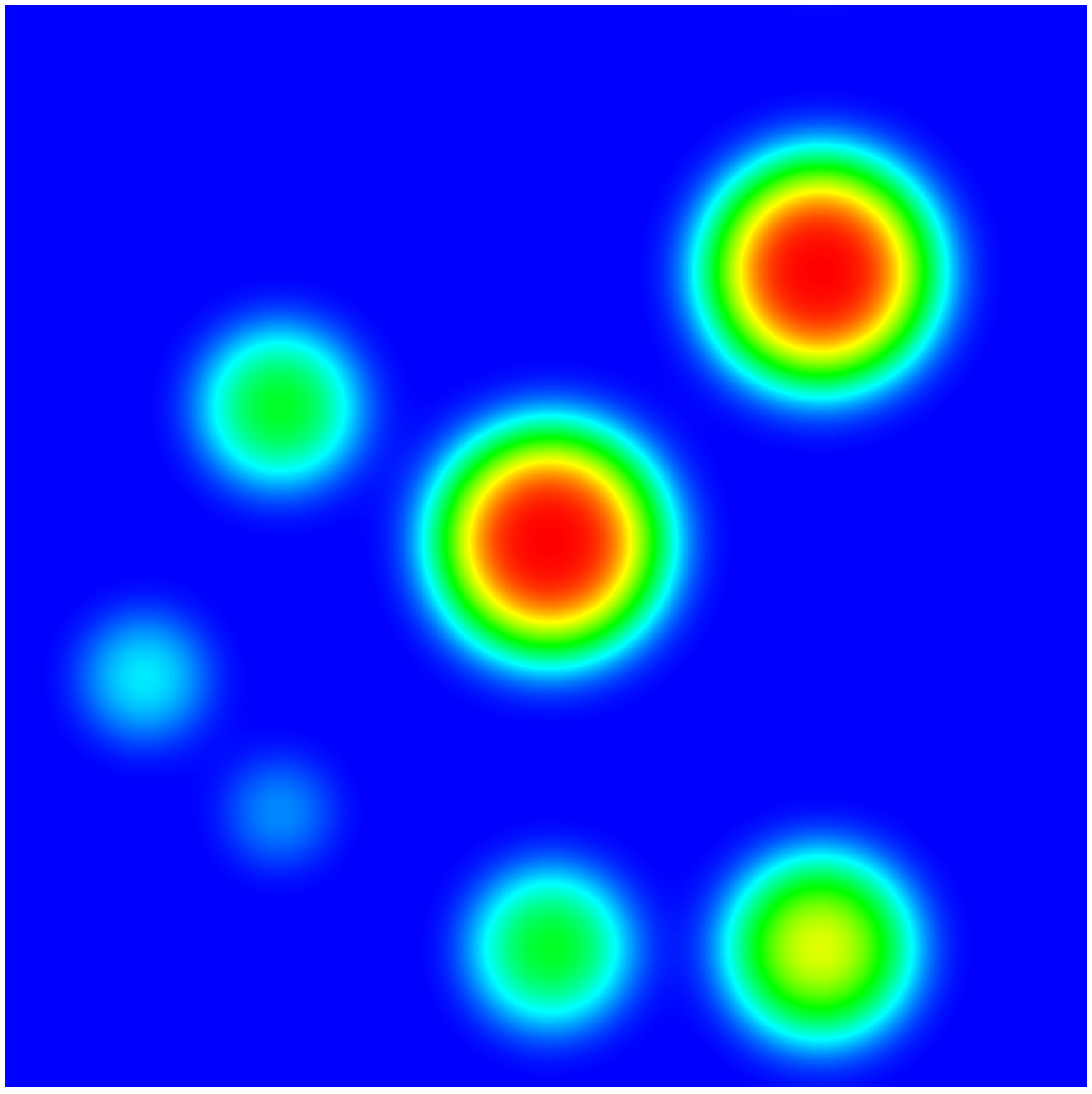}
\includegraphics[width=0.2\textwidth]{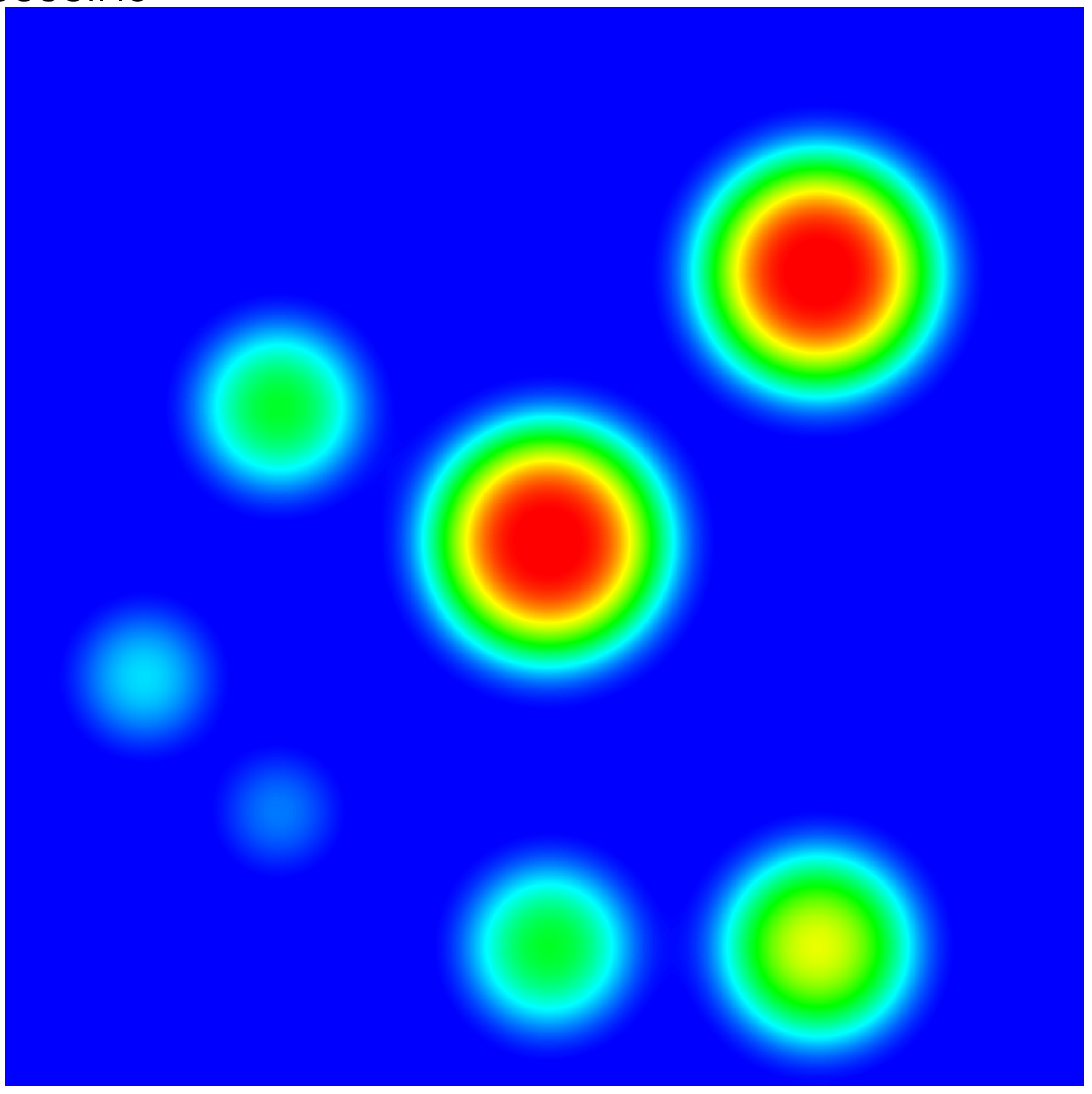}
\includegraphics[width=0.2\textwidth]{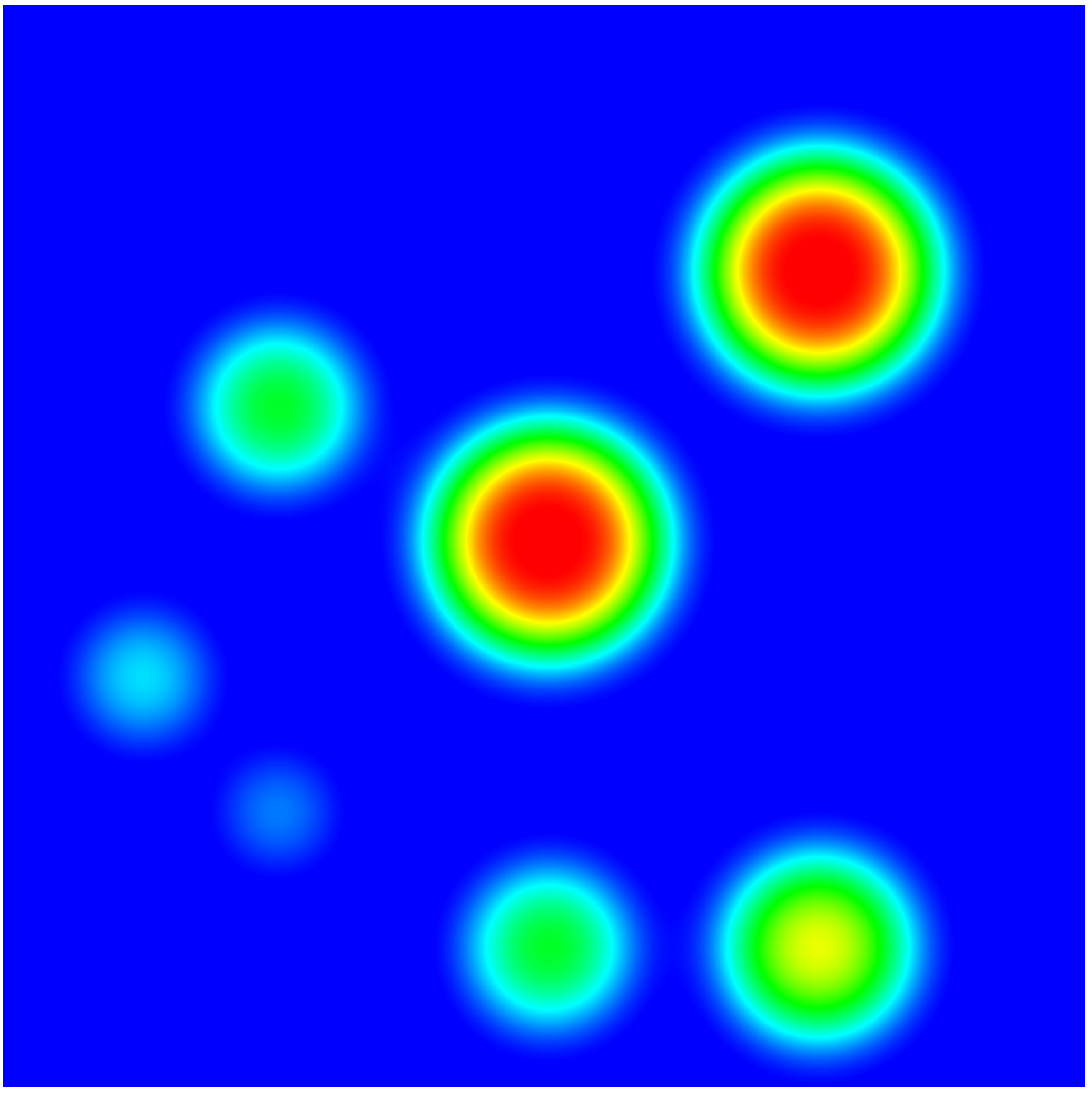}
\includegraphics[width=0.2\textwidth]{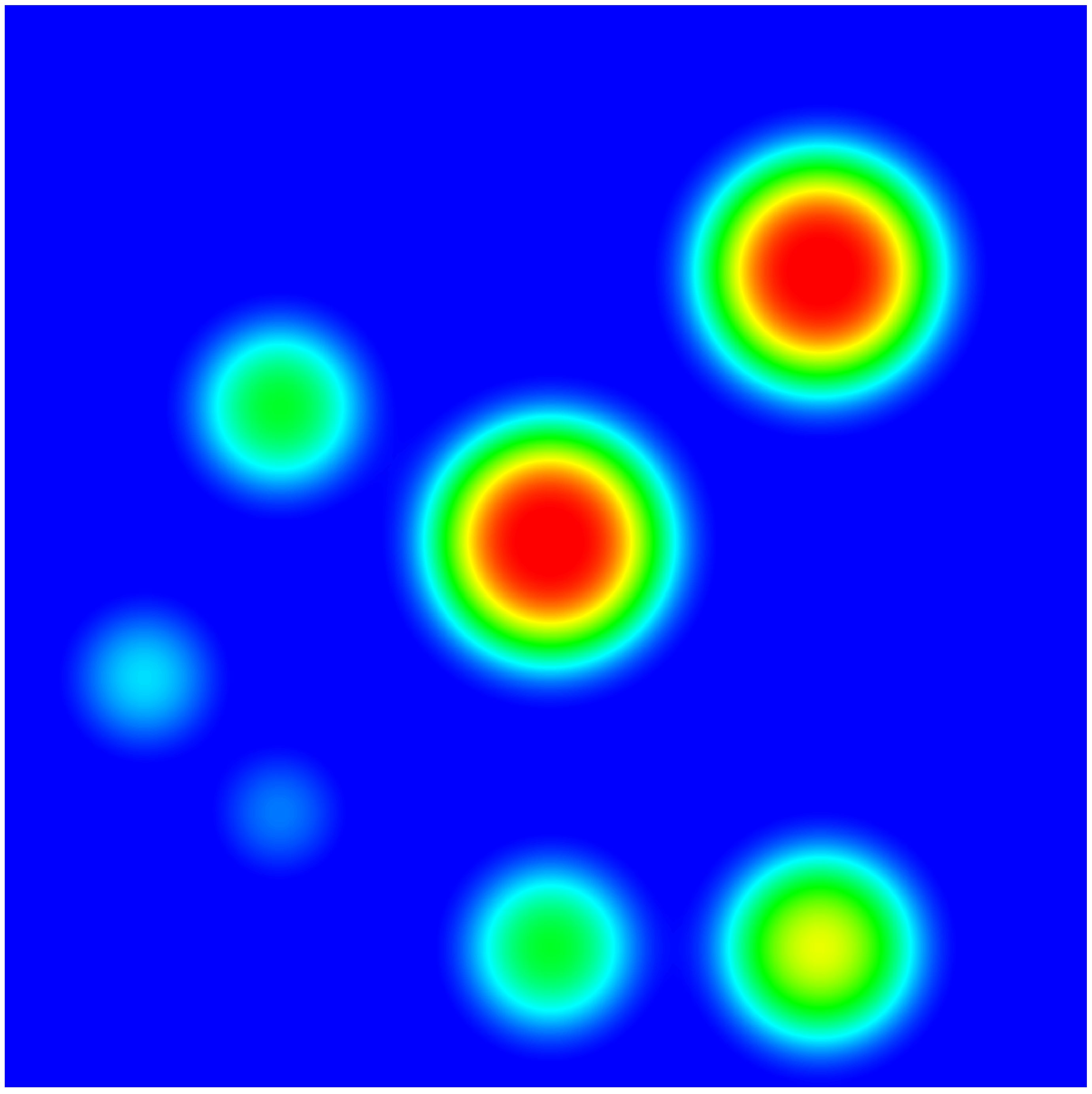}
}

\subfigure[Numerical solution from POD-ROM-II with r=4 at $t=1,5,10,14$]{
\includegraphics[width=0.2\textwidth]{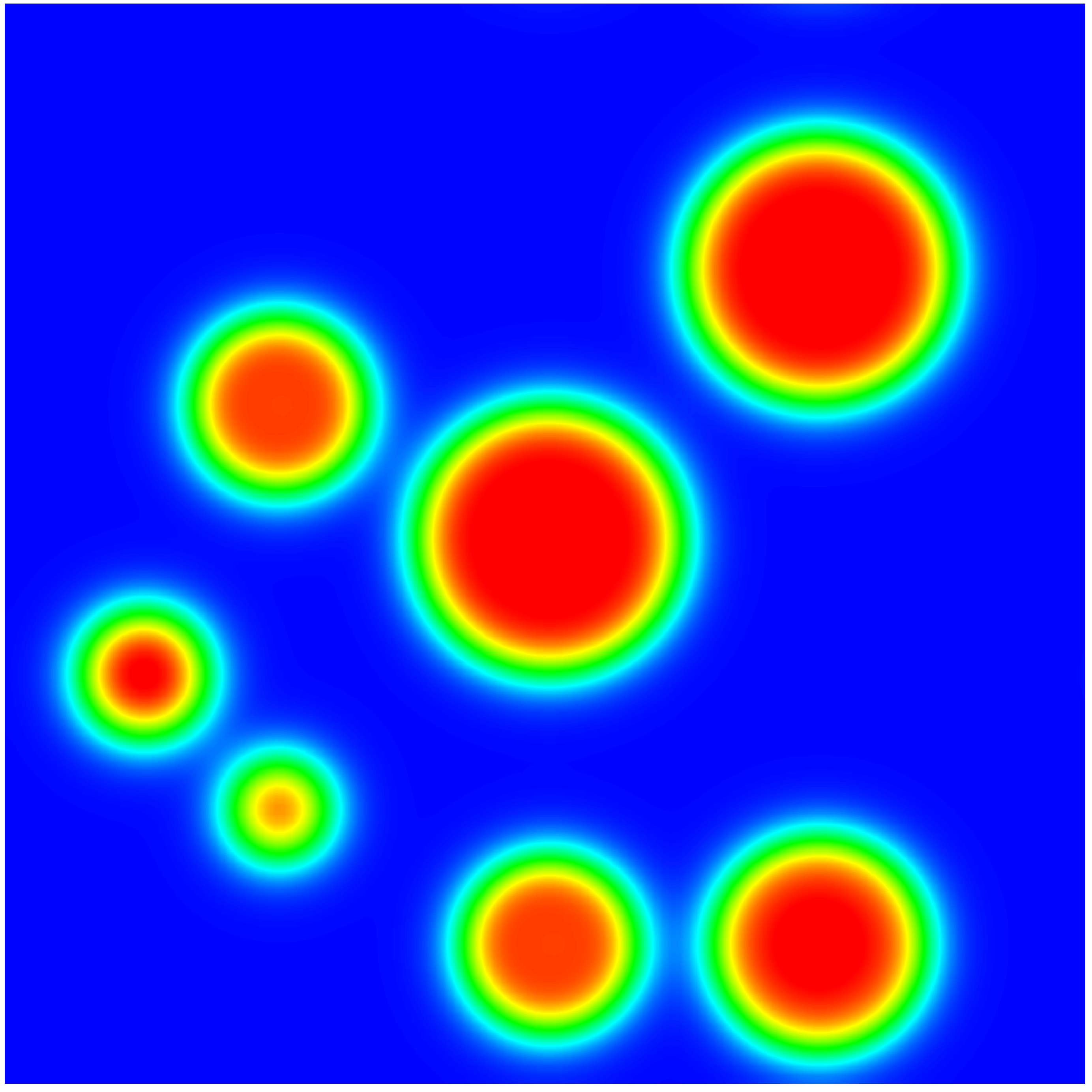}
\includegraphics[width=0.2\textwidth]{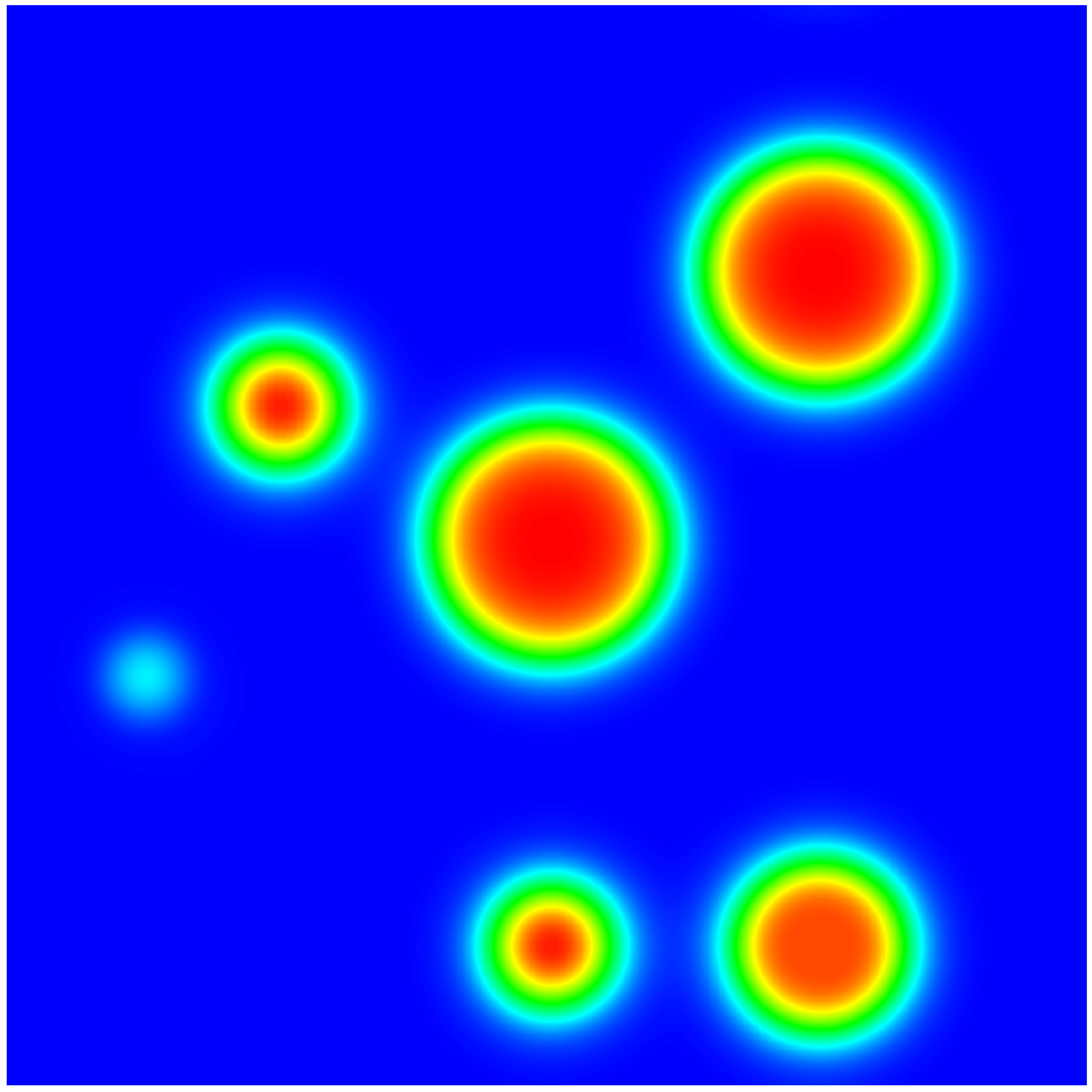}
\includegraphics[width=0.2\textwidth]{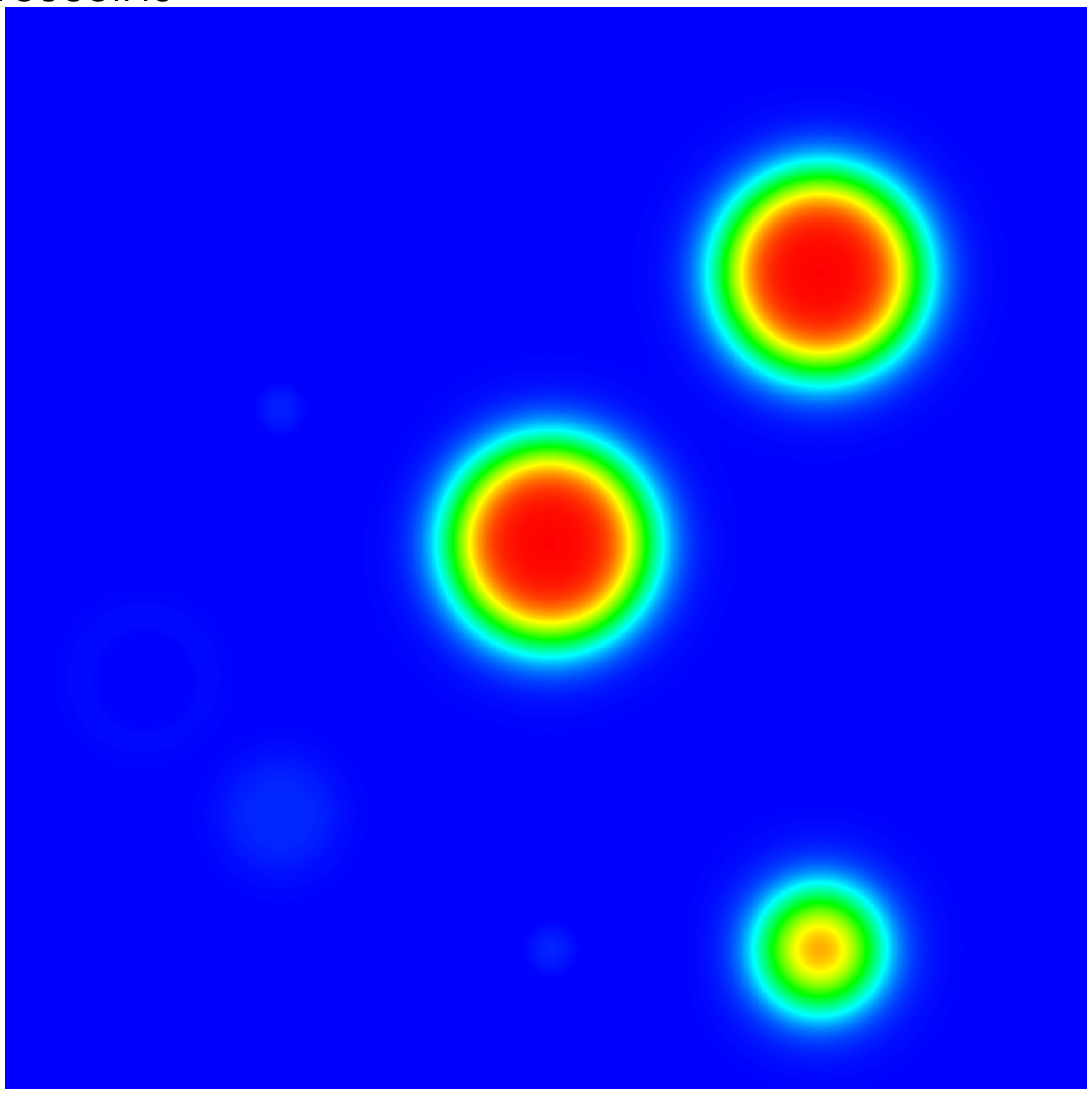}
\includegraphics[width=0.2\textwidth]{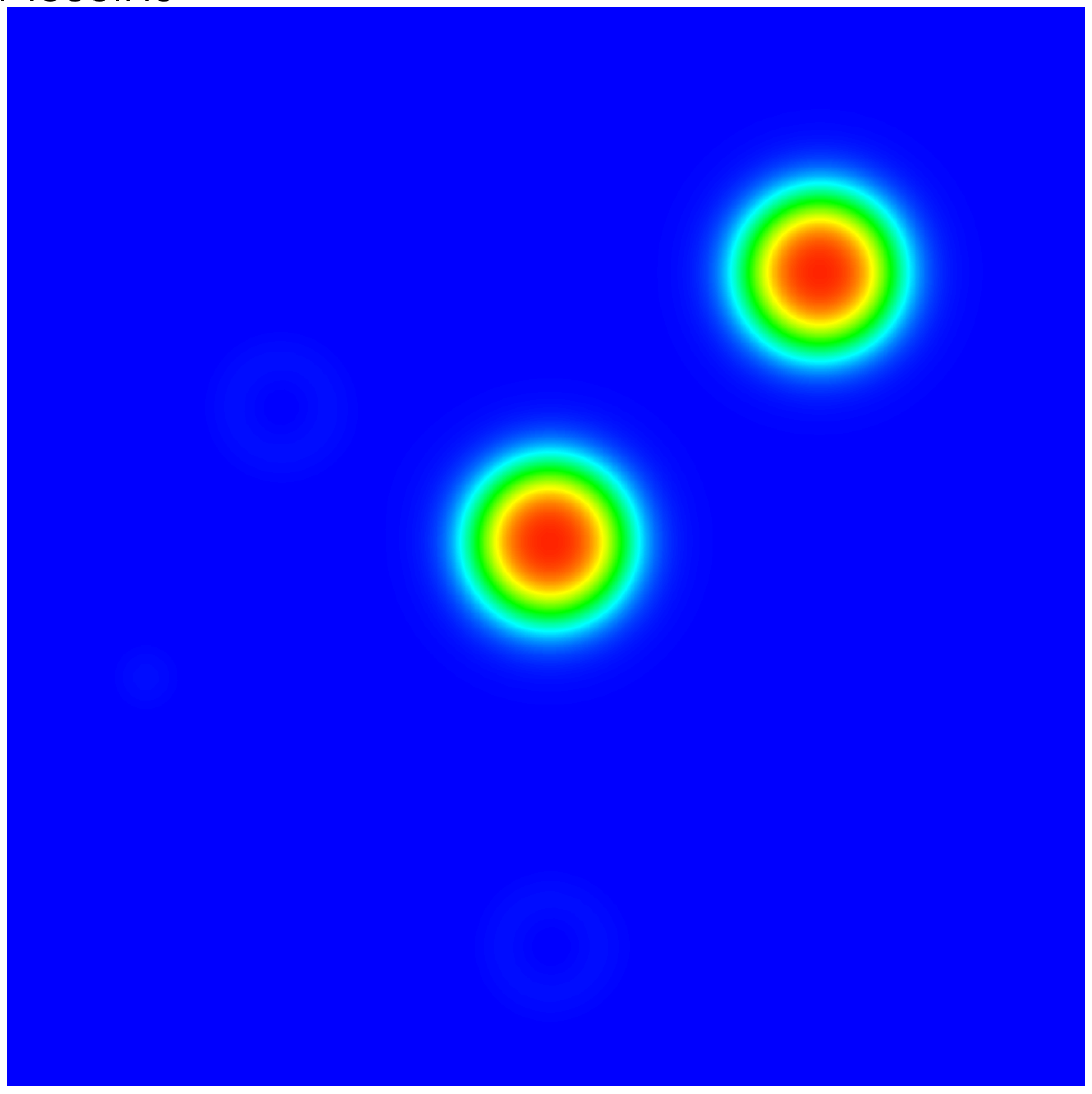}
}

\subfigure[Numerical solution from POD-ROM-II with r=10 at $t=1,5,10,14$]{
\includegraphics[width=0.2\textwidth]{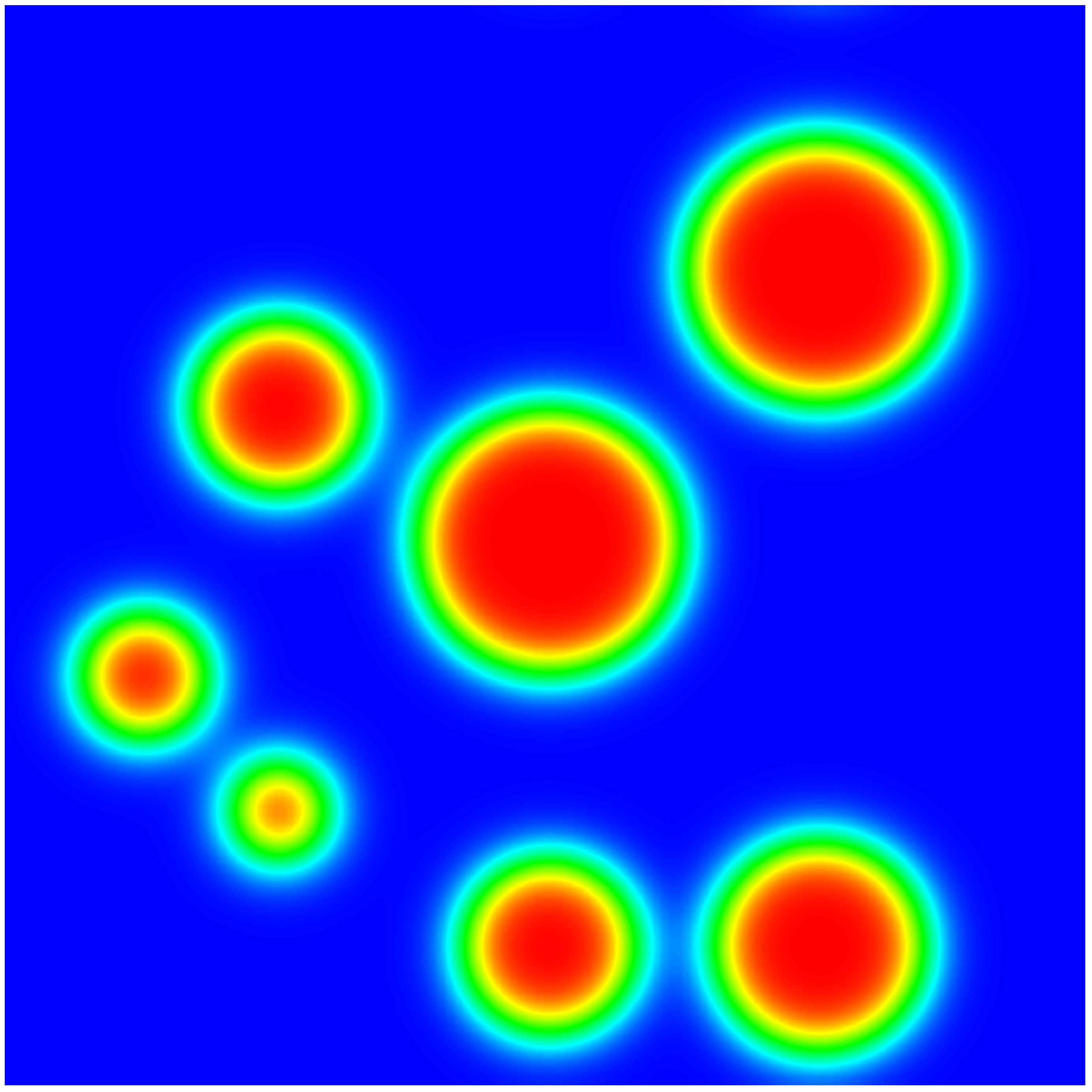}
\includegraphics[width=0.2\textwidth]{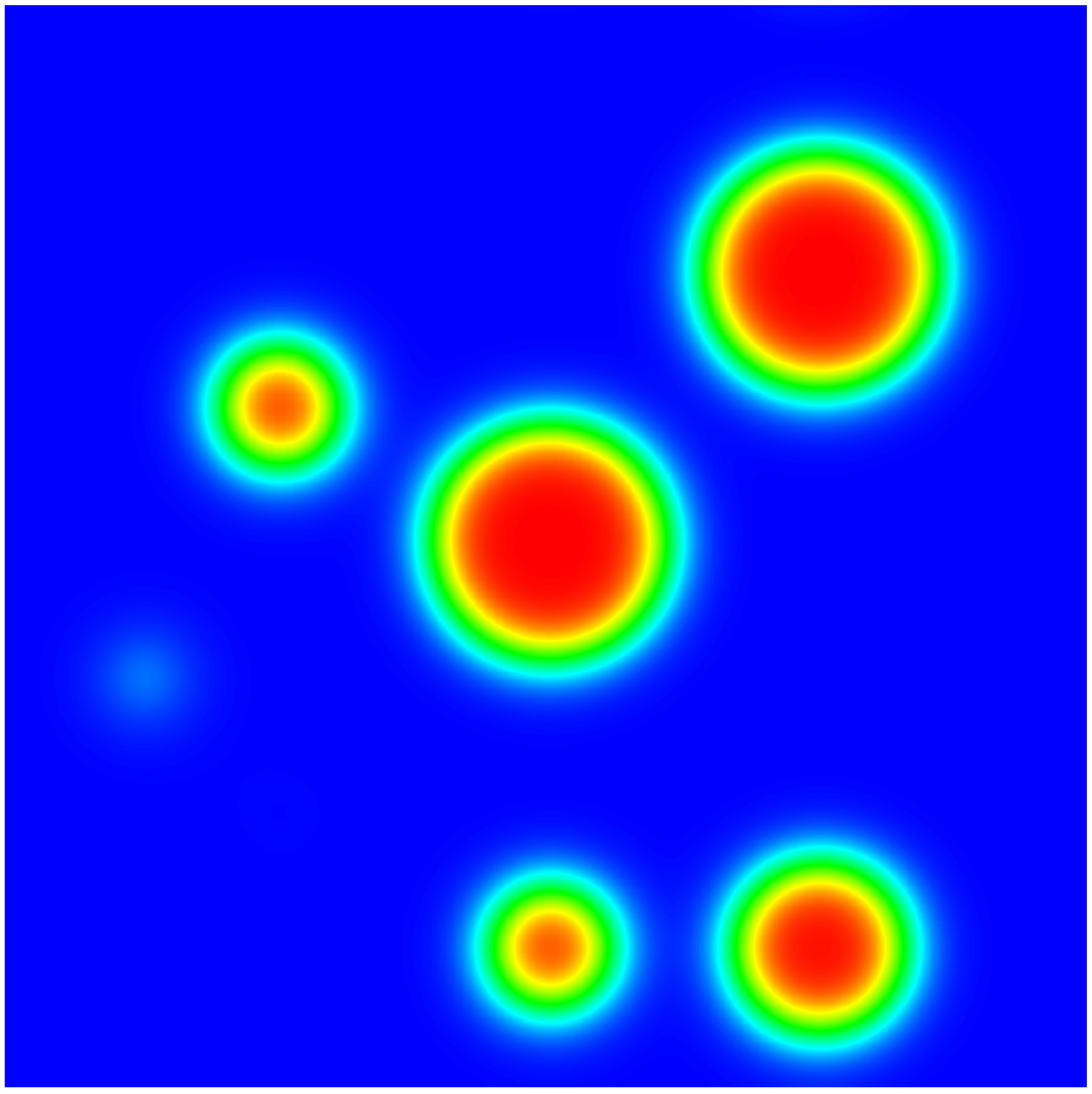}
\includegraphics[width=0.2\textwidth]{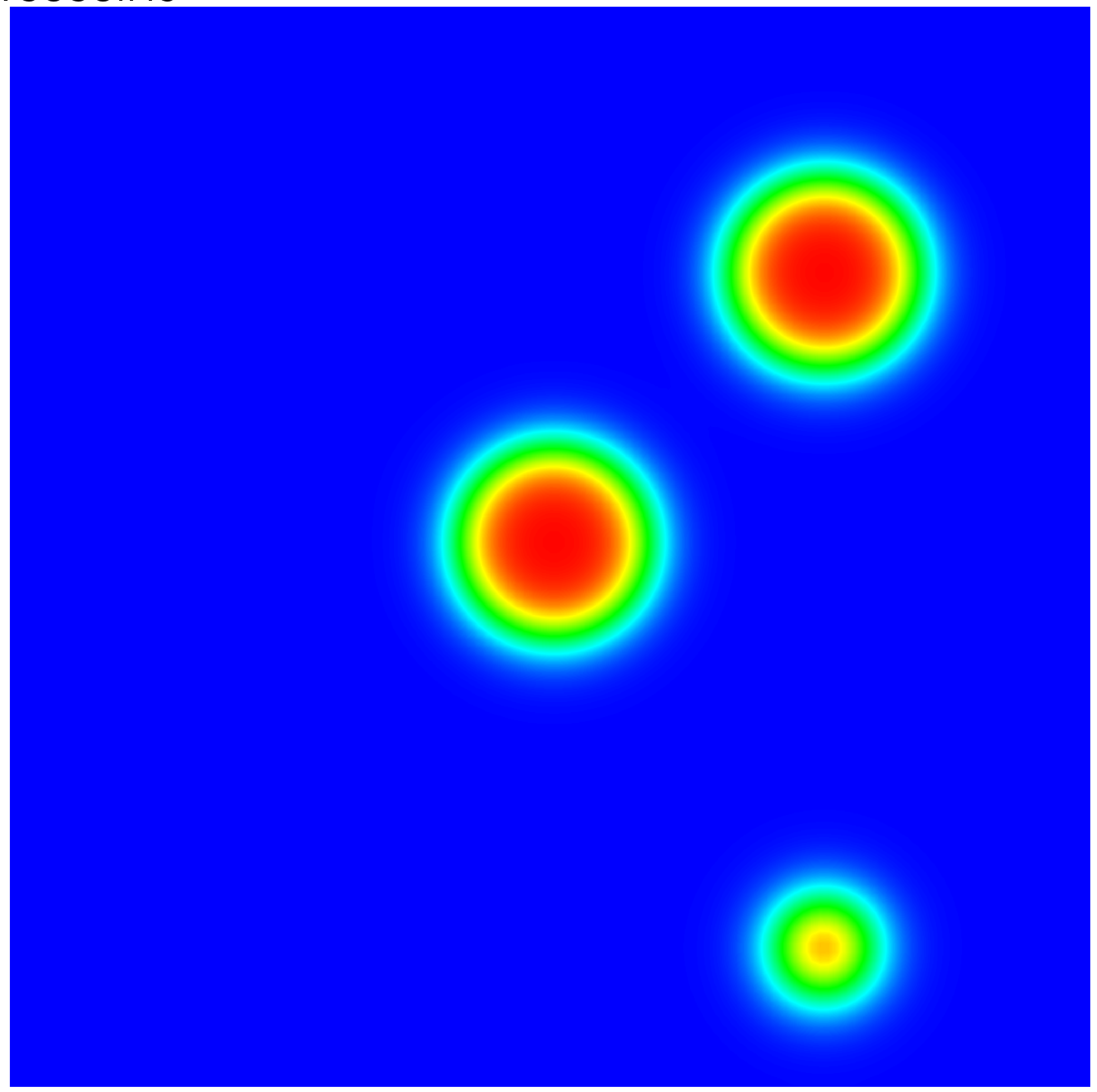}
\includegraphics[width=0.2\textwidth]{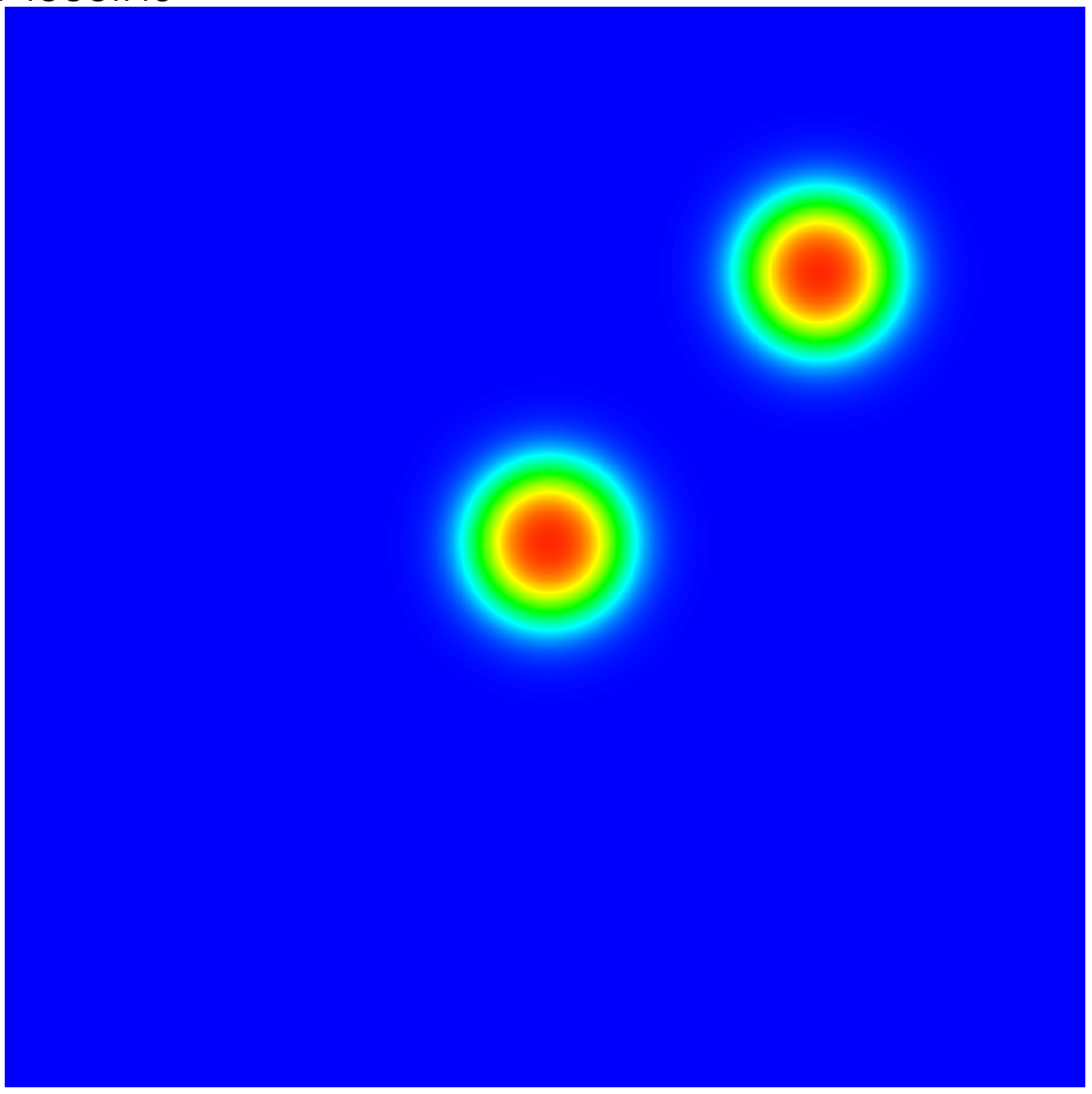}
}

\subfigure[Numerical solution from the full order model at $t=1,5,10,14$]{
\includegraphics[width=0.2\textwidth]{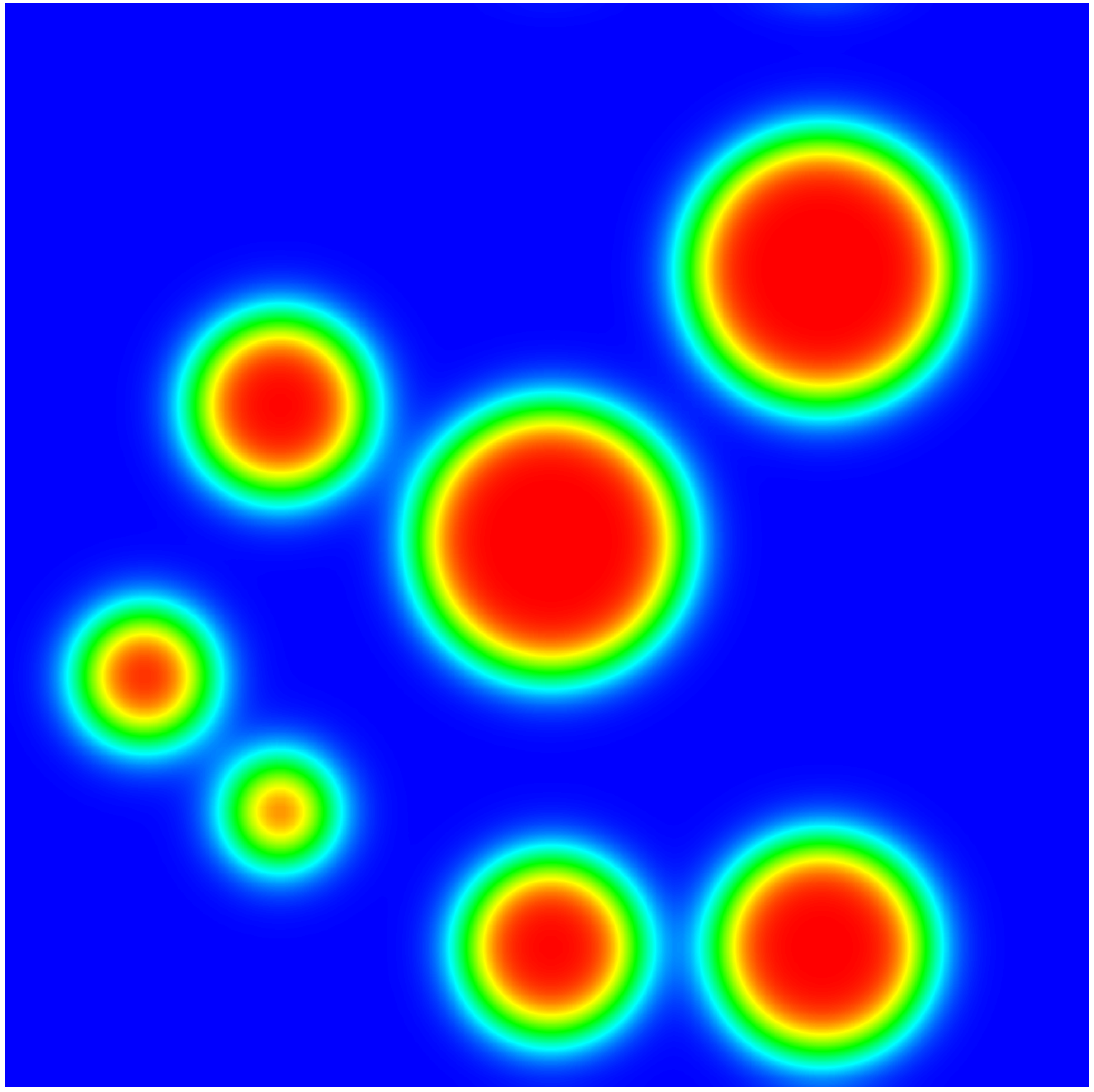}
\includegraphics[width=0.2\textwidth]{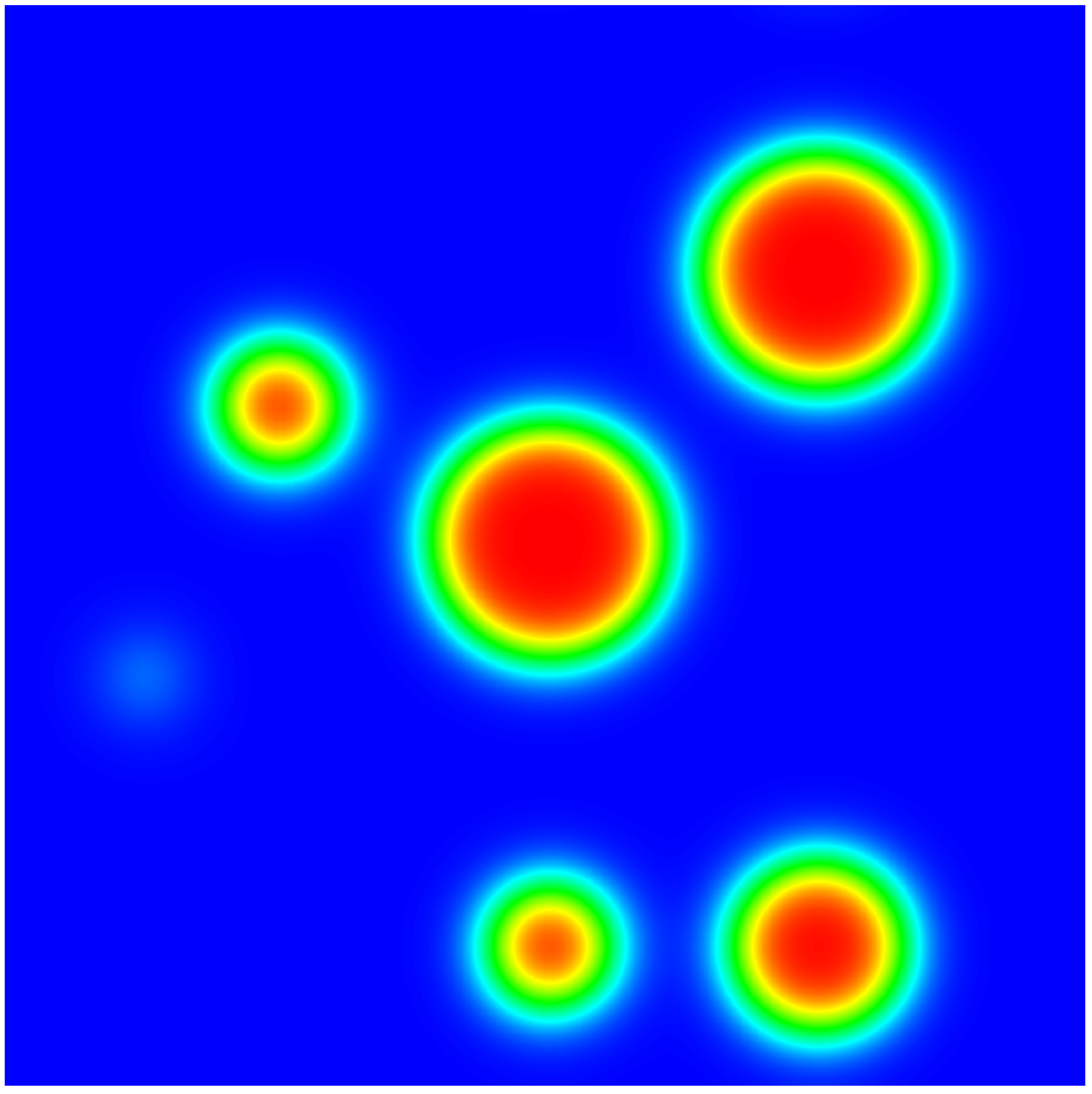}
\includegraphics[width=0.2\textwidth]{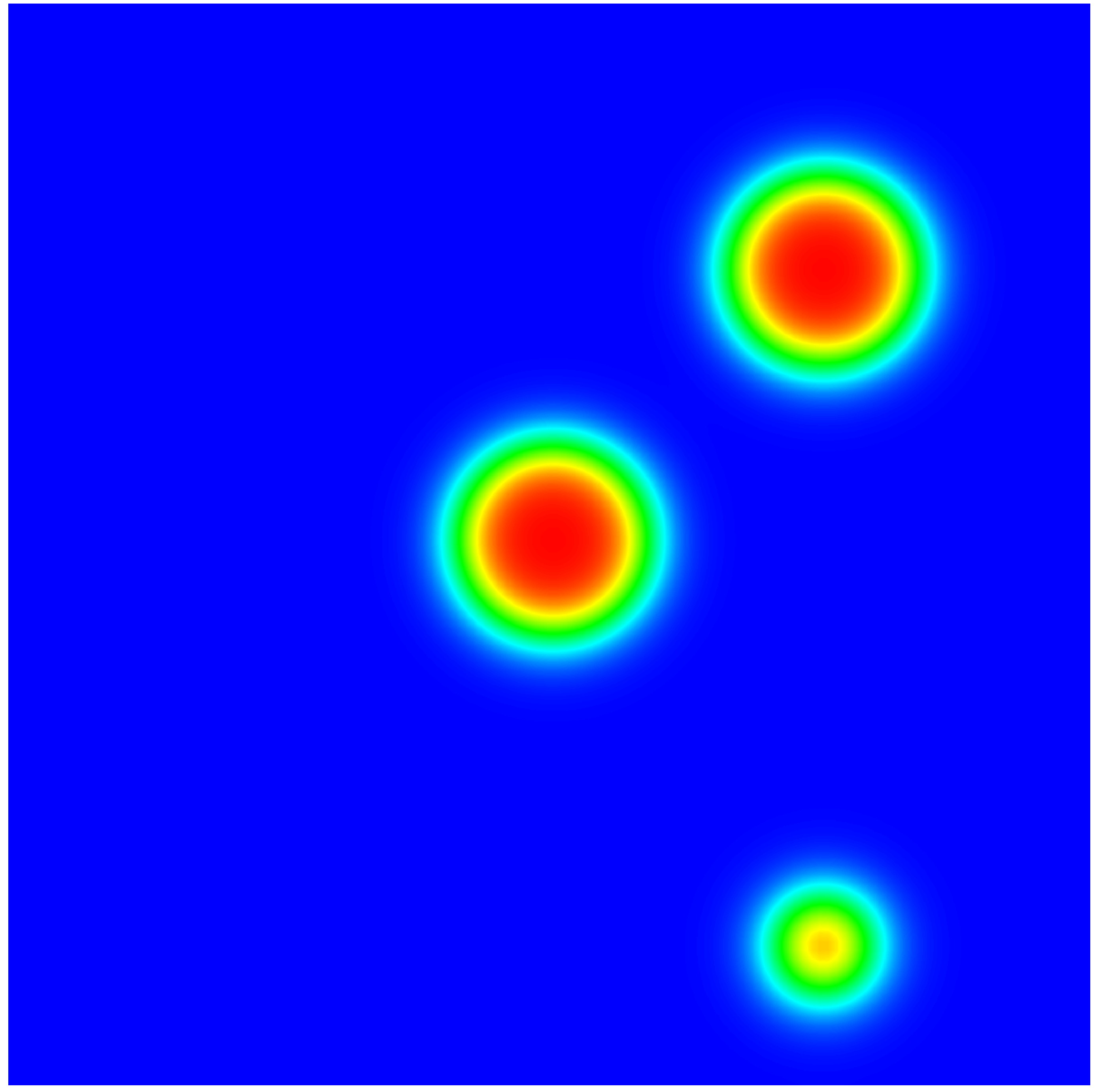}
\includegraphics[width=0.2\textwidth]{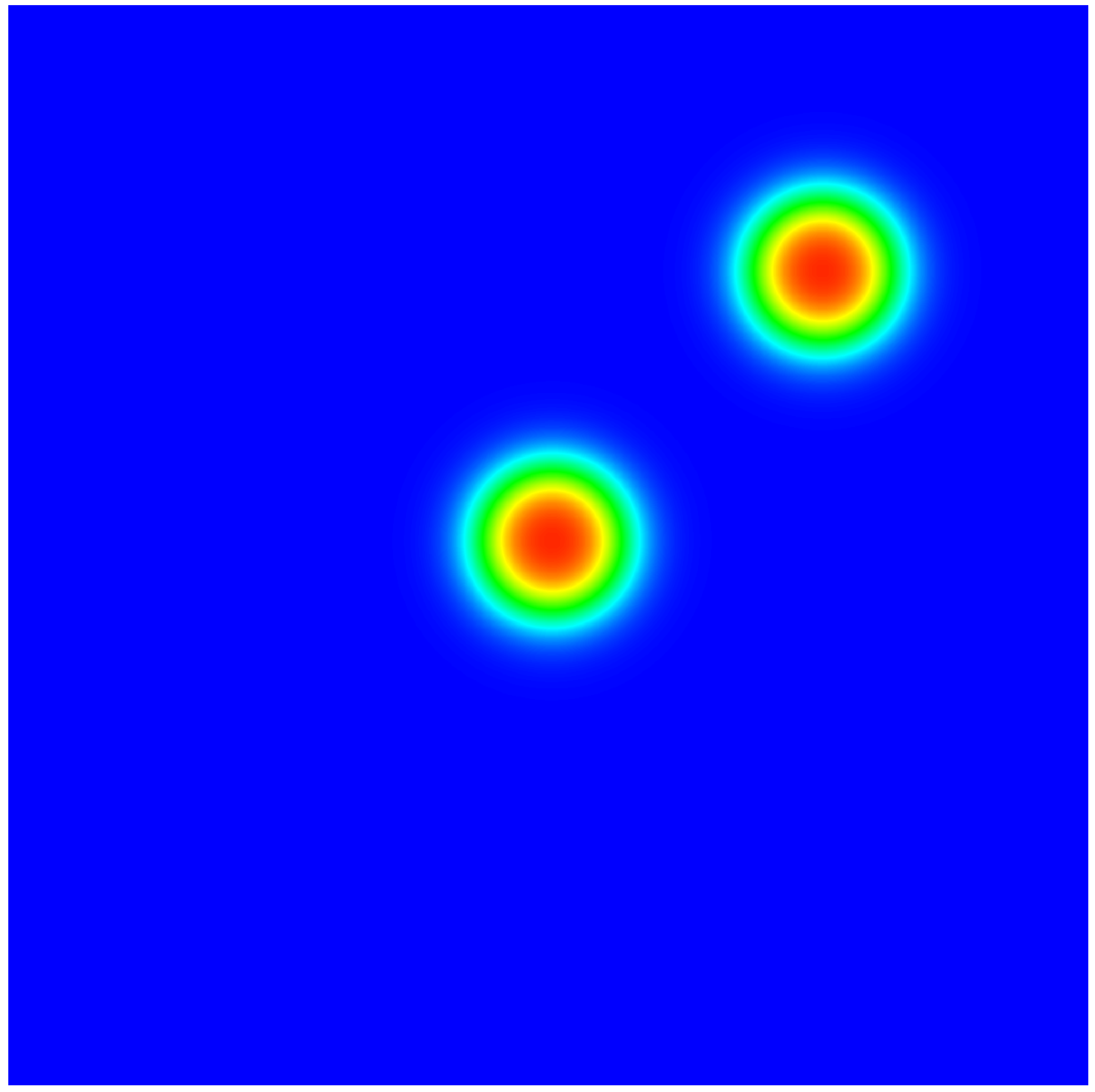}
}

\caption{A comparison of the numerical solutions between the full order model and the POD-ROM-II with various numbers of modes using Scheme \ref{sch:Relaxed-POD-ROM-CN}.}
\label{fig:AC-Example1-Compare}
\end{figure}

The energy dissipation across different modes is illustrated in Figure \ref{fig:AC-Energy-Compare}(a), which highlights that the reduced order model with $r=4$ modes effectively traces the energy dissipation trends, while the one with $r=10$ modes precisely replicates the energy dissipation. Additionally, we present the energy evolution using both approach I and approach II with $r=10$ modes. It appears that approach II provides more accurate results. As we discussed, the reason why the approach I is not as accurate is that it preserves an energy dissipation law with a modified energy dissipation rate.

\begin{figure}[H]
\center
\subfigure[]{\includegraphics[width=0.45\textwidth]{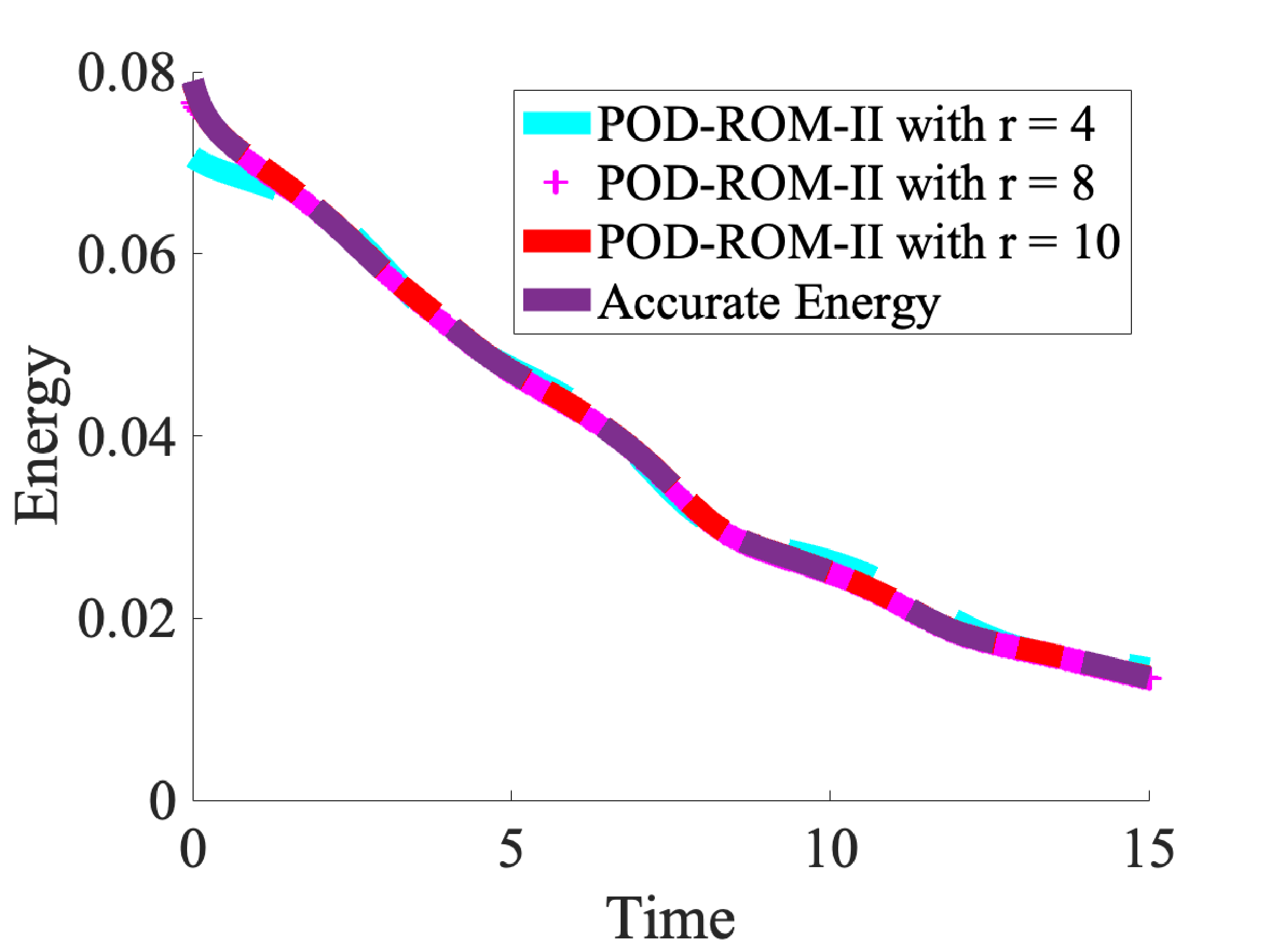}}
\subfigure[]{\includegraphics[width=0.45\textwidth]{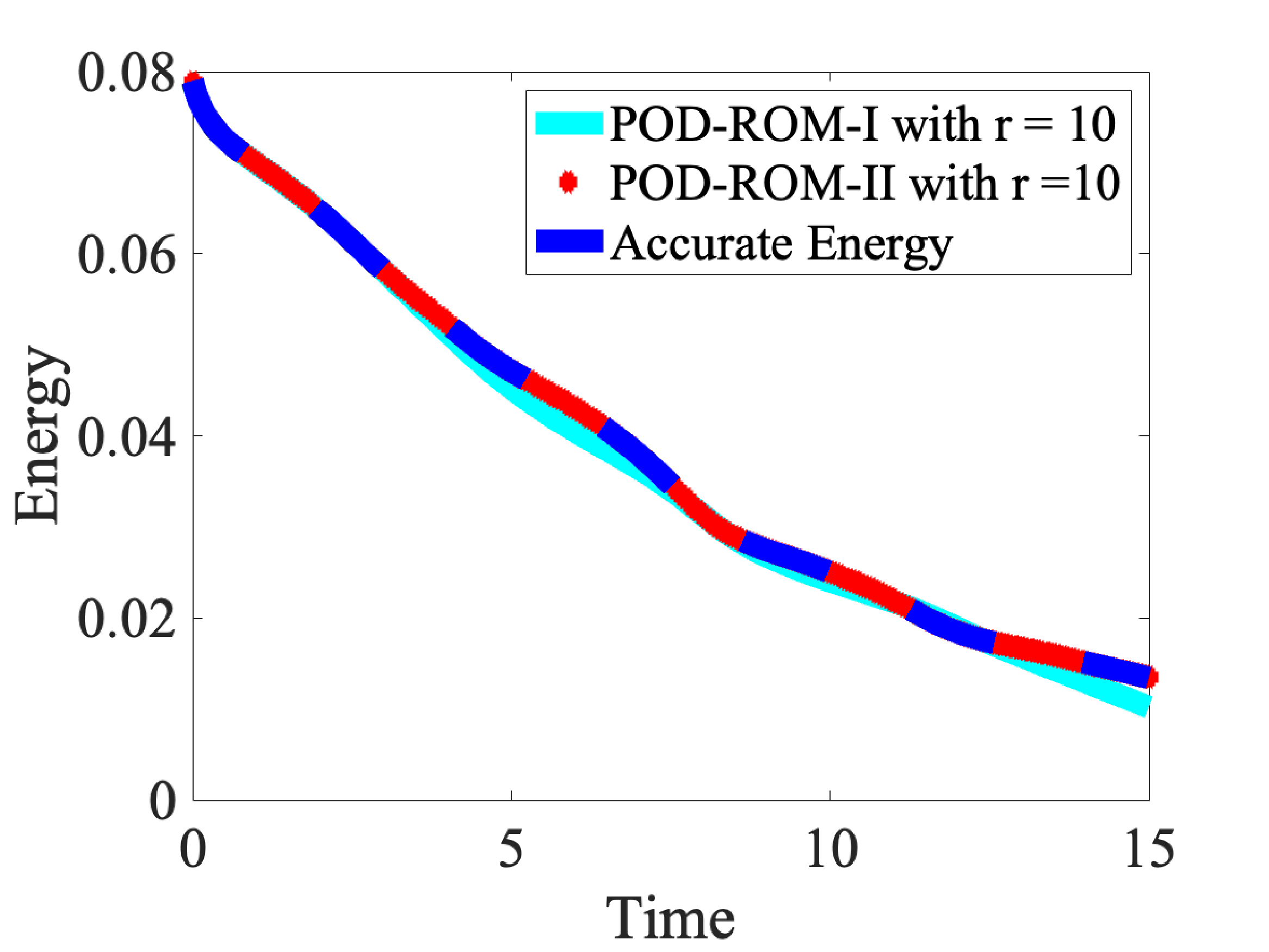}}

\label{fig:AC-Energy-Compare}
\caption{A comparison of the energy dissipation for the Allen-Cahn equation. (a) A comparison of the energy dissipation results between the full order model and POD-ROM-II with various modes using Scheme \ref{sch:Relaxed-POD-ROM-CN}; (b) A comparison of the energy dissipation results between POD-ROM-I using Scheme \ref{sch:Relaxed-POD-ROM-CN-I} and POD-ROM-II using Scheme \ref{sch:Relaxed-POD-ROM-CN} This figure underscores the proficiency of the proposed POD-ROM approach in capturing the thermodynamic structures of the full order model with only a few modes.}
\end{figure}

\subsection{Cahn-Hilliard equation}
In the next problem, we investigate the Cahn-Hilliard (CH) equation
\begin{subequations}
\begin{align}
& \partial_t \phi = M\Delta \mu , \quad (\bx, t) \in \Omega \times (0, T], \\
& \mu = - \varepsilon^2 \Delta  \phi + \phi^3 - \phi, \quad (\bx, t) \in \Omega \times (0, T],\\
& \phi(\bx,0) = \phi_0(\bx), \quad \bx \in \Omega,
\end{align}
\end{subequations}
with periodic boundary conditions. The Onsager triplet for the CH equation is
$$
(\phi,\quad \cG, \quad \cE) := \Big( \phi, \quad  -M\Delta, \quad \int_\Omega \Big[ \frac{\varepsilon^2}{2}|\nabla \phi|^2 + \frac{1}{4}(\phi^2 - 1)^2 \Big] d\bx \Big).
$$
Similarly, we can introduce 
$$
q = \frac{\sqrt{2}}{2}(\phi^2-1 -\gamma_0), \quad g(\phi):=\frac{\partial q}{\partial \phi} = \sqrt{2}\phi.
$$
And we will have the equation rewritten as
\begin{subequations}
\begin{align}
& \partial_t \phi = M \Delta \Big(-\varepsilon^2 \Delta \phi + q g(\phi) \Big), \\
& \partial_t q = g(\phi) \partial_t \phi.
\end{align}
\end{subequations}
Denote $\Psi = \begin{bmatrix}
\phi \\ q
\end{bmatrix}$, $\cG_s = -M\Delta $, $\cL_0 = -\varepsilon^2 \Delta + \gamma_0$, $\cL = \begin{bmatrix}
\cL_0  &  0  \\
0 & 1
\end{bmatrix}$ and $\cN_0 = \begin{bmatrix}
\bI  & g(\phi)
\end{bmatrix}$, such that the equation is written as
\beq
\partial_t \Psi = -\cN(\Psi) \cL  \Psi, \quad \mbox{ where } \quad \cN(\Psi) = \cN_0^T\cG_s\cN_0 .
\eeq 
Hence, the POD-ROM method introduced in previous sections can be directly applied.

%\todo[inline]{In the code for the CH equation, we have used the volume renormalization. Is this necessary? WHY? Any idea on designing ROM preserve total volume? In particular, the reduced-order initial profile is not volume-preserving itself. Shall we normalize it by $\phi = \phi - \overline{\phi}$ }

\textbf{Example 2.}
In this example, we choose the model parameters $M=0.01$ and  $\varepsilon=0.02$. The domain $[0, 1]^2$ is considered. The numerical parameters used are $\gamma_0 = 2$ and $N_x=N_y=128$. The initial condition is chosen the same as \eqref{eq:AC-initial-condition}. This represents seven disks of various sizes in different locations of the domain. Consider $T=15$, and $\delta t = 10^{-3}$. We first generate the data using an accurate numerical solver. And collect the data
$
\begin{bmatrix}
\Phi_1 & \Phi_2 & \cdots & \Phi_m
\end{bmatrix},
$
where $\Phi_k \in \mathbb{R}^n$ with $n=N_xN_y$ are the numerical solution at $t=0.1k$ in a vector form. In this example, we choose $N_x=N_y=128$ and $m=150$. The singular value distribution is summarized in Figure \ref{fig:CH-SVD}.

\begin{figure}[H]
\centering
\includegraphics[width=0.85\textwidth]{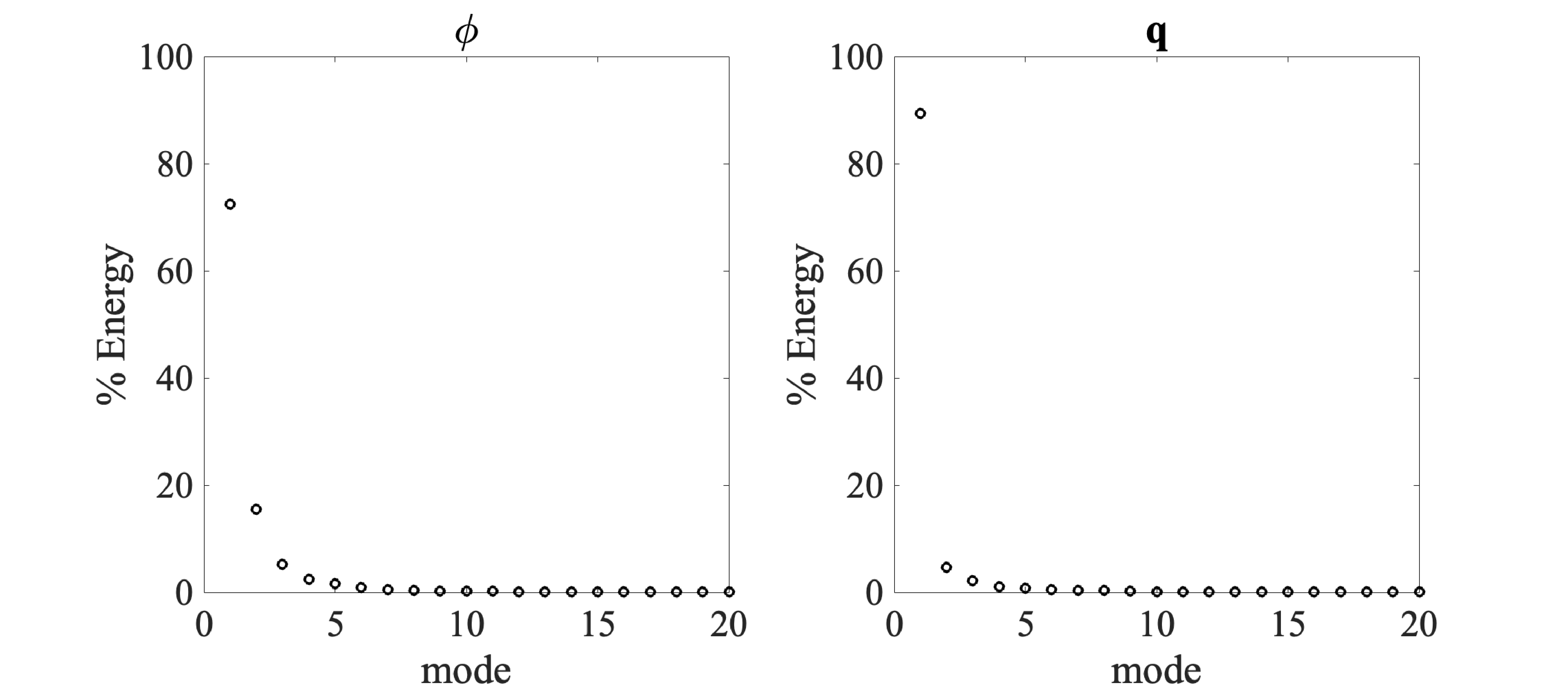}
\caption{Singular value distributions for the data collected from the Cahn-Hilliard equation.}
\label{fig:CH-SVD}
\end{figure}

Then, we follow the POD-ROM numerical framework proposed in previous sections. 
The numerical results are summarized in Figure \ref{fig:CH-Example1-Compare}. It illustrates that the reduced order model with $r=10$ modes can have an accurate approximation of the coarsening dynamics. The numerical results with $r=15$ modes show similar dynamics to the results from the full order model.

\begin{figure}[H]
\centering

\subfigure[Numerical solution from POD-ROM-II with r=10 at $t=10,30,50,90$]{
\includegraphics[width=0.2\textwidth]{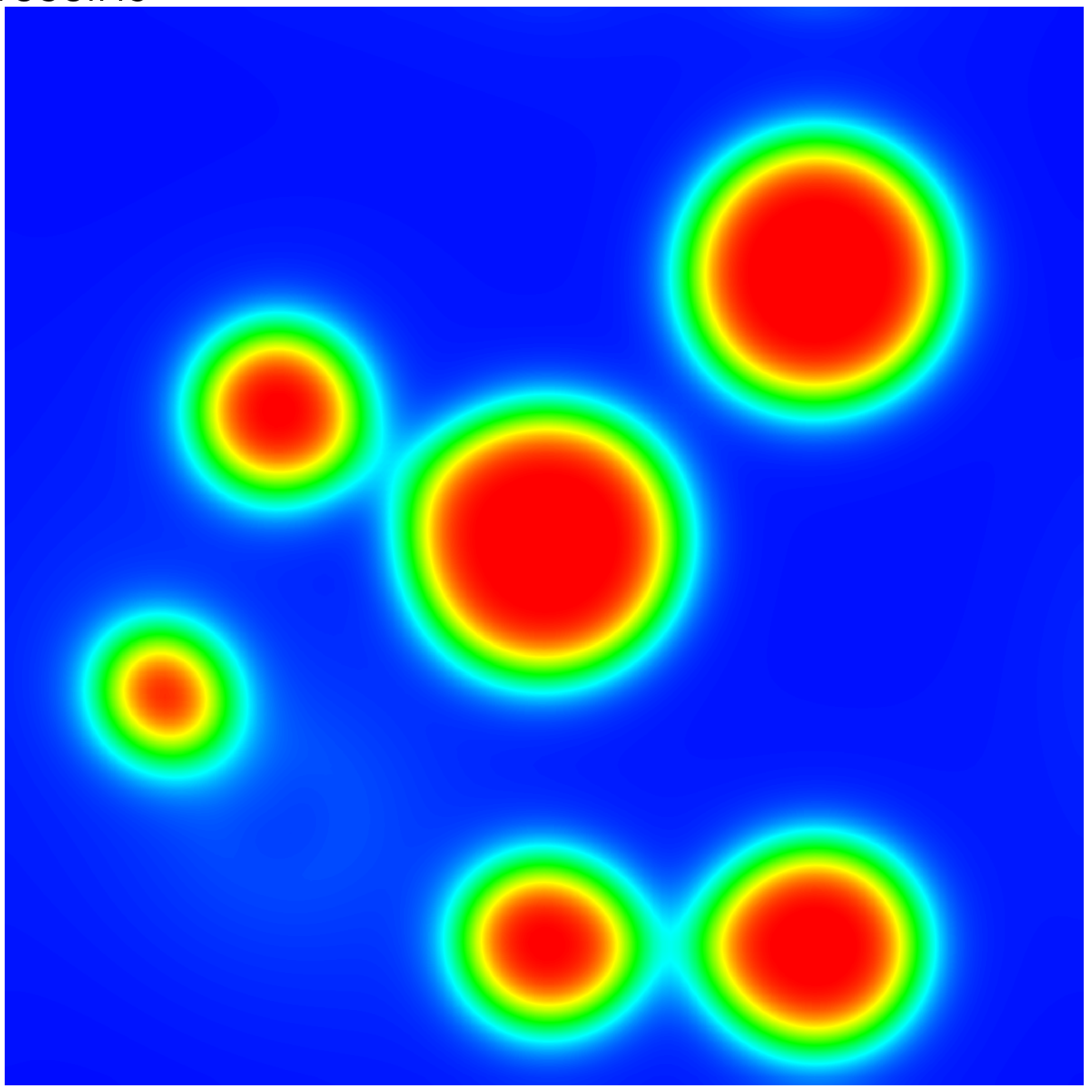}
 \includegraphics[width=0.2\textwidth]{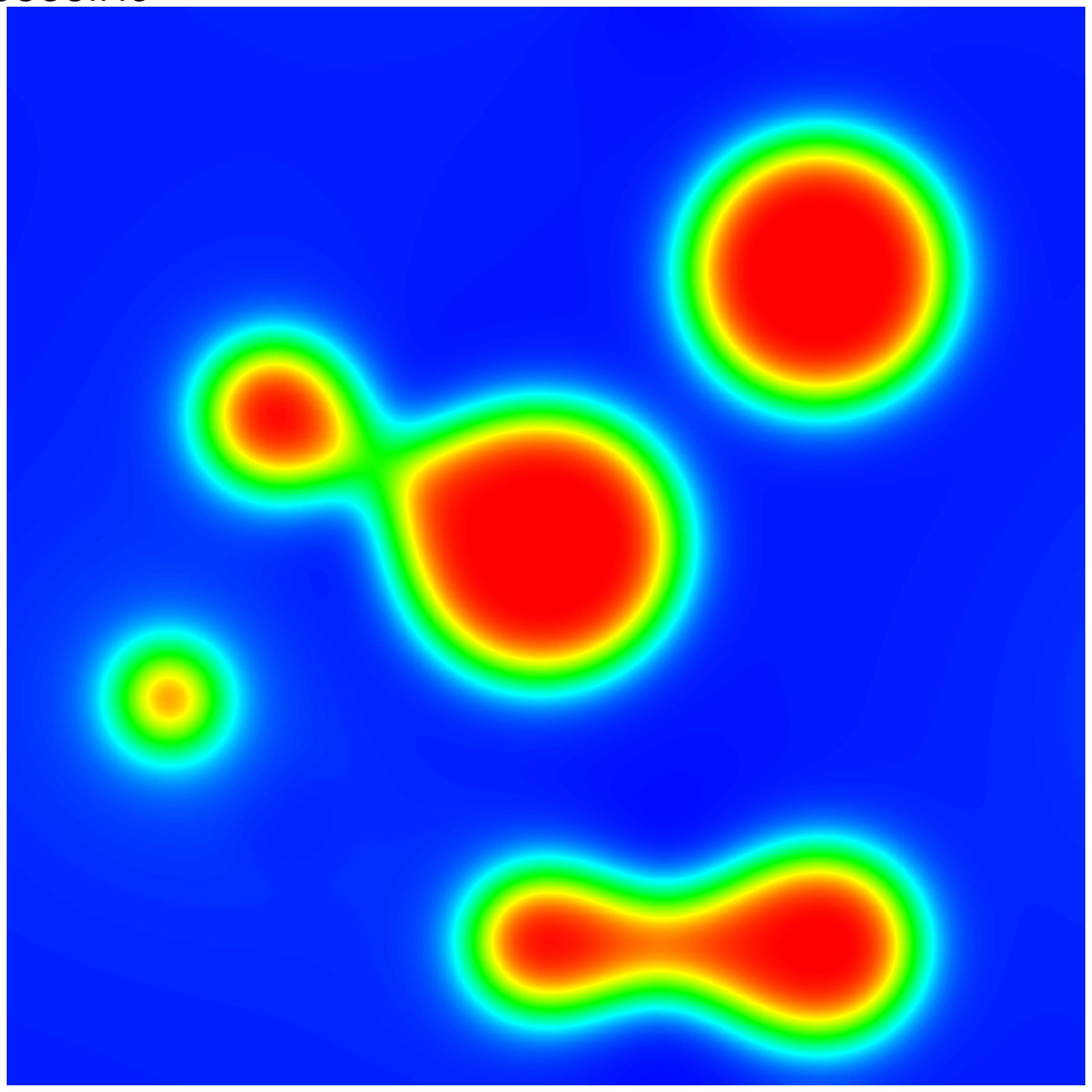}
  \includegraphics[width=0.2\textwidth]{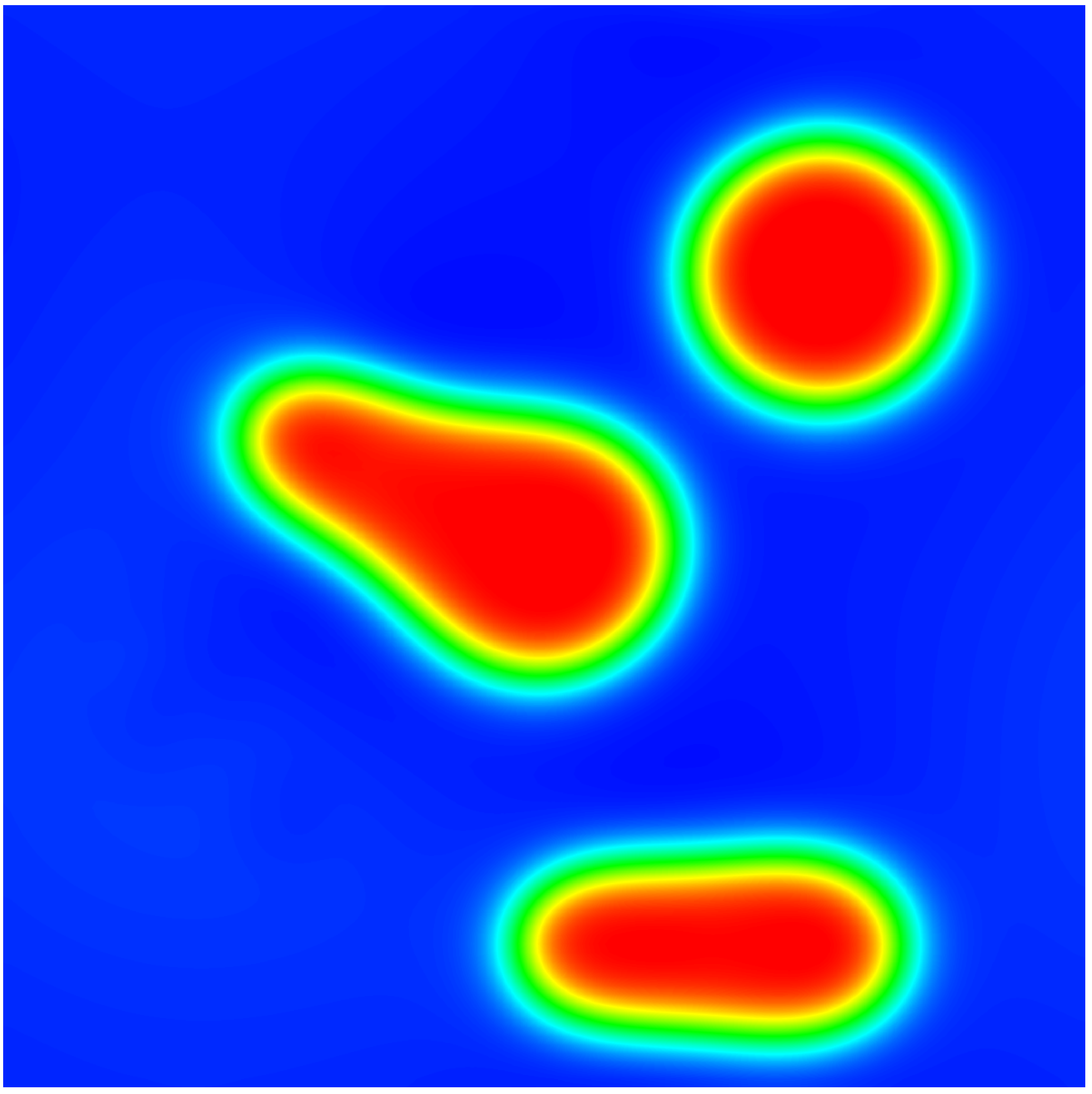}
   \includegraphics[width=0.2\textwidth]{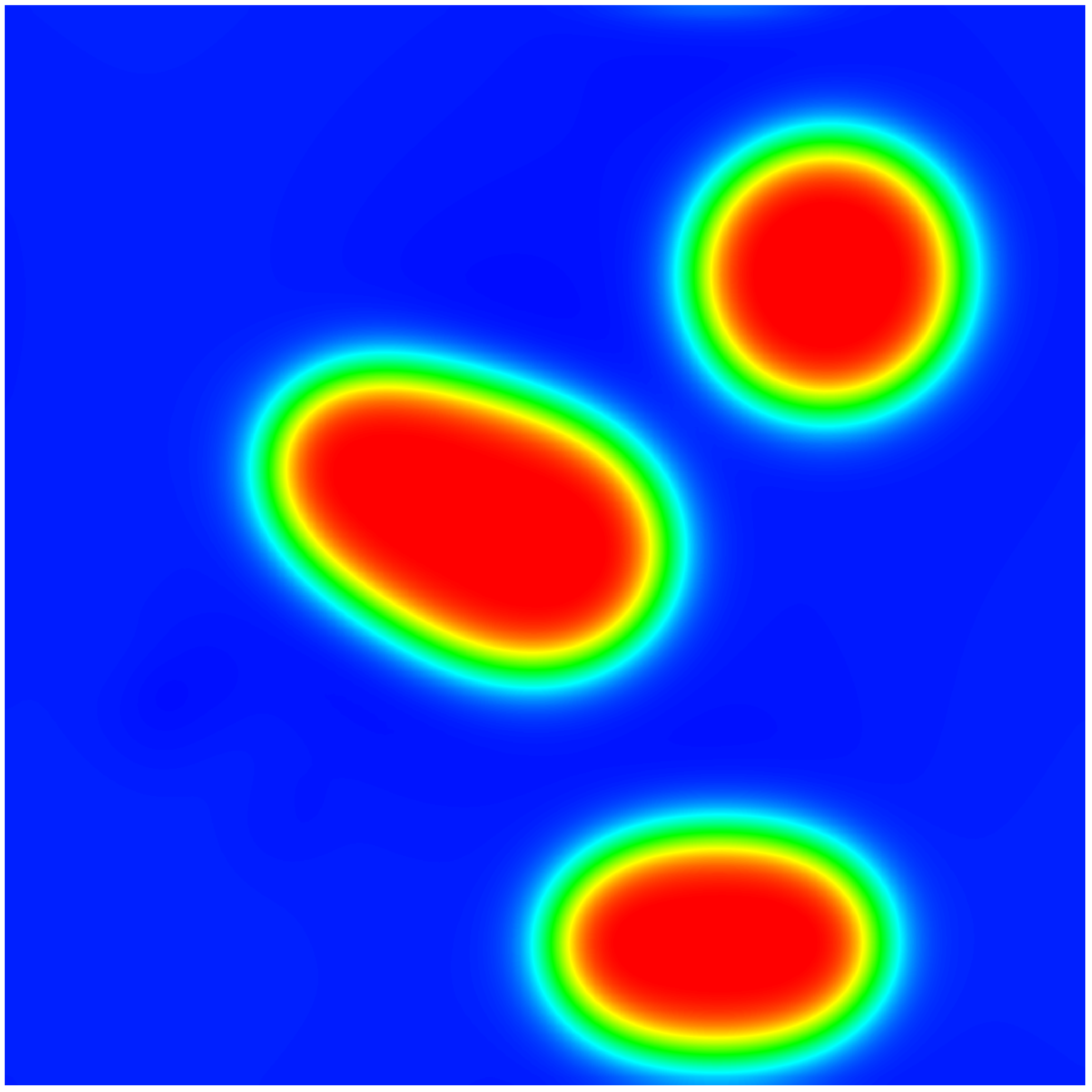}
}

\subfigure[Numerical solution from POD-ROM-II with r=15 at $t=10,30,50,90$]{
\includegraphics[width=0.2\textwidth]{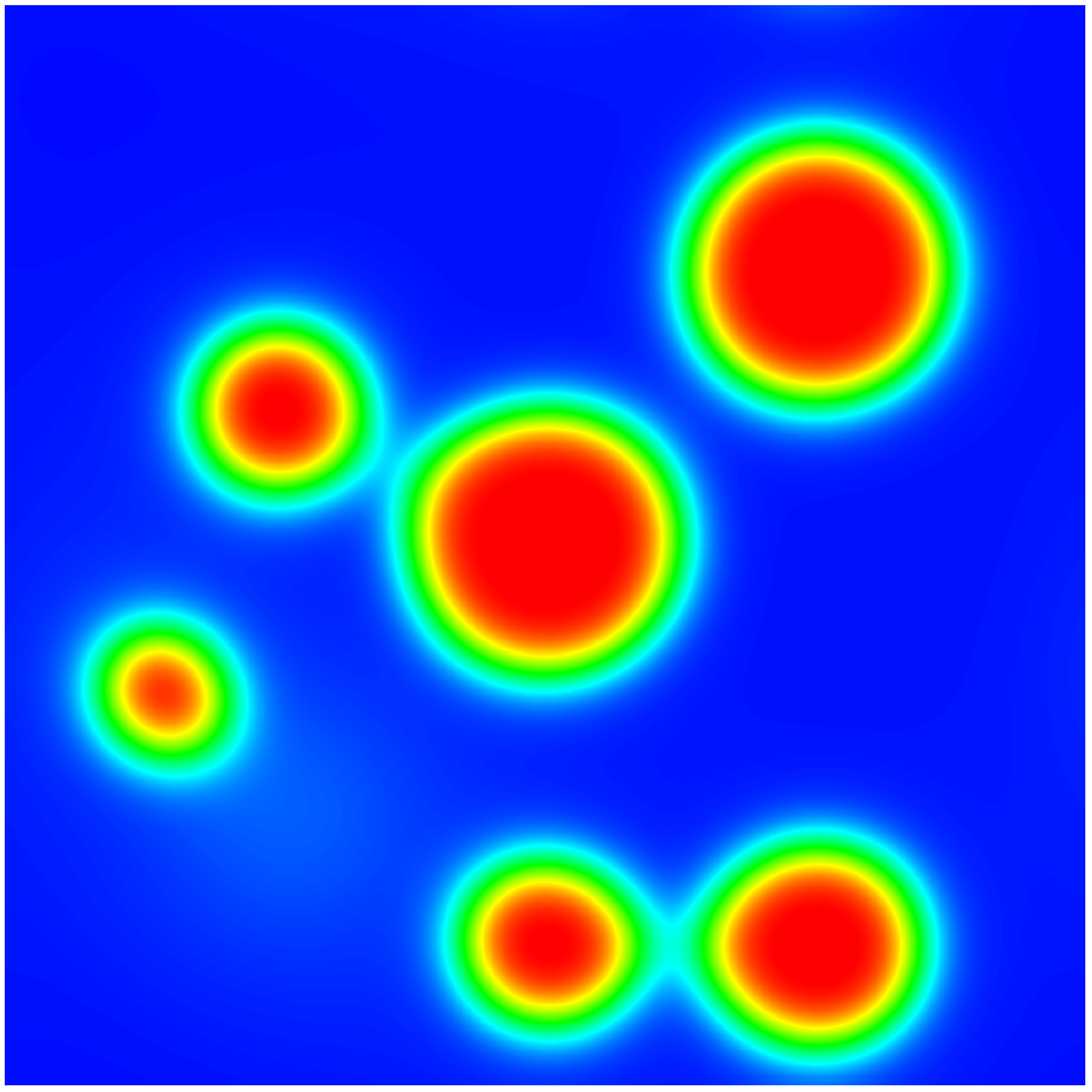}
 \includegraphics[width=0.2\textwidth]{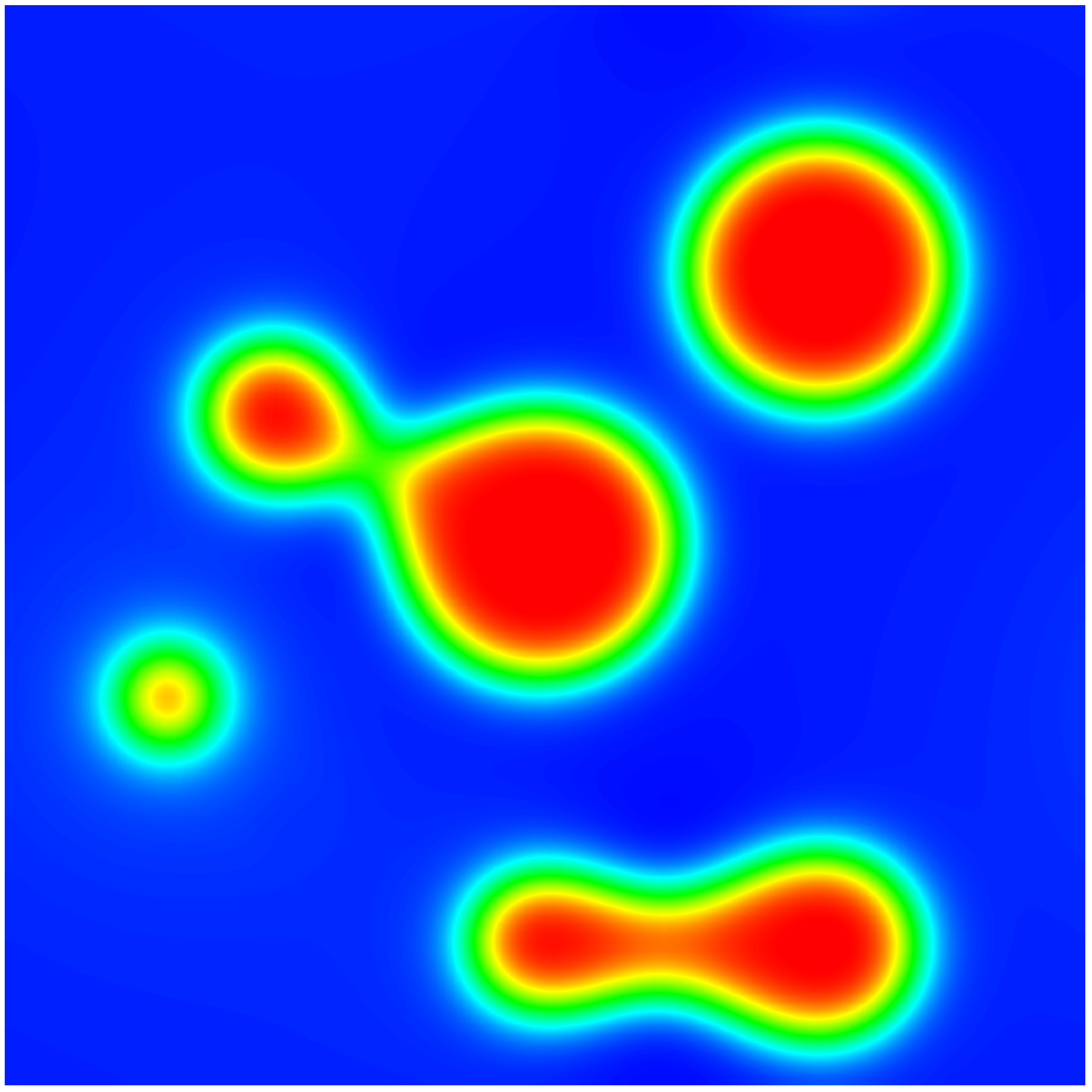}
  \includegraphics[width=0.2\textwidth]{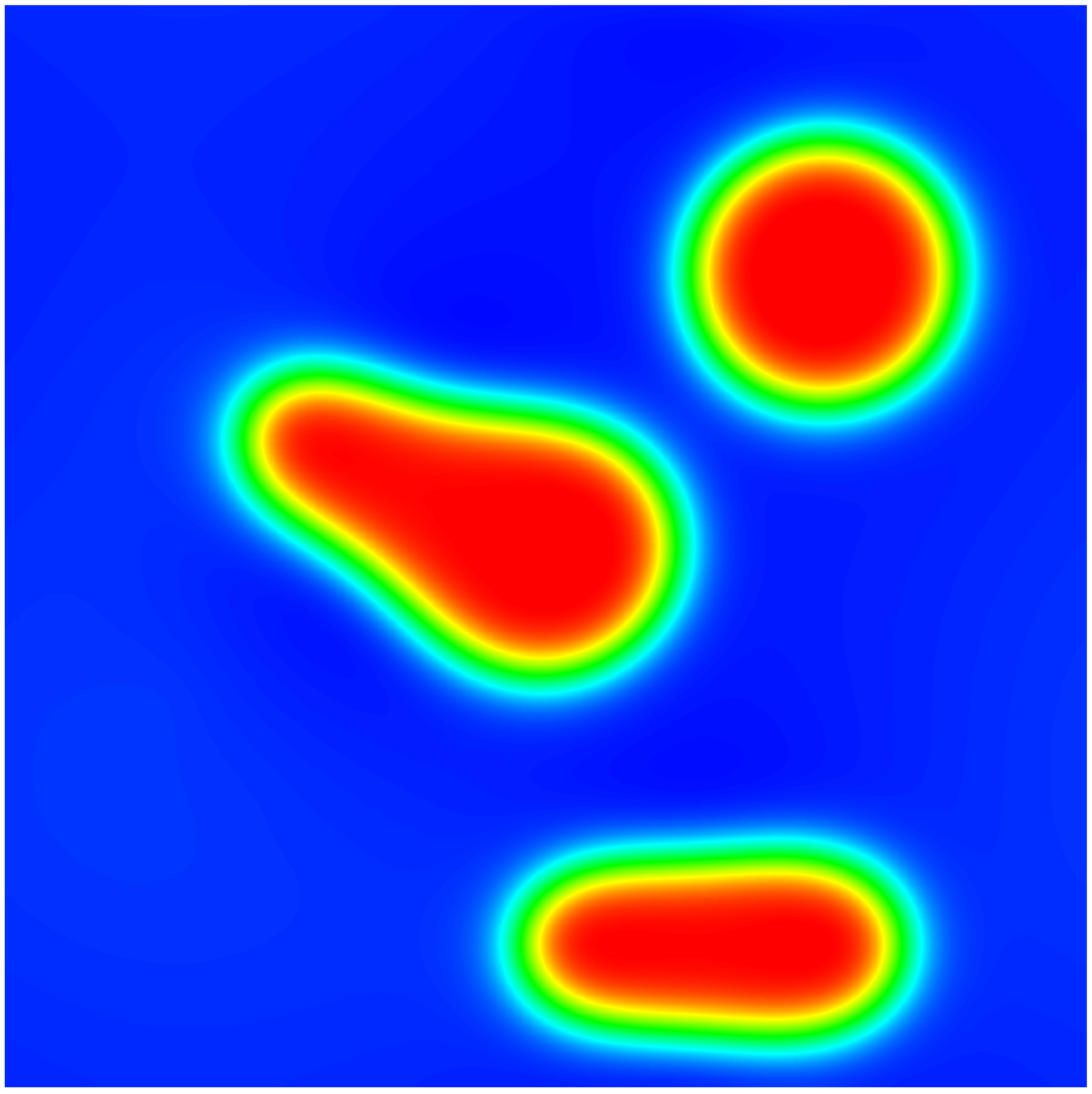}
   \includegraphics[width=0.2\textwidth]{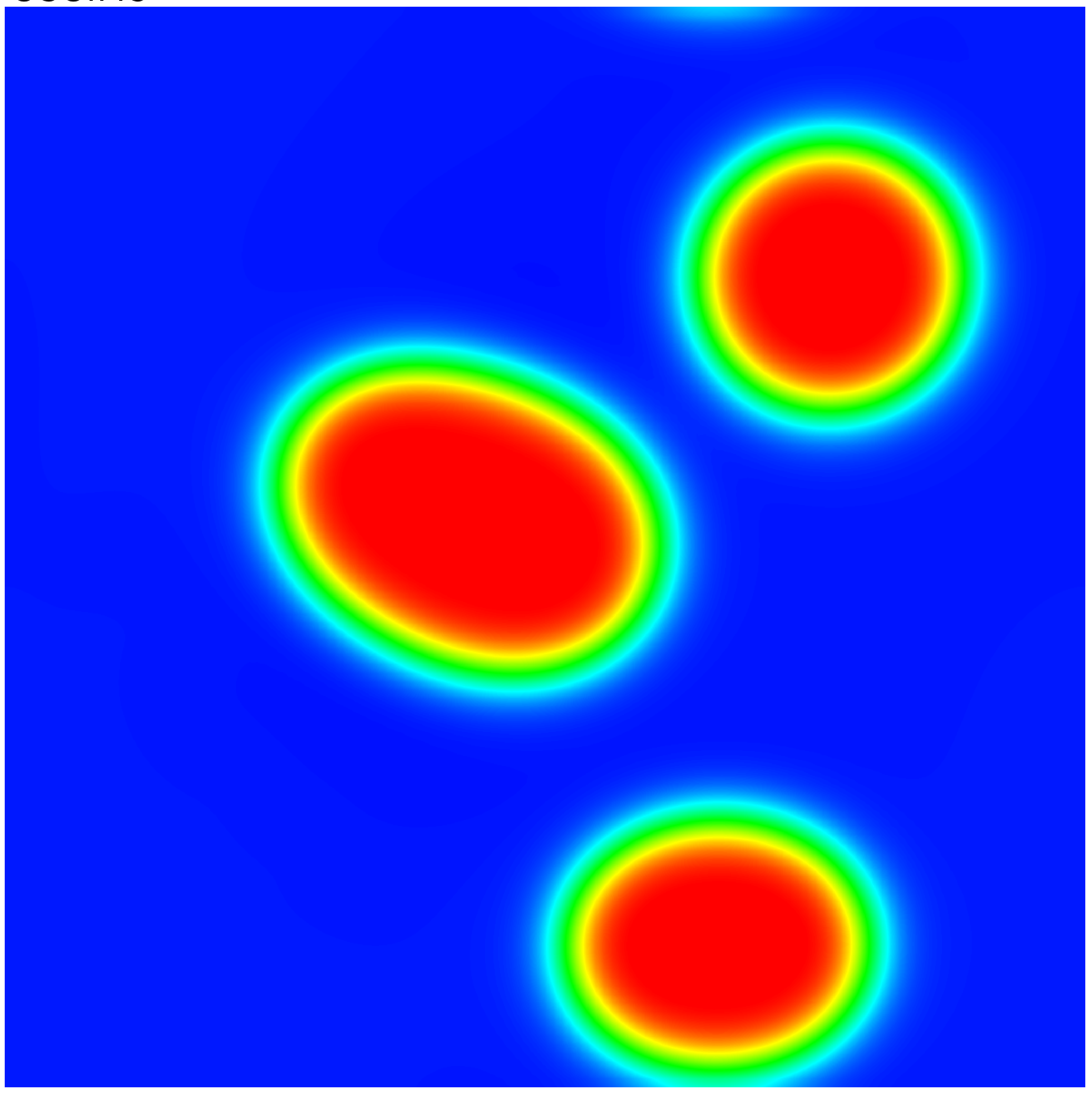}
}

\subfigure[Numerical solution from the full order model at $t=10,30,50,90$]{
\includegraphics[width=0.2\textwidth]{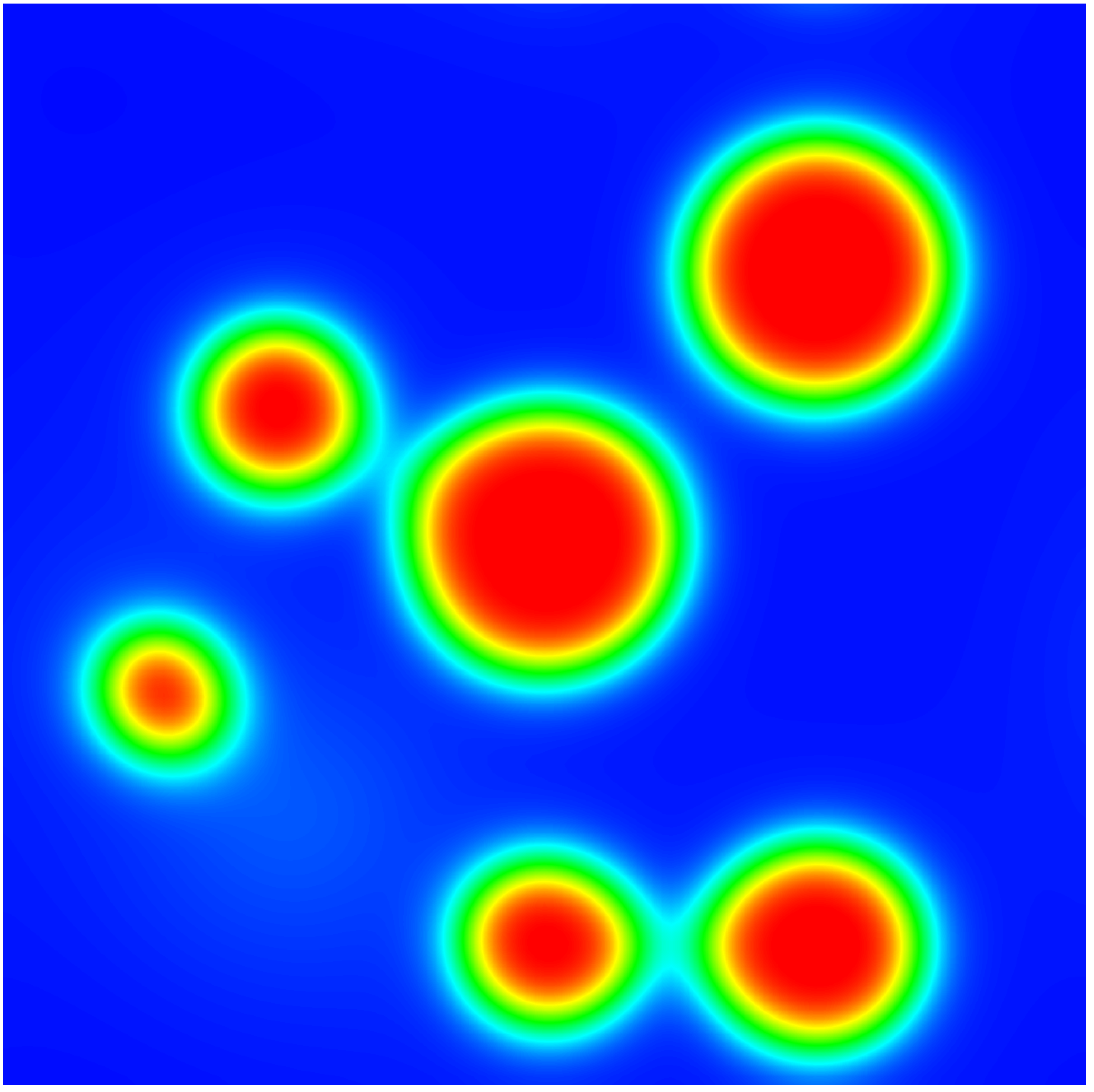}
 \includegraphics[width=0.2\textwidth]{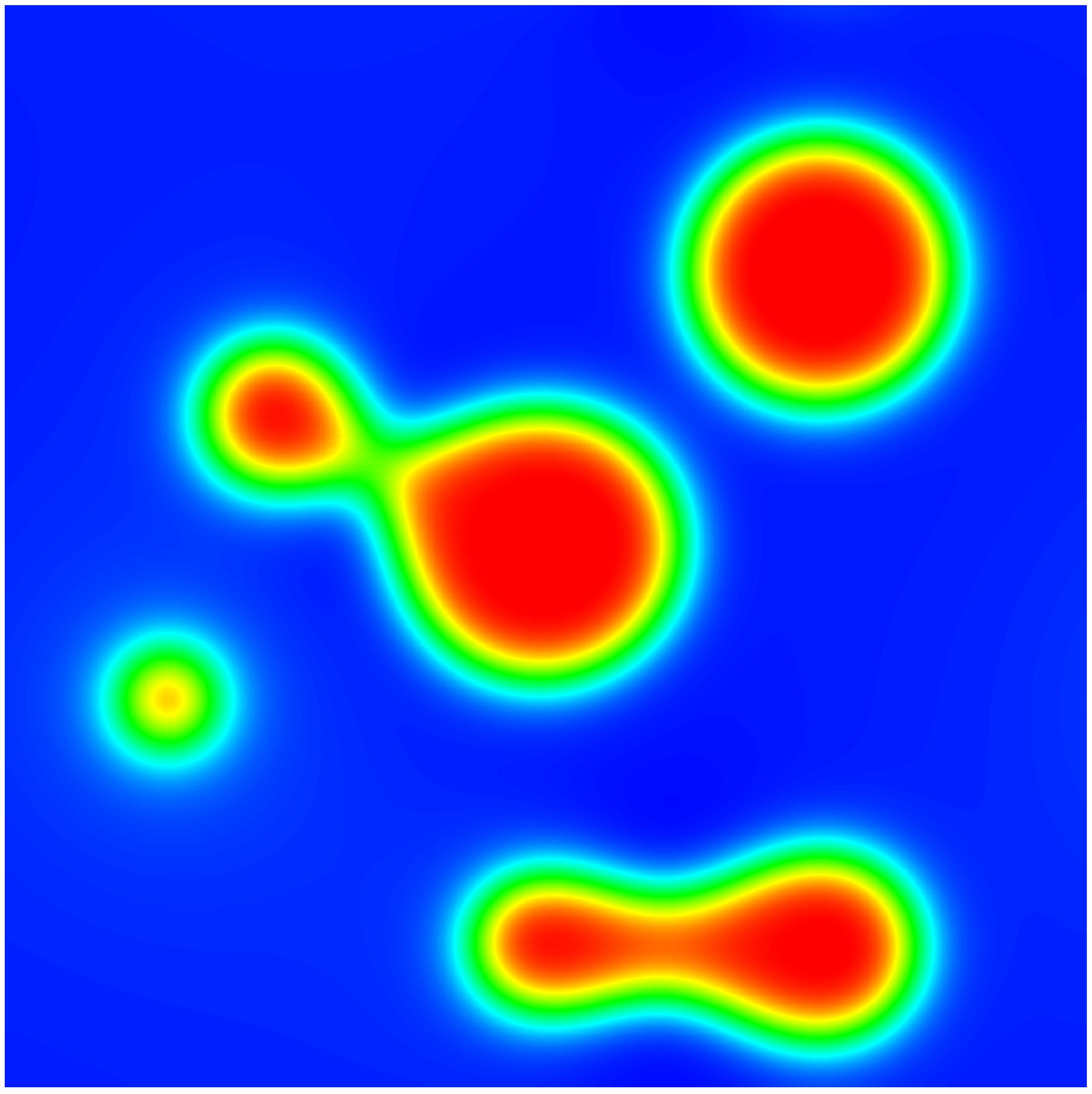}
  \includegraphics[width=0.2\textwidth]{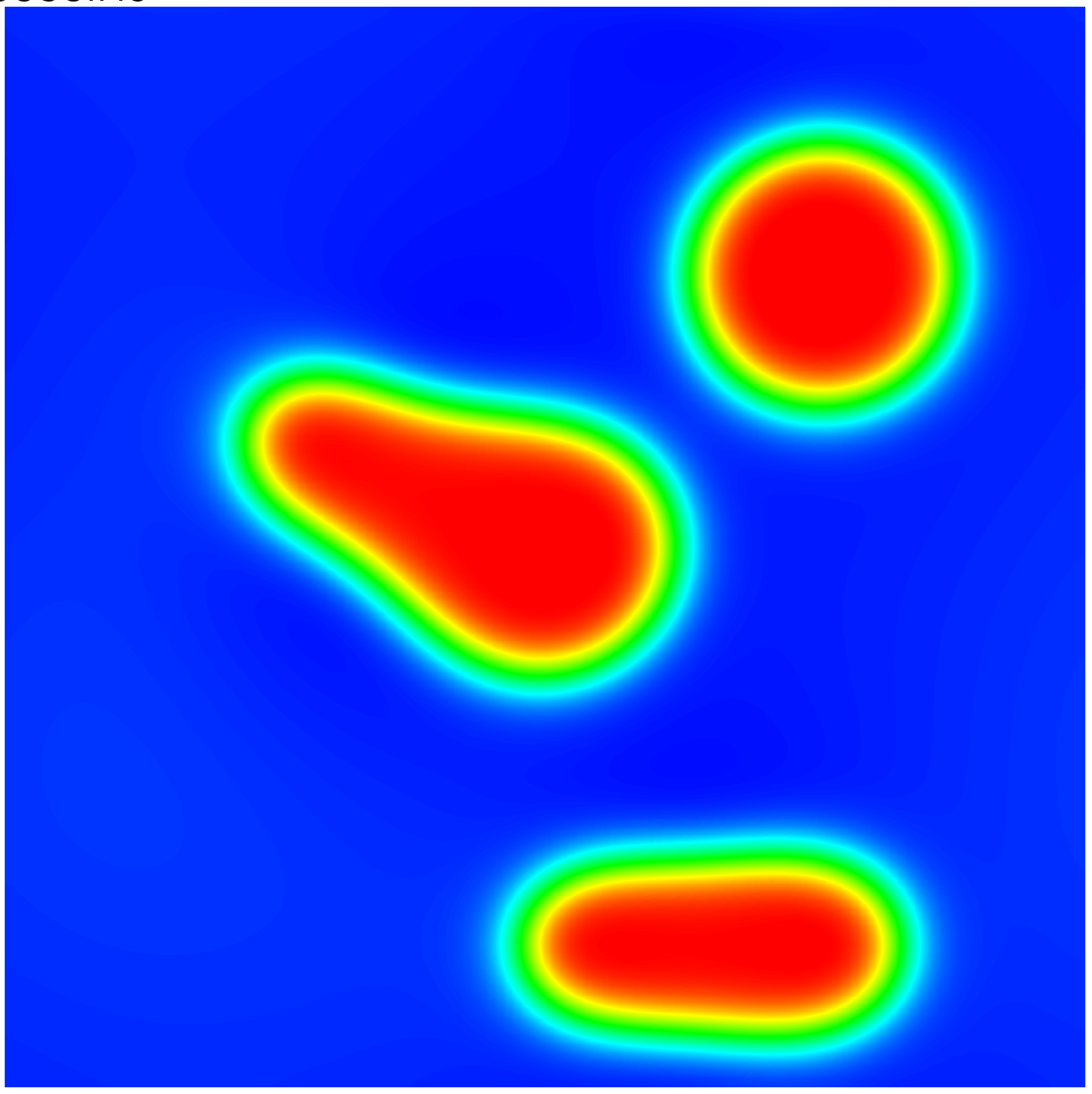}
   \includegraphics[width=0.2\textwidth]{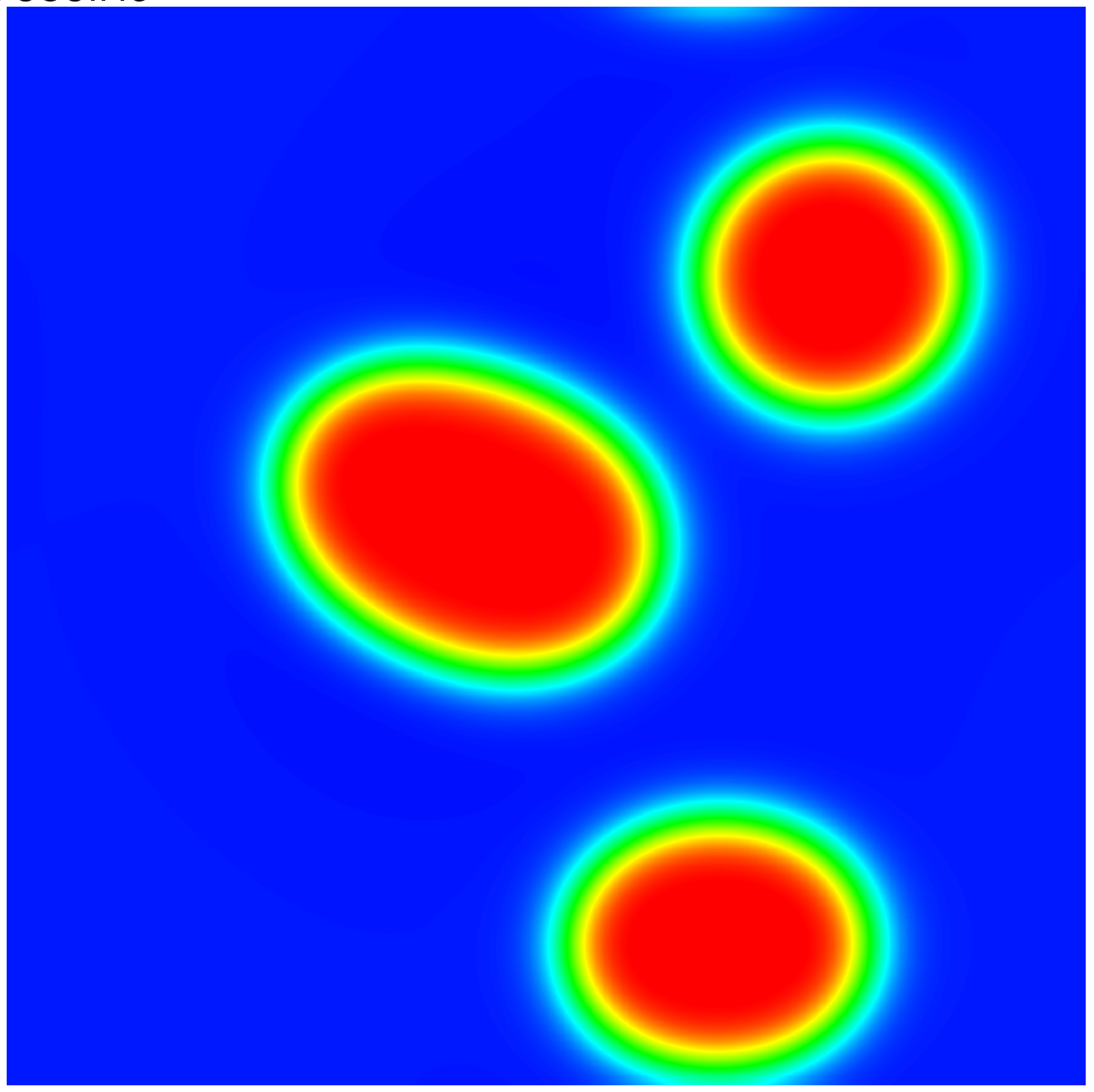}
}

\caption{A comparison between the numerical solutions of the Cahn-Hilliard equation from the full order model and the numerical solutions from the POD-ROM-II with various numbers of modes.}
\label{fig:CH-Example1-Compare}
\end{figure}

Furthermore, we have summarized the results of the energy dissipation for POD-ROM-II using different modes in Figure \ref{fig:CH-Energy-Compare}(a). It illustrates that the reduced order model with $r=15$ mode can accurately capture the Cahn-Hilliard equation's energy dissipation. Additionally, from Figure \ref{fig:CH-Energy-Compare}(b), we observe that POD-ROM-II provides a more accurate prediction than POD-ROM-I for energy dissipation.

\begin{figure}[H]
\center
\subfigure[]{\includegraphics[width=0.45\textwidth]{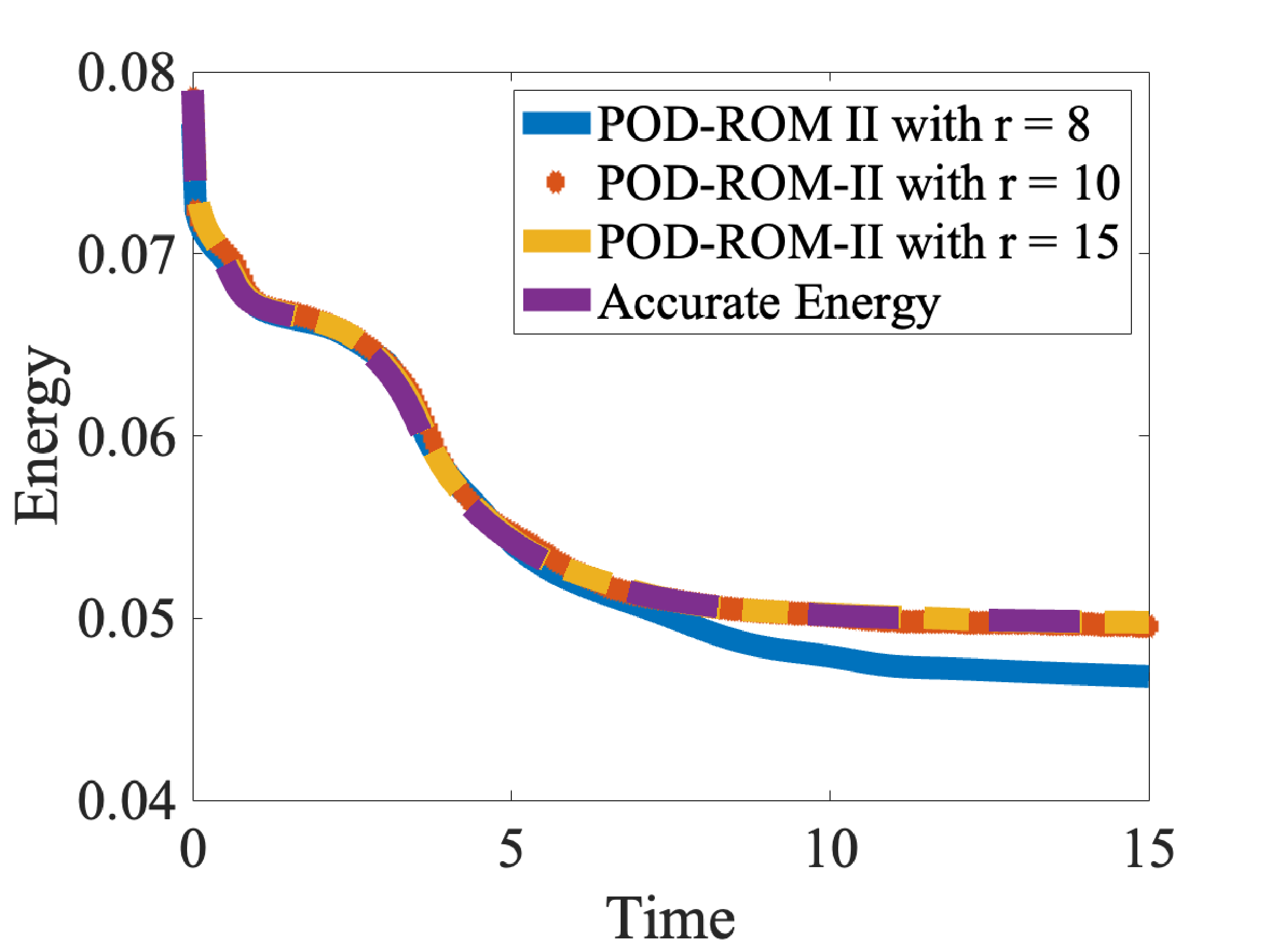}}
\subfigure[]{\includegraphics[width=0.45\textwidth]{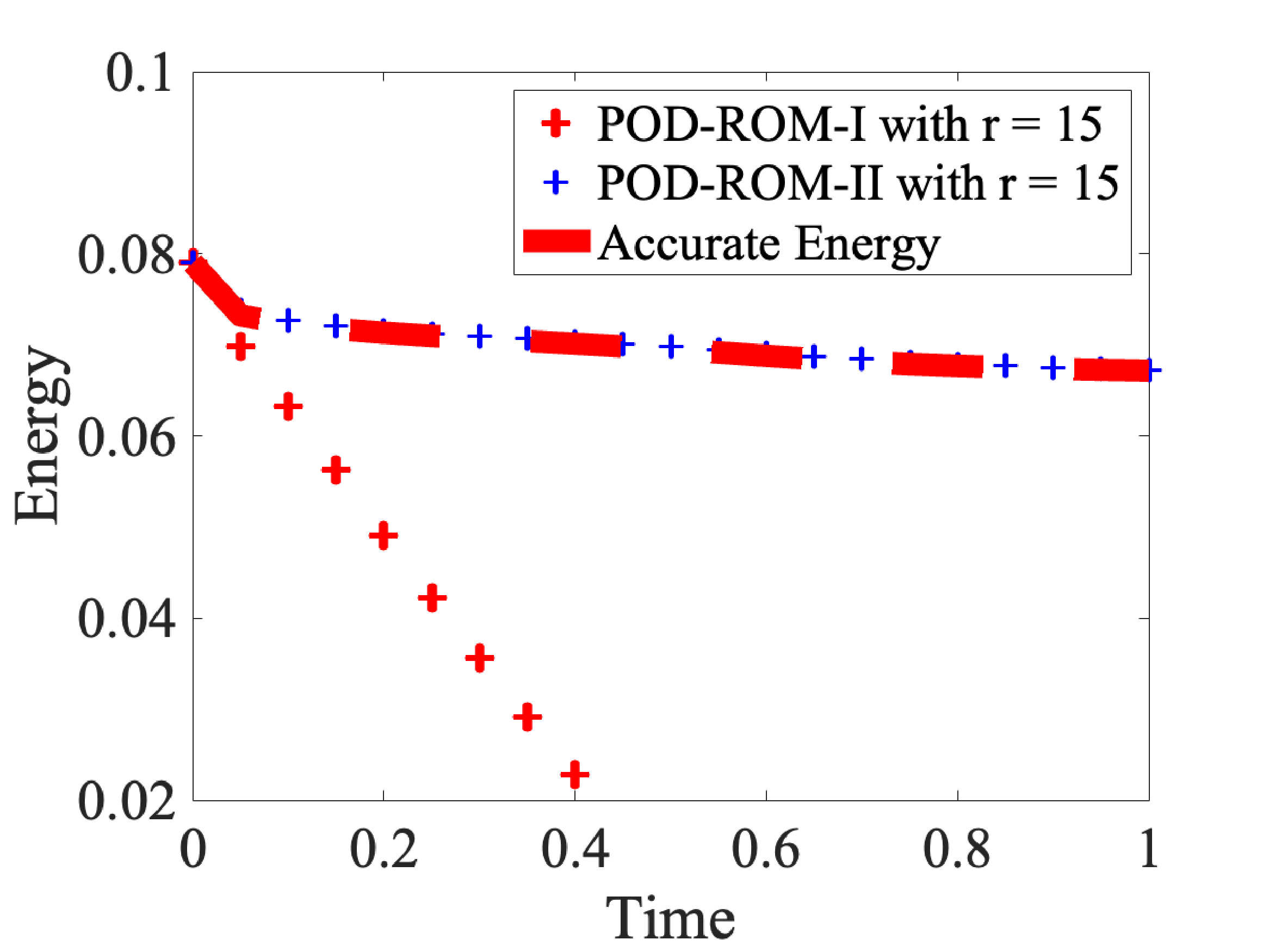}}
\caption{A comparison of the energy dissipation for the Cahn-Hilliard equation. (a) A comparison of the energy dissipation results between the full order model and the POD-ROM-II with various modes using Scheme \ref{sch:Relaxed-POD-ROM-CN}; (b) A comparison of the energy dissipation results between POD-ROM-I using Scheme \ref{sch:Relaxed-POD-ROM-CN-I} and POD-ROM-II using Scheme \ref{sch:Relaxed-POD-ROM-CN} }
\label{fig:CH-Energy-Compare}
\end{figure}

\subsection{Phase field crystal model}
Next, we investigate the phase field crystal (PFC) model with our proposed reduced-order model techniques. The PFC model reads as
\beq
\partial_t \phi = M\Delta \Big[ (a_0 + \Delta)^2 \phi + f'(\phi)\Big],
\eeq 
where $f(\phi) = \frac{1}{4} \phi^4 - \frac{b_0}{2}\phi^2$. Here $a_0$ and $b_0$ are model parameters. The PFC model can also be derived from the generalized Onsager principle with the Onsager triplet
$$
(\phi,\quad \cG, \quad \cE) := (\phi, \quad  -M\Delta, \quad \int_\Omega \Big[ \frac{1}{2}\phi (-b_0 + (a_0 + \Delta)^2 ) \phi  +\frac{1}{4}\phi^4 \Big] d\bx).
$$
Introduce the auxiliary variables
$$
q(\bx,t):= \frac{\sqrt{2}}{2}(\phi^2 - b_0 - \gamma_0), \quad g(\phi): = \frac{\partial q}{\partial \phi} = \sqrt{2} \phi, 
$$
the PFC model could be transformed as
\begin{subequations}
\begin{align}
& \partial_t \phi =  M\Delta \Big( (a_0+\Delta)^2 \phi + \gamma_0 \phi + g(\phi) q \Big), \\
& \partial_t q = g(\phi) \partial_t \phi.
\end{align}
\end{subequations}
Denote $\Psi = \begin{bmatrix}
\phi \\ q
\end{bmatrix}$, $\cG_s = -M\Delta $, $\cL_0 = (a_0+\Delta)^2 + \gamma_0$, $\cL = \begin{bmatrix}
\cL_0  &  0  \\
0 & \bI
\end{bmatrix}$ and $\cN_0 = \begin{bmatrix}
\bI  & g(\phi)
\end{bmatrix}$, such that the equation is written as
\beq
\partial_t \Psi = -\cN(\Psi) \cL  \Psi, \quad \mbox{ where } \quad \cN(\Psi) = \cN_0^T\cG_s\cN_0 .
\eeq 
Hence, the POD-ROM method introduced in previous sections can be directly applied.

In the numerical example, we choose $a_0=1$, $b_0=0.325$, $L_x=L_y=100$, $N_x=N_y=128$, $\gamma_0=1$, $M=1$, and $T=100$. We choose the domain $\Omega=[0, L_x] \times [0, L_y]$. We use $\delta t=10^{-3}$. We initialize three dots in the domain as the initial profile, following the strategy in \cite{Vignal2015}.   The singular value distribution is summarized in Figure \ref{fig:PFC-SVD}.

\begin{figure}[H]
\centering
\includegraphics[width=0.85\textwidth]{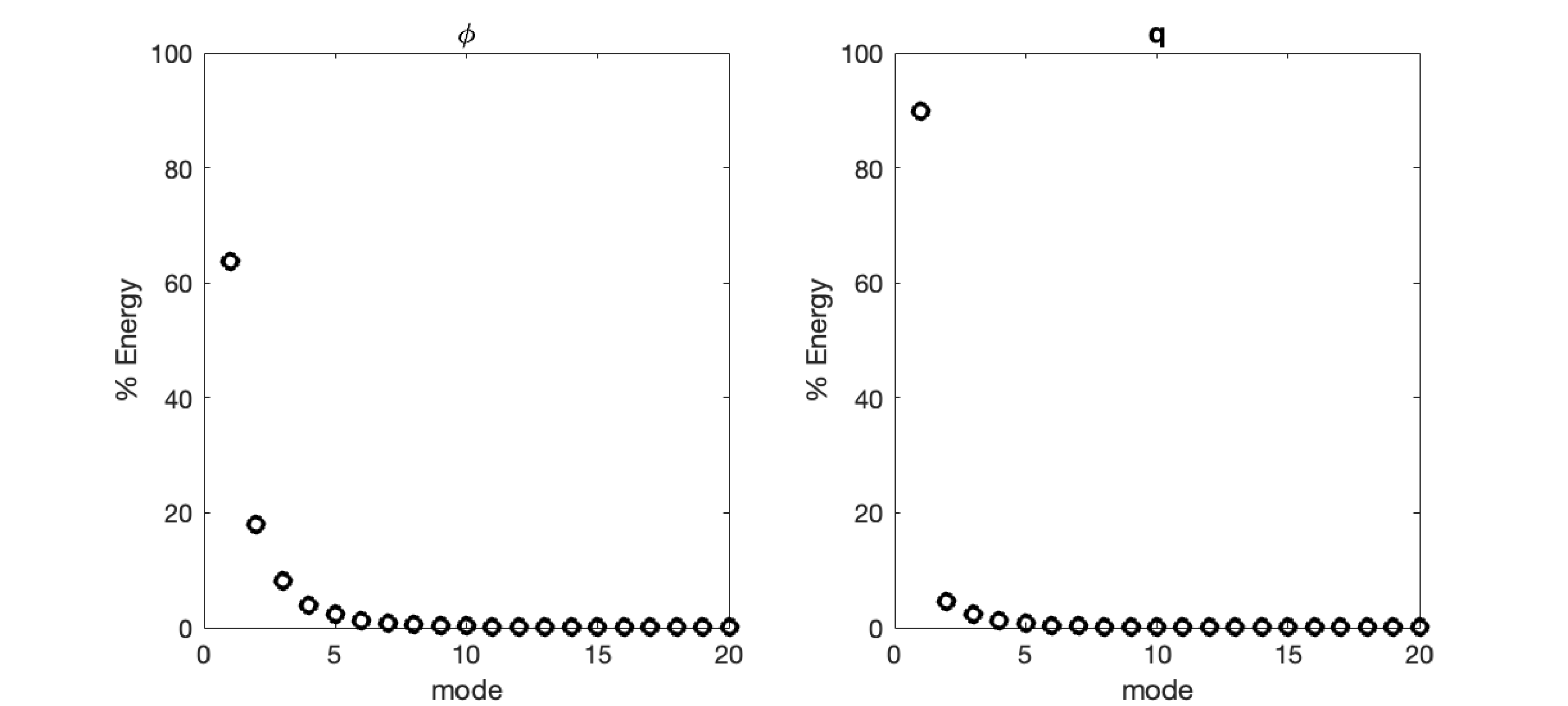}
\caption{Singular value distributions for the data collected from the phase field crystal equation.}
\label{fig:PFC-SVD}
\end{figure}

The results are summarized in Figure \ref{fig:PFC-Example-Compare}. We can observe that with $r=4$ modes, the ROM can already capture the dynamics properly. Furthermore, with $r=8$ modes, the ROM model captures the dynamics accurately.

\begin{figure}[H]
\centering

\subfigure[Numerical solution from POD-ROM-II with r=4 at $t=10,30,50,90$]{
\includegraphics[width=0.2\textwidth]{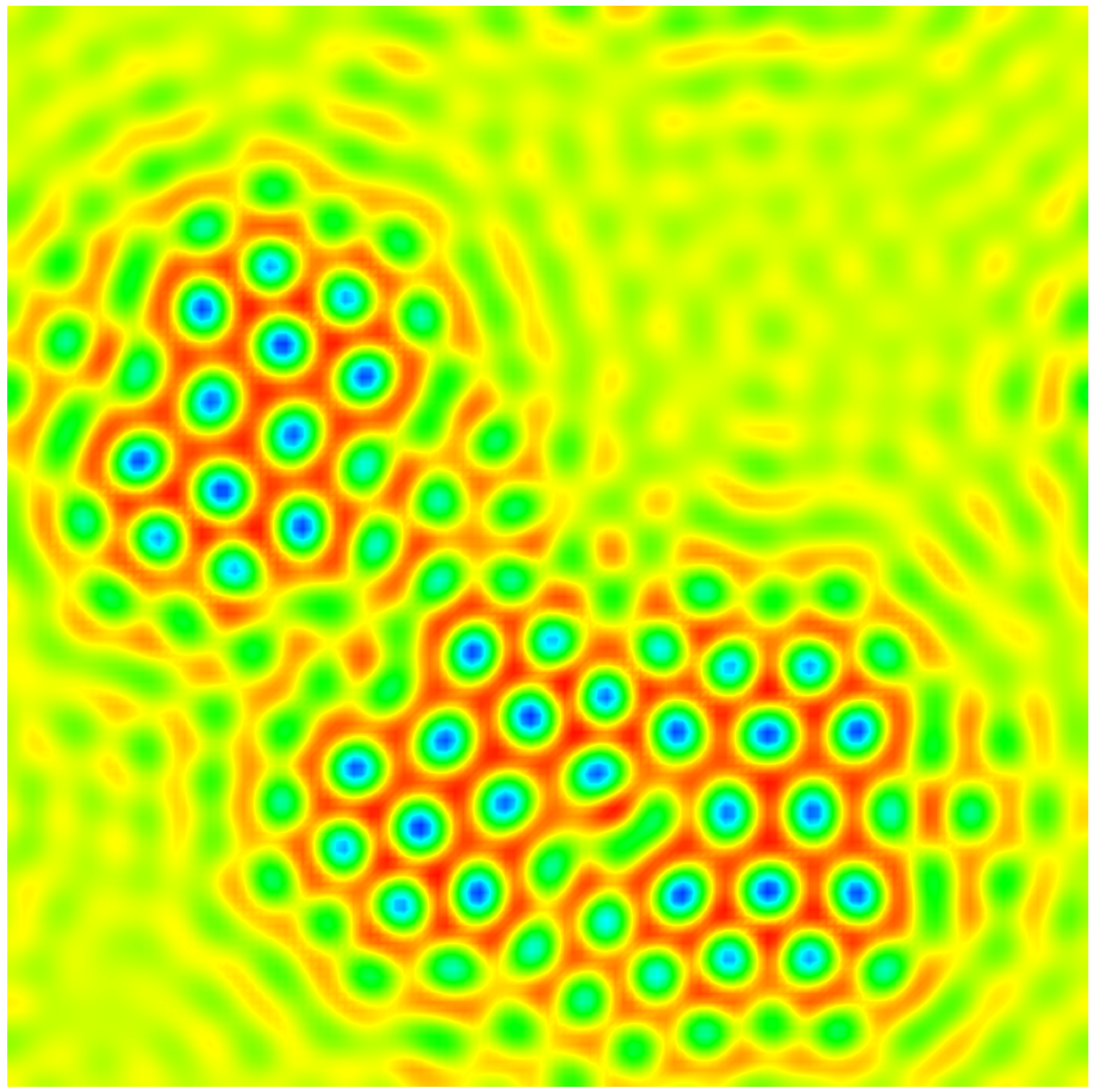}
 \includegraphics[width=0.2\textwidth]{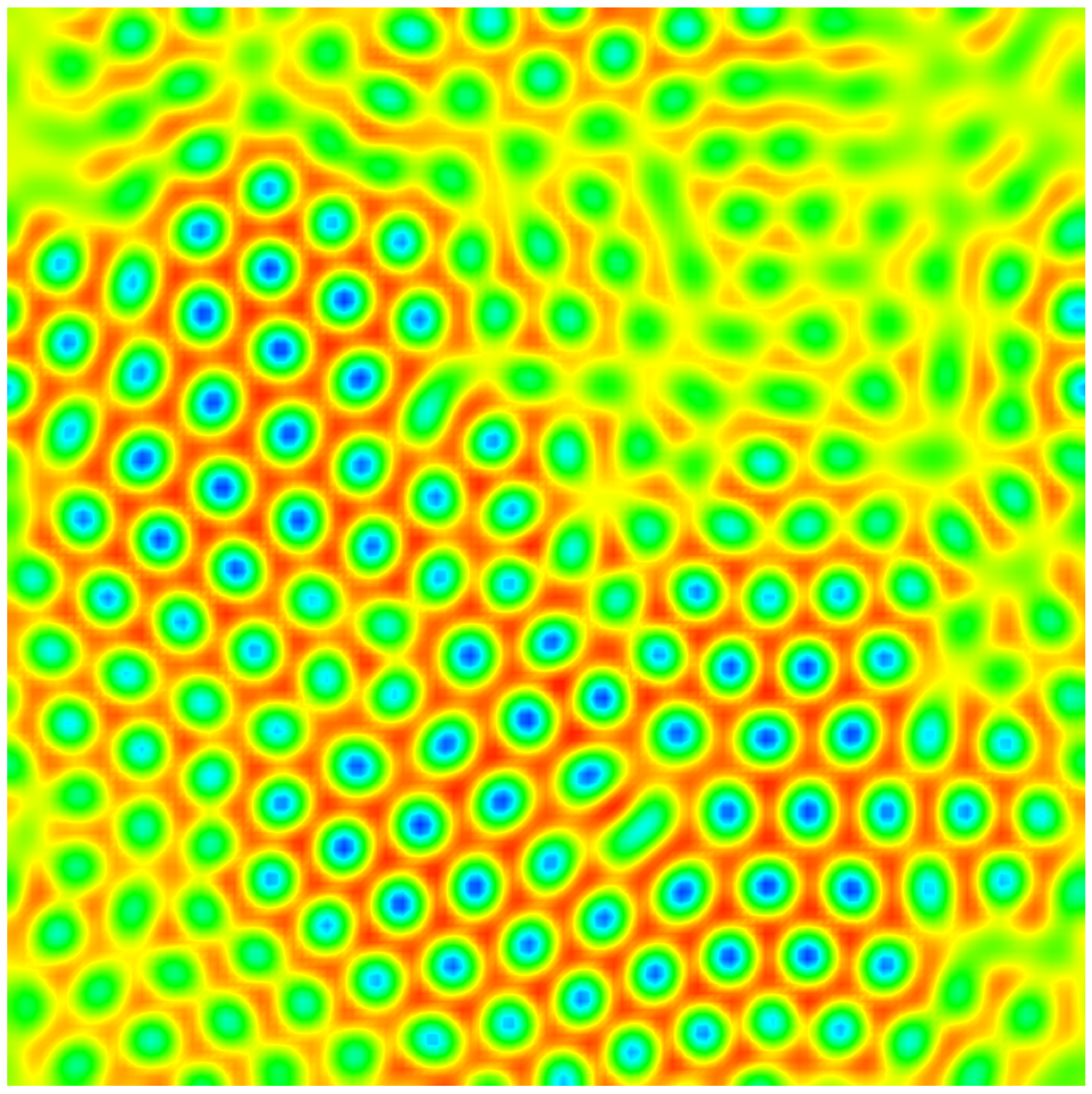}
  \includegraphics[width=0.2\textwidth]{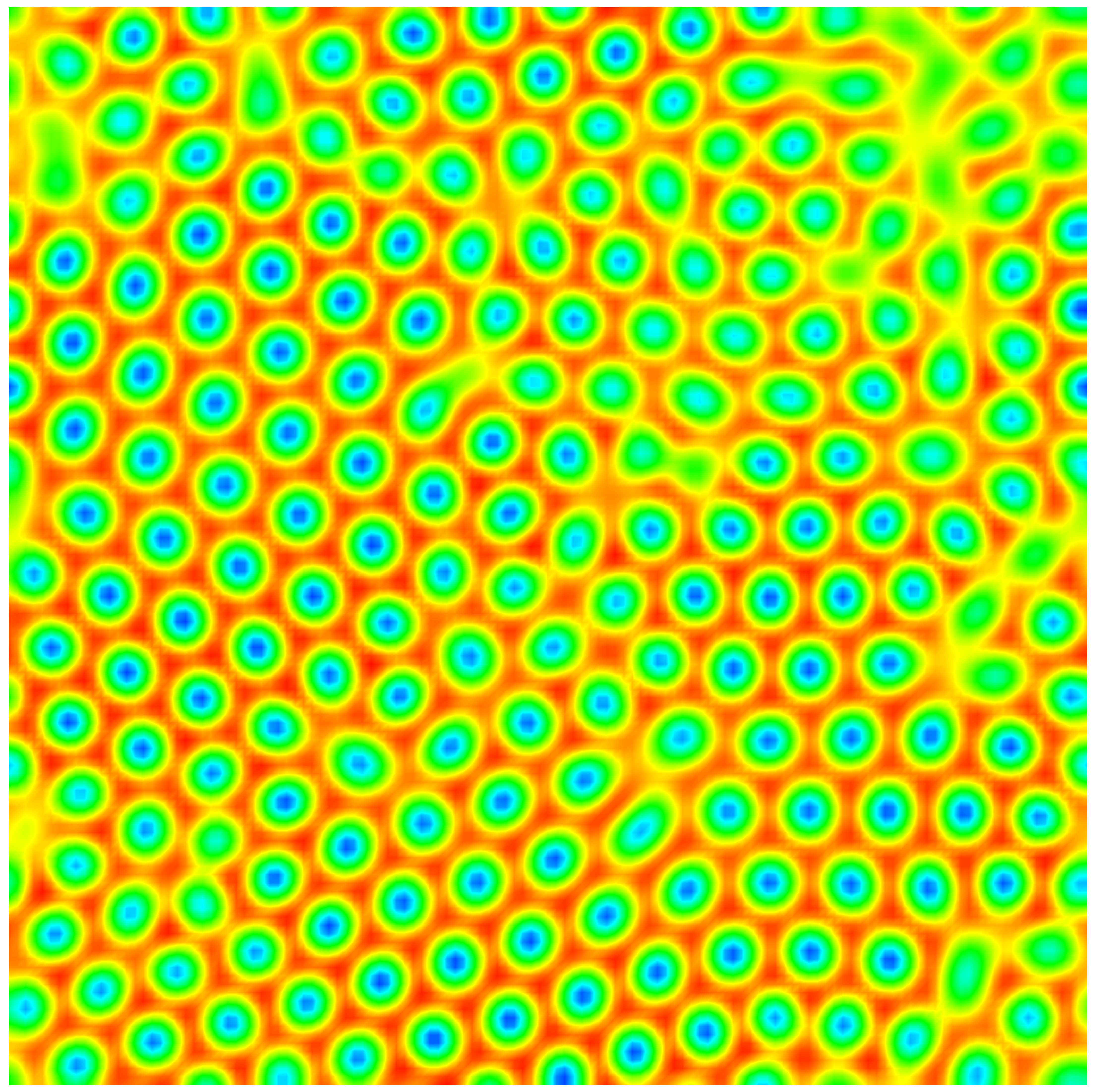}
   \includegraphics[width=0.2\textwidth]{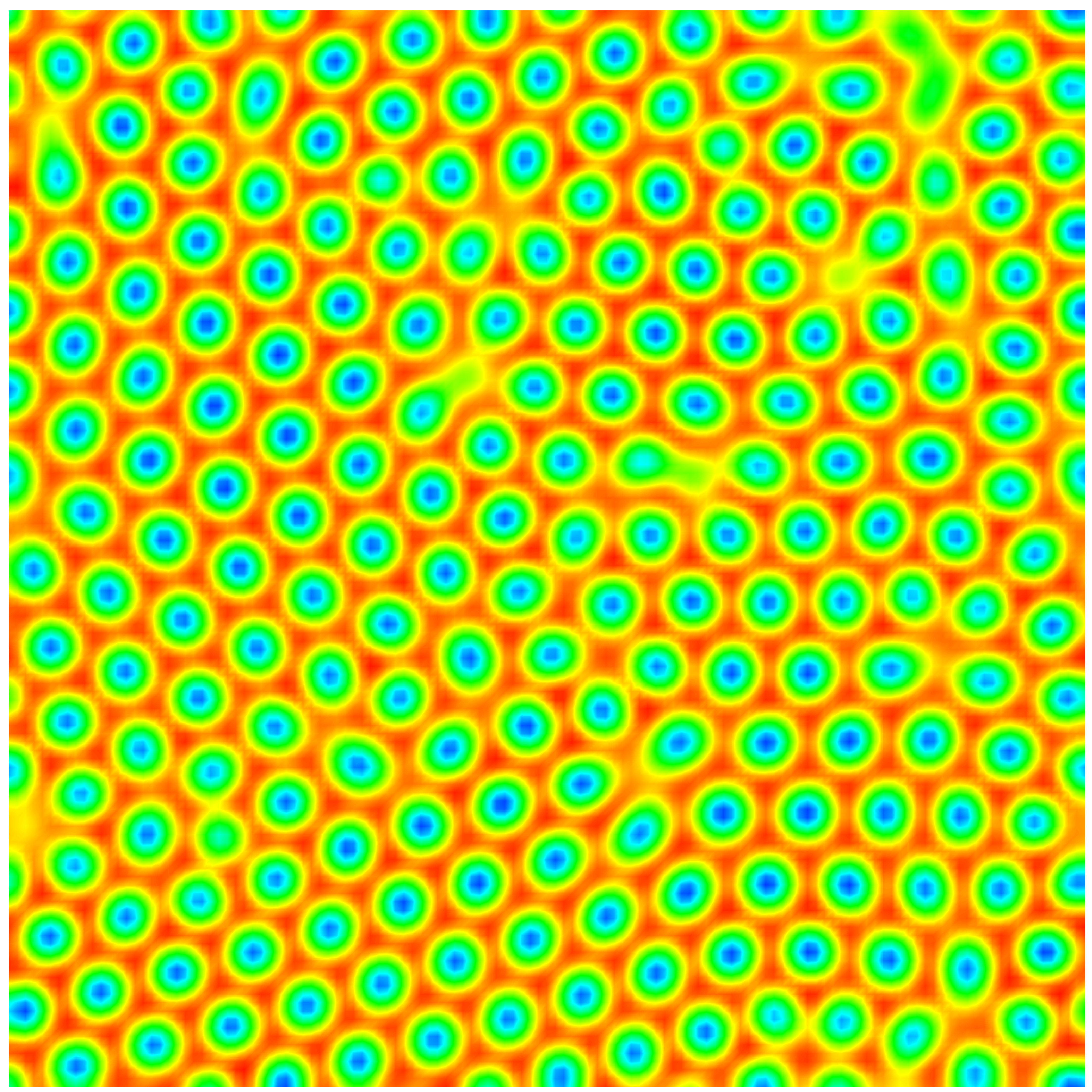}
}

\subfigure[Numerical solution from POD-ROM-II with r=8 at $t=10,30,50,90$]{
\includegraphics[width=0.2\textwidth]{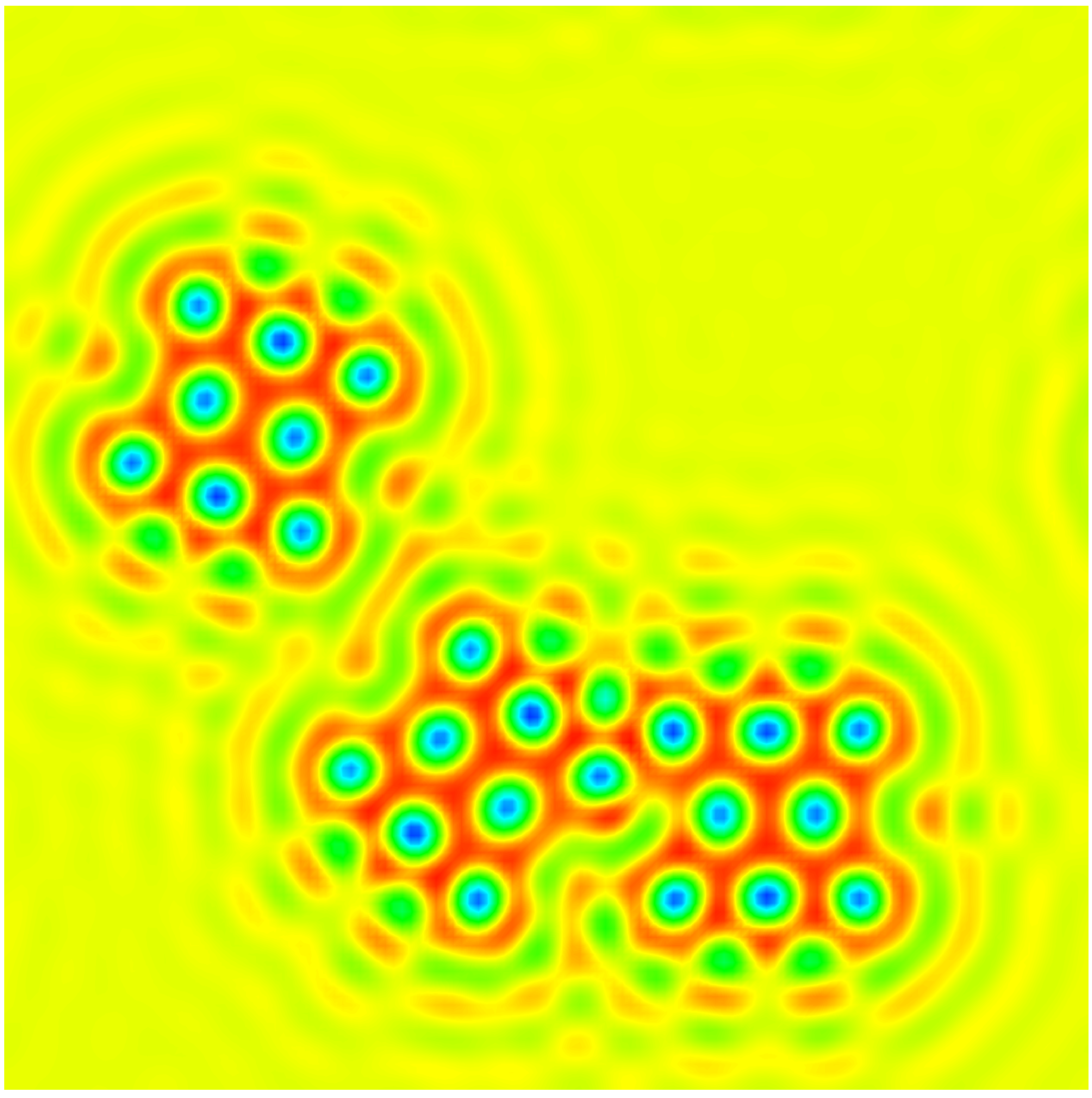}
 \includegraphics[width=0.2\textwidth]{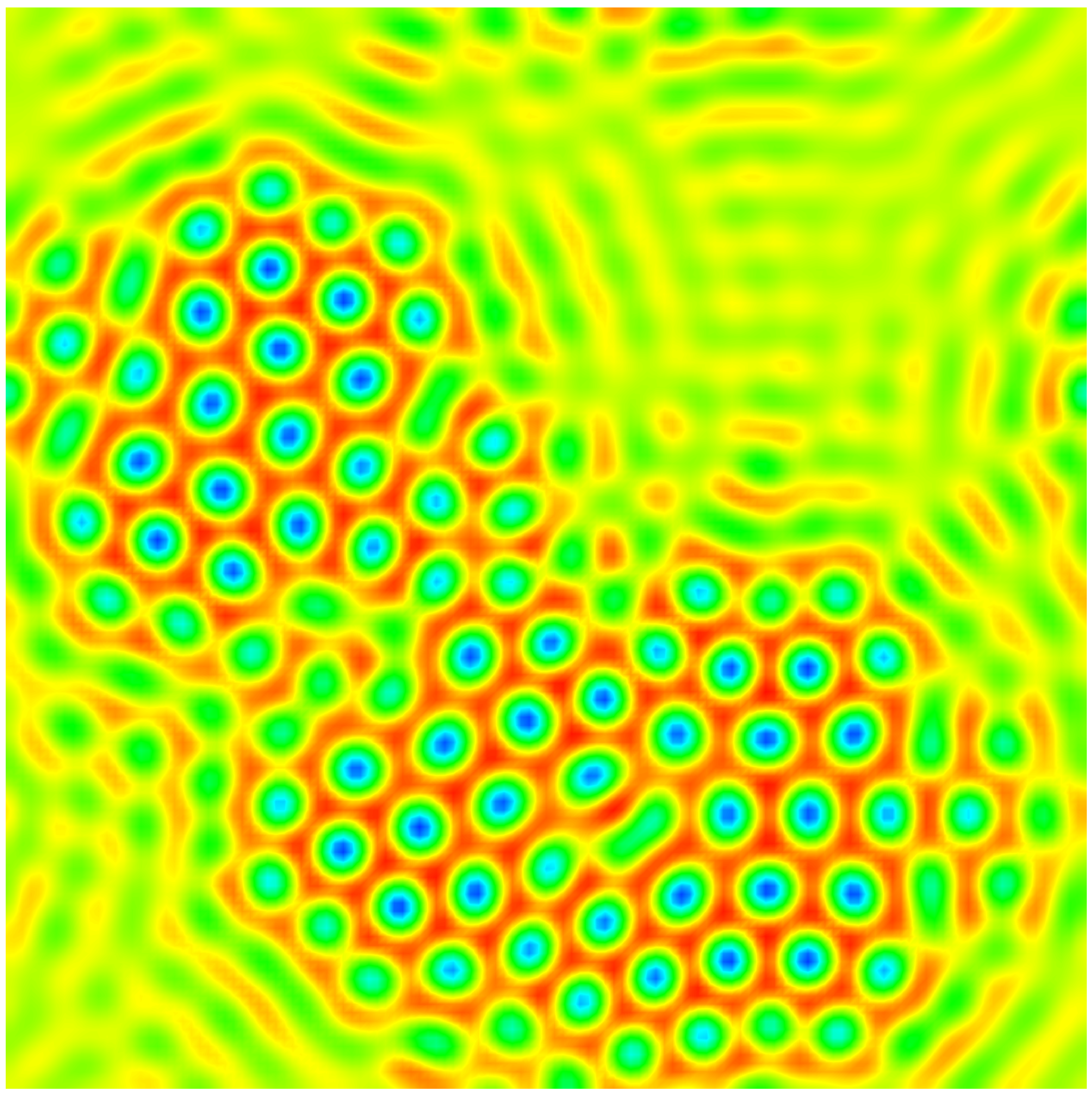}
  \includegraphics[width=0.2\textwidth]{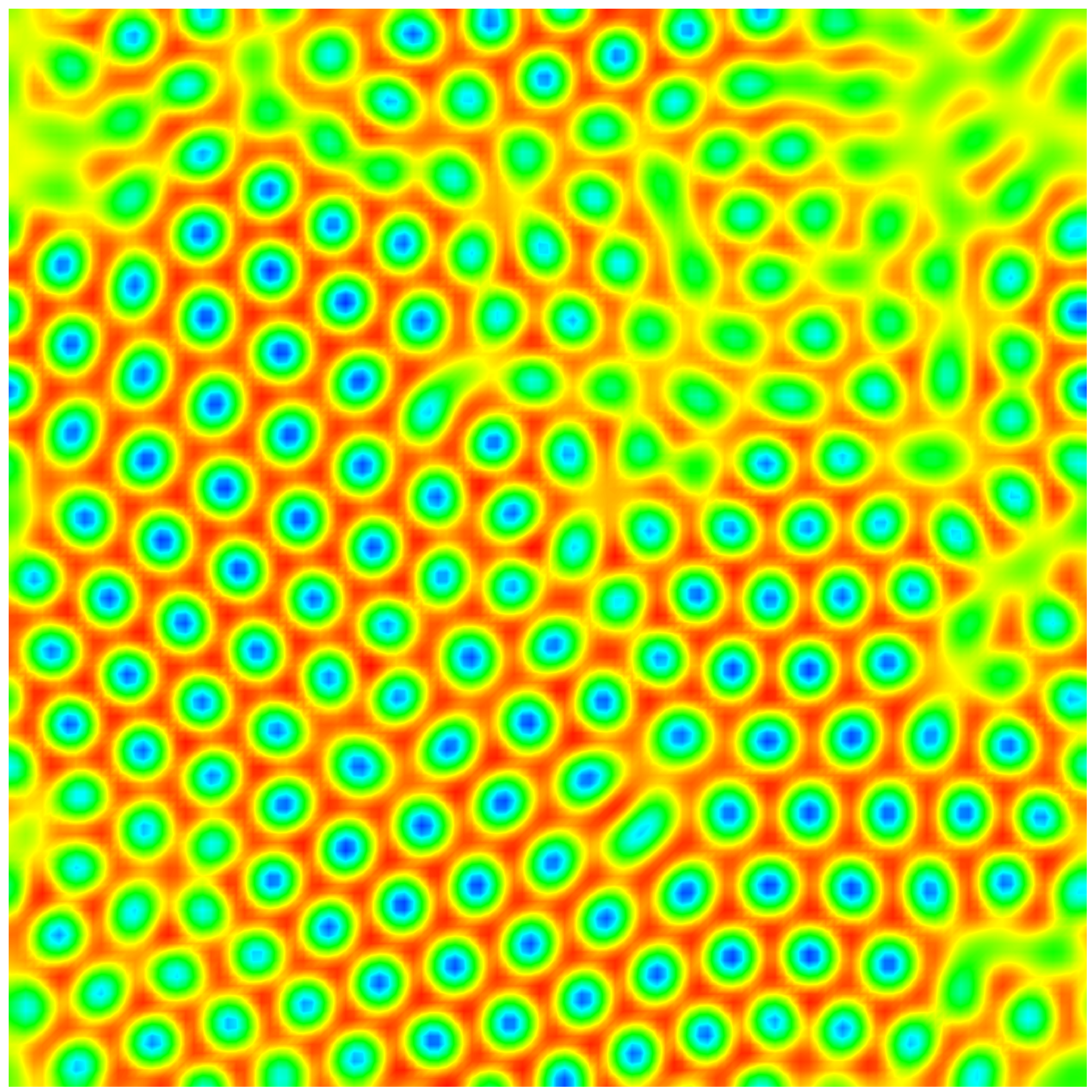}
   \includegraphics[width=0.2\textwidth]{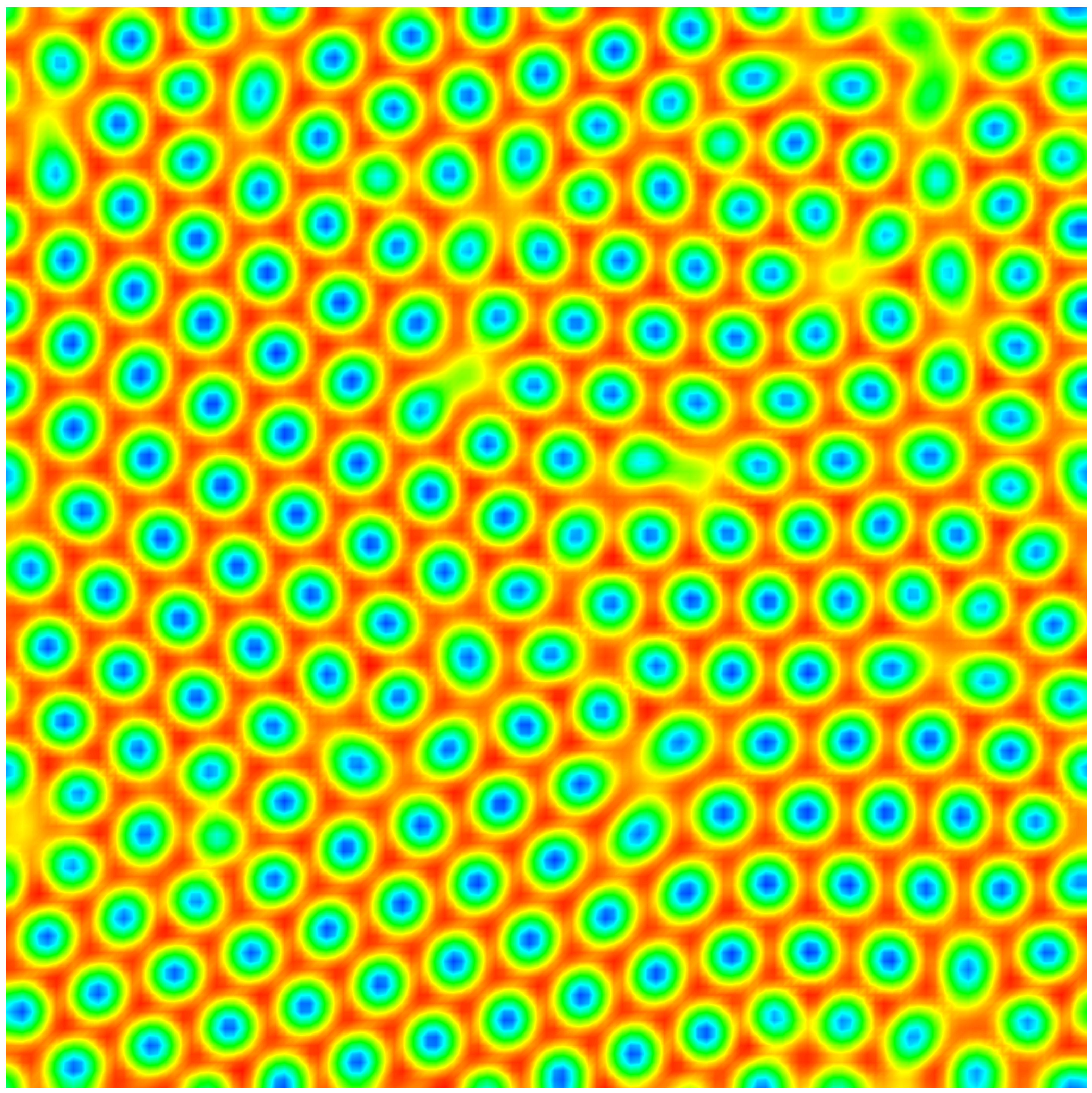}
}

\subfigure[Numerical solution from the full order model at $t=10,30,50,90$]{
\includegraphics[width=0.2\textwidth]{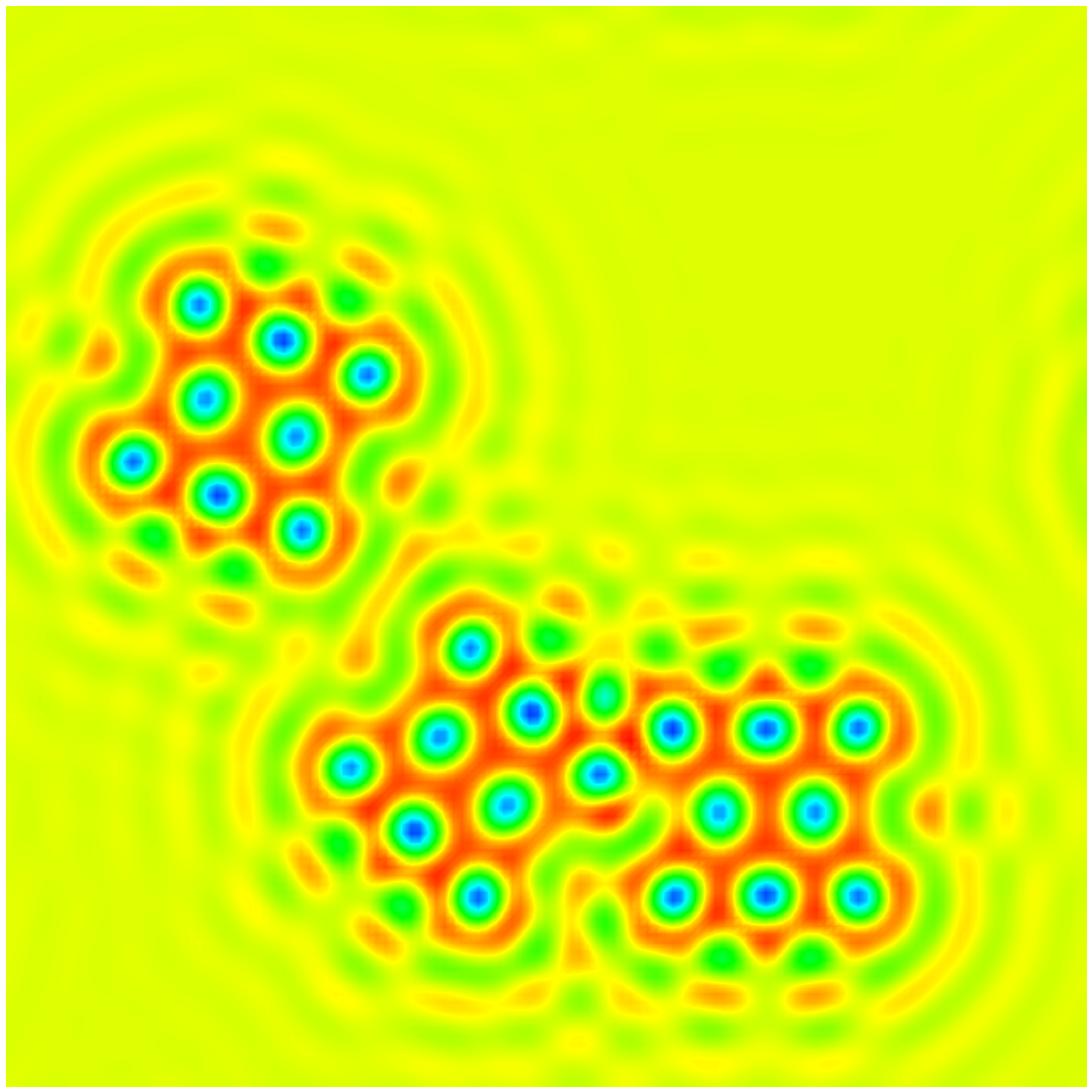}
 \includegraphics[width=0.2\textwidth]{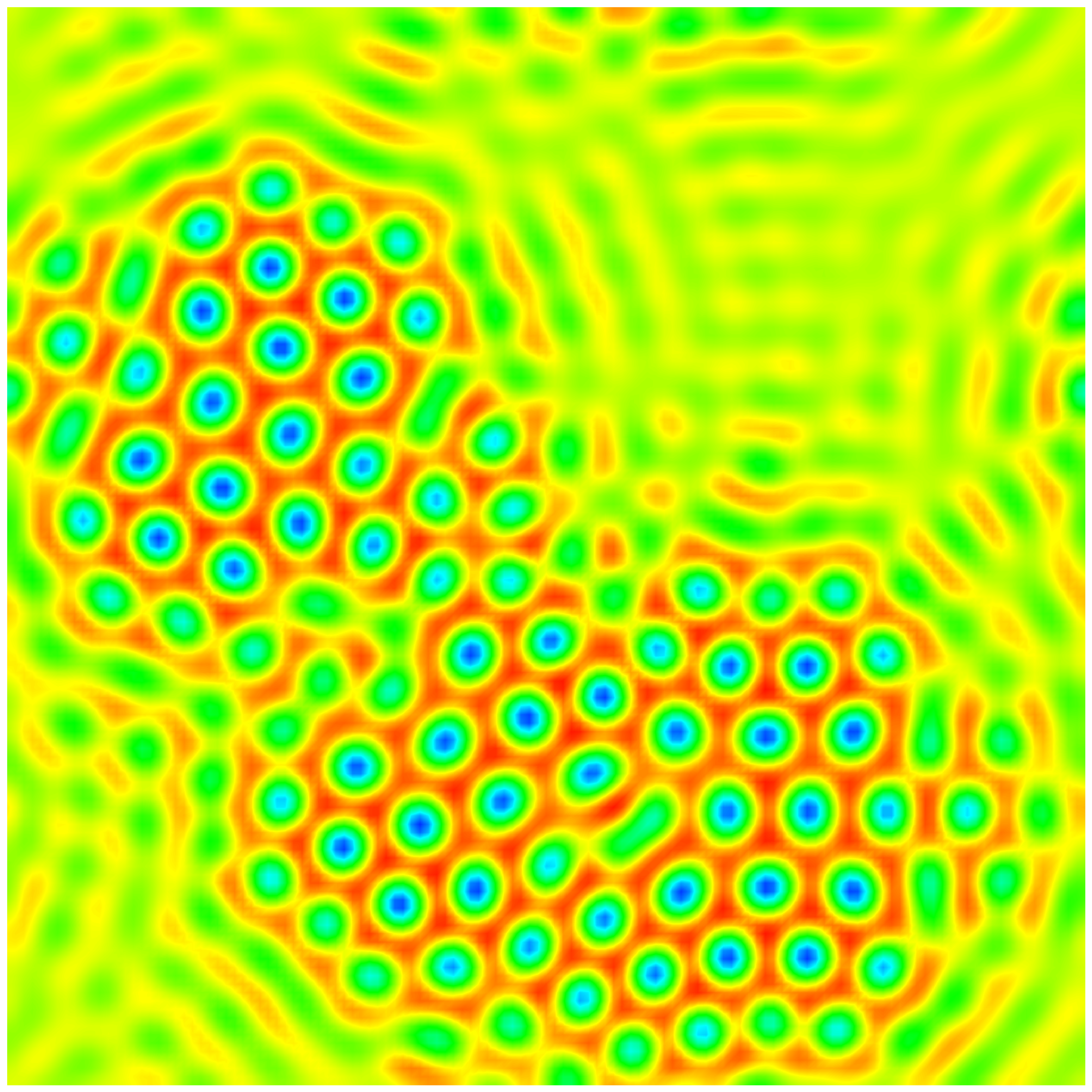}
  \includegraphics[width=0.2\textwidth]{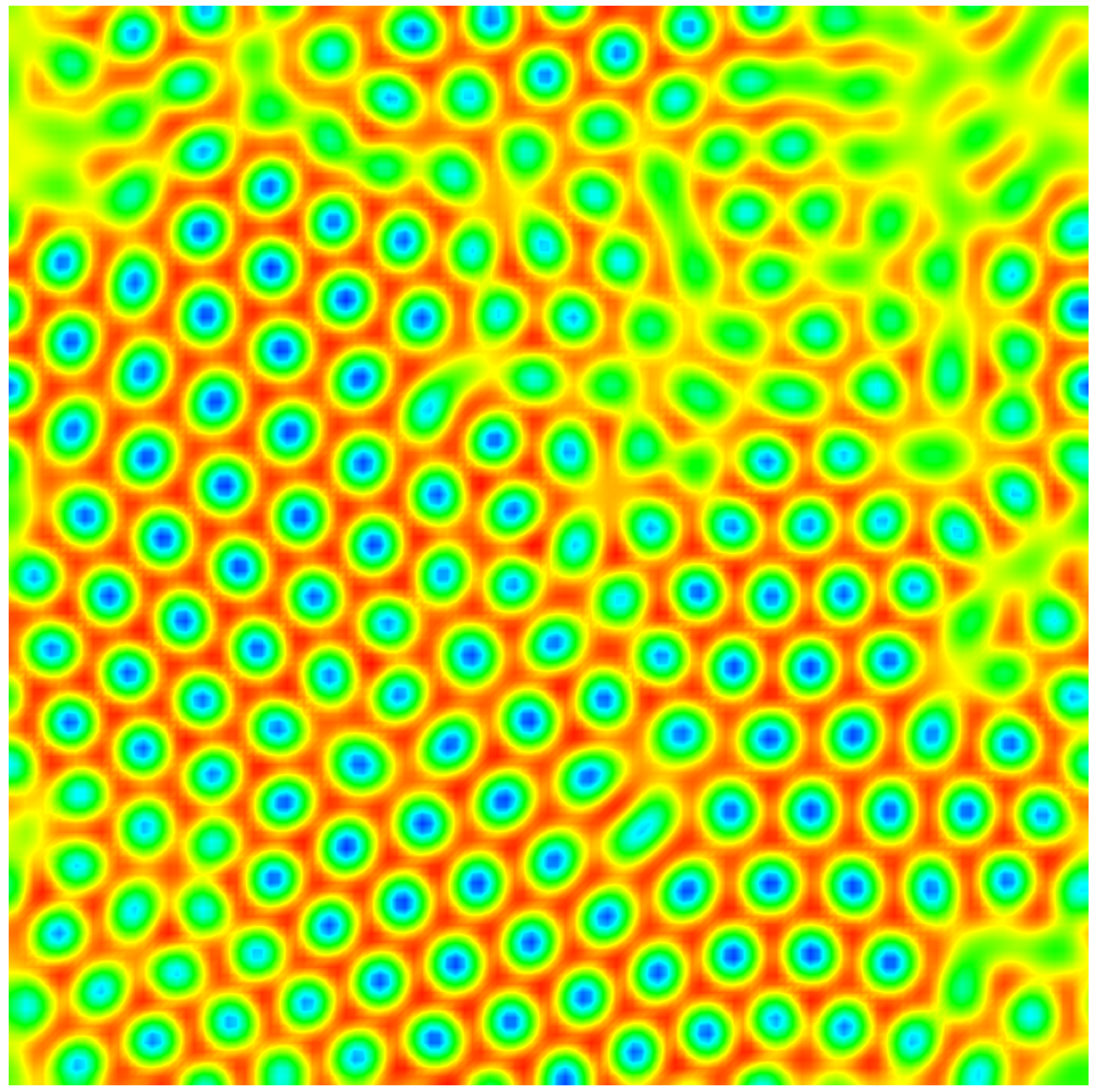}
   \includegraphics[width=0.2\textwidth]{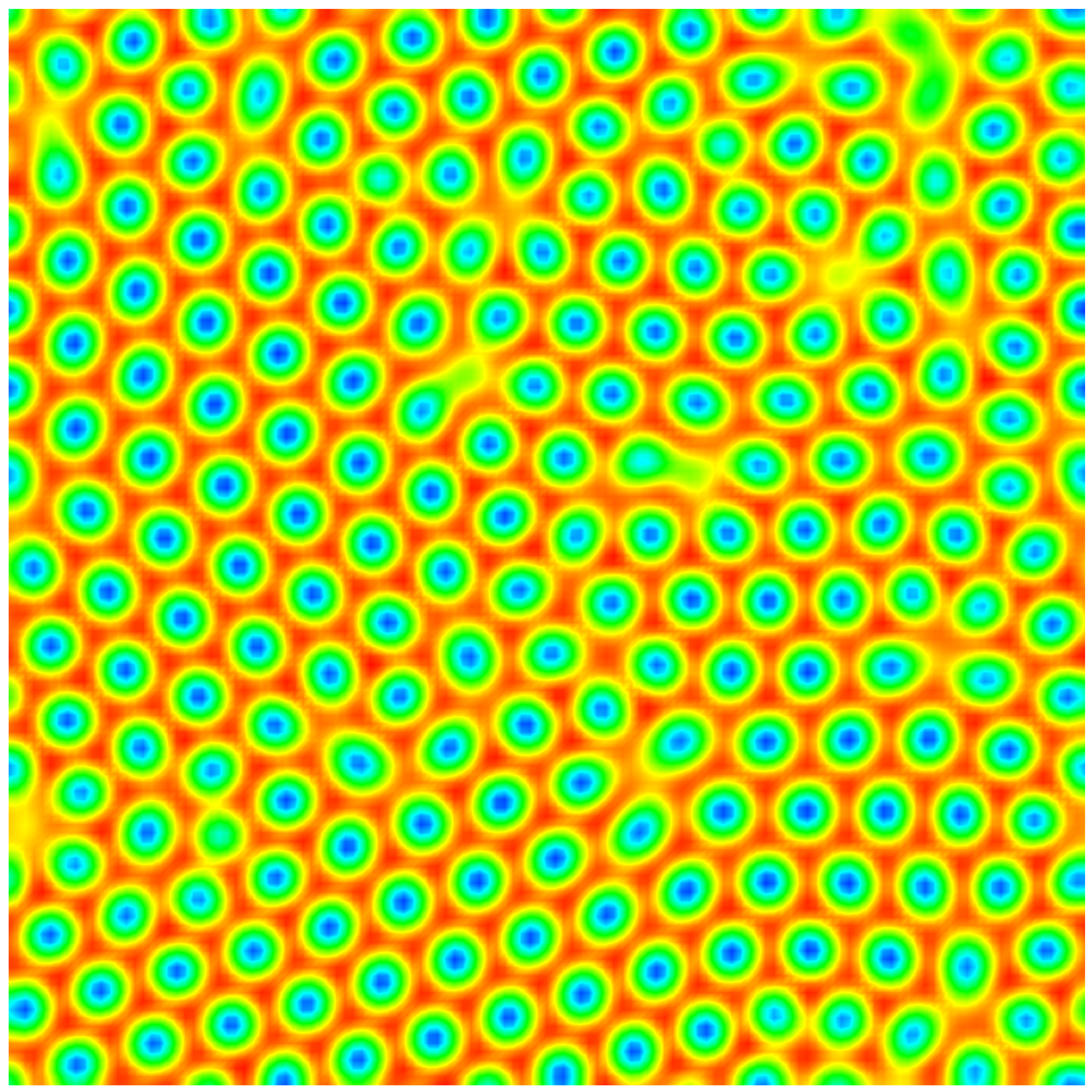}}

\caption{A comparison between the numerical solutions from the full model and the numerical solutions from POD-ROM-II for the phase field crystal equations with various numbers of modes.}
\label{fig:PFC-Example-Compare}
\end{figure}

Then, we compare the energy dissipation curves for POD-ROM-II with different modes, as shown in Figure \ref{fig:PFC-Energy-Compare}(a). It appears that the POD-ROM-II can accurately predict the energy evolution with $r=10$ modes already. Furthermore, we compare the results between POD-ROM-I and POD-ROM-II shown in Figure \ref{fig:PFC-Energy-Compare}(b), and it appears that POD-ROM-II can provide a more accurate prediction for the energy evolution than POD-ROM-I with the same modes. This is reasonable since POD-ROM-I has a modified energy dissipation rate.

\begin{figure}[H]
\center
\subfigure[]{\includegraphics[width=0.45\textwidth]{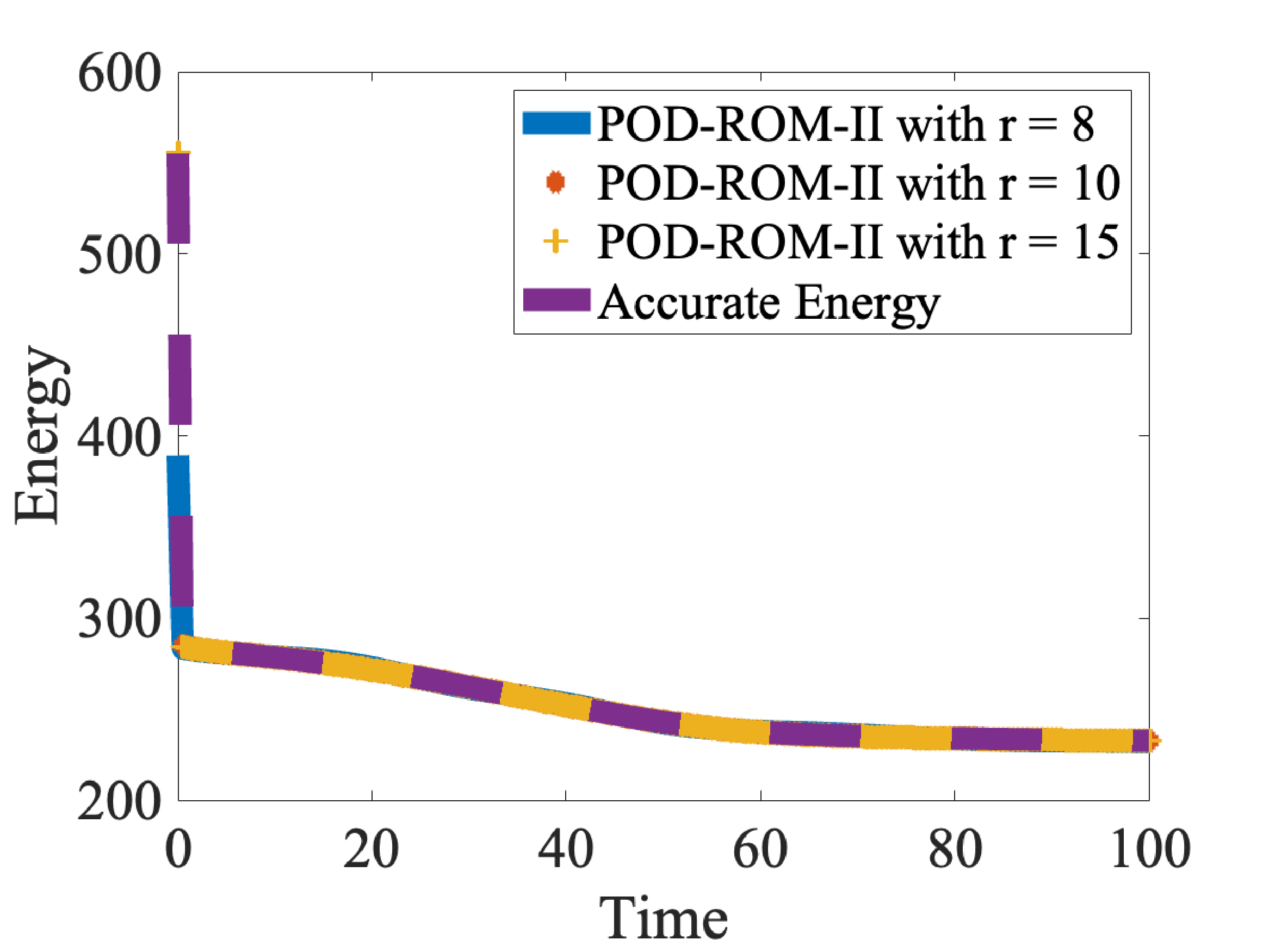}}
\subfigure[]{\includegraphics[width=0.45\textwidth]{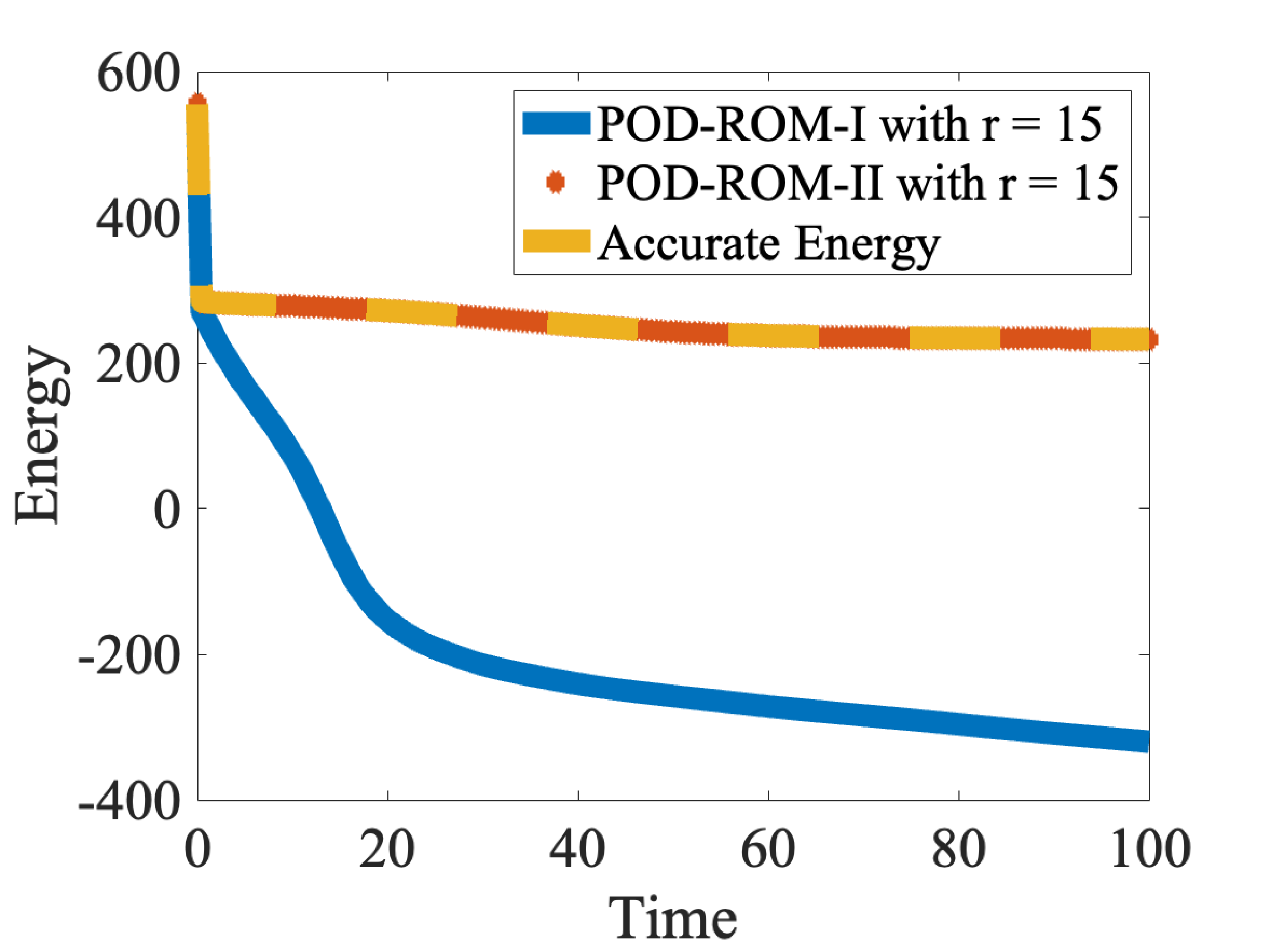}}
\caption{A comparison of the energy dissipation for the phase field crystal equation. (a) A comparison of the energy dissipation results between the full order model and the POD-ROM-II with various modes using Scheme \ref{sch:Relaxed-POD-ROM-CN}; (b) A comparison of the energy dissipation results between POD-ROM-I using Scheme \ref{sch:Relaxed-POD-ROM-CN-I} and POD-ROM-II using Scheme \ref{sch:Relaxed-POD-ROM-CN} }
\label{fig:PFC-Energy-Compare}
\end{figure}

\section{Conclusion and future work}
In this paper, we introduce a general numerical framework to develop structure-preserving reduced order models (ROMs) for thermodynamically consistent reversible-irreversible PDEs. Our framework is rather general in that it provides a unified approach to develop structure-preserving reduced order models for PDE systems that uphold the free energy dissipation laws. An extension of the current approach to investigate models with thermodynamic PDEs with temperature is possible. 
Meanwhile, there are still several open questions. For instance, the maximum principle for the Allen-Cahn equation is well-studied. How to derive reduced order model to preserve the maximum principle while preserving the energy dissipation law is still unclear. Additionally, for some models, mass conservation is essential, saying the Cahn-Hilliard equation, which is the main reason making it different from the Allen-Cahn equation. The current framework can't guarantee mass conservation and energy dissipation simultaneously. These open questions are to be addressed in our subsequent works.

\section*{Acknowledgments}
Zengyan Zhang and Jia Zhao would like to acknowledge the support from the National Science Foundation with grant NSF-DMS-2111479. They would also like to acknowledge NVIDIA Corporation for the donation of GPUs for conducting some of the numerical simulations in this paper.

%\bibliography{ROM-Reference.bib}{}
%\bibliographystyle{plain}

\end{document}